\documentclass[10pt,a4paper,reqno]{amsart}

\usepackage{a4wide}
\usepackage{amsmath,amsthm,amssymb, amsfonts, amsbsy, mathrsfs}
\usepackage[foot]{amsaddr}
\usepackage{bbm}
\usepackage{enumerate}
\usepackage{xparse}
\usepackage{filecontents} 
\usepackage{todonotes}
\usepackage{tikz}
\usetikzlibrary{calc}
\usetikzlibrary{decorations.markings}
\usepackage{hyperref} 
\usepackage{cite}
\usepackage{tcolorbox}
\usepackage{appendix}

\numberwithin{equation}{section}

\newcommand{\Pv}{\mathbb{P}}
\newcommand{\Ev}{\mathbb{E}}
\newcommand*{\be}{\begin{equation}}
\newcommand*{\ee}{\end{equation}}
\newcommand*{\ba}{\begin{aligned}}
	\newcommand*{\ea}{\end{aligned}}
\newcommand*{\barr}{\begin{array}{c}}
	\newcommand*{\earr}{\end{array}}
\def \toinp    {\buildrel {\Pv}\over{\longrightarrow}}
\def \toindis  {\buildrel {d}\over{\longrightarrow}}

\newcommand{\wit}{\widetilde}

\newcommand{\GIRG}{\mathrm{GIRG}_{W,L}(n)}

\newcommand{\Ilambda}{\mathrm{IGIRG}_{W,L}(\lambda)}

\newcommand{\Ione}{\mathrm{IGIRG}_{W,L}(1)}
\newcommand{\BGIRG}{\mathrm{BGIRG}_{W,L}(n)}
\newcommand{\SFPWL}{\mathrm{SFP}_{W,L}}

\newcommand*{\ind}{\mathbbm{1}}

\newcommand{\SB}{\mathrm{SB}}

\newcommand{\CB}{\mathcal {B}}
\newcommand{\CC}{\mathcal {C}}

\newcommand{\CE}{\mathcal {E}}

\newcommand{\CG}{\mathcal {G}}
\newcommand{\CI}{\mathcal {I}}

\newcommand{\CN}{\mathcal {N}}

\newcommand{\CS}{\mathcal {S}}

\newcommand{\CV}{\mathcal {V}}

\newcommand{\N}{\mathbb{N}}

\newcommand*{\wih}{\widehat}

\newcommand*{\la}{\lambda}

\newcommand*{\de}{\delta}

\newcommand*{\ve}{\varepsilon}

\newcommand*{\al}{\alpha}

\newcommand*{\un}[1]{\underline{#1}}
\newcommand*{\cki}{c_{k,i}}

	\newcommand{\calX}{\mathcal{X}}

	\newcommand{\norm}[1]{\left\|#1\right\|}
	\newcommand{\Z}{\mathbb{Z}}
	\newcommand{\lhigh}{\overline{\ell}}
	\newcommand{\llow}{\underline{\ell}}

	\newcommand{\saw}{T_\mathsf{SAW}^{\le t_0}}
	\newcommand{\Exp}{\textrm{Exp}}

\newcommand{\E}{\mathbb{E}}
\newcommand{\e}{\mathrm{e}}
\renewcommand{\Pr}{\mathbb{P}}

\newcommand{\ceil}[1]{\lceil #1 \rceil}
\newcommand{\floor}[1]{\lfloor #1 \rfloor}

\newcommand{\R}{\mathbb{R}}

\newcommand{\calA}{\mathcal{A}}
\newcommand{\calB}{\mathcal{B}}
\newcommand{\calC}{\mathcal{C}}
\newcommand{\calE}{\mathcal{E}}
\newcommand{\calG}{\mathcal{G}}
\newcommand{\calI}{\mathcal{I}}

\newcommand{\thin}{_\mathrm{thin}}
\newcommand{\Ext}{_\mathrm{ext}}
\newcommand{\eps}{\varepsilon}
\renewcommand{\epsilon}{\eps}

\newcommand{\cbb}{\color{black}}

\newtheorem{theorem}{Theorem}[section]
\newtheorem*{theorem*}{Theorem}
\newtheorem{lemma}[theorem]{Lemma}
\newtheorem{corollary}[theorem]{Corollary}
\newtheorem{definition}[theorem]{Definition}

\newtheorem{assumption}[theorem]{Assumption}

\newtheorem{claim}[theorem]{Claim}

\newtheorem*{rep@theorem}{\rep@title}
\newcommand{\newreptheorem}[2]{%
	\newenvironment{rep#1}[1]{%
		\def\rep@title{#2 \ref{##1} (restated)}%
		\begin{rep@theorem}}%
		{\end{rep@theorem}}}
\newreptheorem{theorem}{Theorem}
\newreptheorem{lemma}{Lemma}

\title[Weighted distances with penalty]{Stopping explosion by penalising transmission to hubs in scale-free spatial random graphs}
\author{J\'ulia Komj\'athy$^*$}
\address{\footnotemark{*} Eindhoven University of Technology, Eindhoven, The Netherlands.}
\author{John Lapinskas$^\dag$}
\address{\footnotemark{\dag} University of Bristol, Bristol, UK.}
\author{Johannes Lengler$^\ddag$}
\address{\footnotemark{\ddag} ETH Z\"urich, Z\"urich, Switzerland.}
\date{\today}
\thanks{The research leading to these results has received funding from the European Research Council under the European Union's Seventh Framework Programme (FP7/2007-2013) ERC grant agreement no.\ 334828. The paper reflects only the authors' views and not the views of the ERC or the European Commission.  The European Union is not liable for any use that may be made of the information contained therein.}

\begin{document}
	\maketitle

\begin{abstract}
We study the spread of information in finite and infinite inhomogeneous spatial random graphs. We assume that each edge has a transmission cost that is a product of an i.i.d.\! random variable $L$ and a penalty factor: edges between vertices of expected degrees $w_1$ and $w_2$ are penalised by a factor of $(w_1w_2)^{\mu}$ for all $\mu >0$. We study this process for scale-free percolation, for (finite and infinite) Geometric Inhomogeneous Random Graphs, and for Hyperbolic Random Graphs, all with power law degree distributions with exponent $\tau >1$. For $\tau <3$, we find a threshold behaviour, depending on how fast the cumulative distribution function of $L$ decays at zero. If it decays at most polynomially with exponent smaller than $(3-\tau)/(2\mu)$ then explosion happens, i.e., with positive probability we can reach infinitely many vertices with finite cost (for the infinite models), or reach a linear fraction of all vertices with bounded costs (for the finite models). On the other hand, if the cdf of $L$ decays at zero at least polynomially with exponent larger than $(3-\tau)/(2\mu)$, then no explosion happens. This behaviour is arguably a better representation of information spreading processes in social networks than the case without penalising factor, in which explosion always happens unless the cdf of $L$ is doubly exponentially flat around zero. 
Finally, we extend the results to other penalty functions, including arbitrary polynomials in $w_1$ and $w_2$. In some cases the interesting phenomenon occurs that the model changes behaviour (from explosive to conservative and vice versa) when we reverse the role of $w_1$ and $w_2$. Intuitively, this could corresponds to reversing the flow of information: gathering information might take much longer than sending it out.
\end{abstract}

	\section{Introduction}
	Many real-world social and technological networks share a surprising number of fundamental properties, including a heavy-tailed degree distribution, strong clustering, and community structures~\cite{albert2002statistical,boccaletti2006complex,newman2011structure,newman2018networks}. These features are known to have opposing effects on the spread of information or infections in such networks. On the one hand, nodes of large degree (also called hubs, super-spreaders, or influencers) contribute to fast dissemination, and foster explosive propagation of information or infections~\cite{gruhl2004information,pastor2015epidemic,pastor2007evolution,dorogovtsev2008critical}. On the other hand, clustering and community structures provide natural barriers that slow down the process~\cite{karsai2011small,merler2015spatiotemporal,bajardi2011dynamical,janssen2017rumors,isham2011spread}. 
	
	The interplay of these effects is complex, but until recently there were no appropriate random graph models in which to study it; many models exhibited heavy-tailed degree distributions, strong clustering, or community structures individually, but none combined the three. Recently, this problem has been solved by a family of inhomogeneous spatial random graph models which do combine these features, namely Scale-free Percolation (SFP)~\cite{DeiHofHoo13} and continuum scale-free percolation \cite{DepWut18}, (finite and infinite) Geometric Inhomogeneous Random Graphs (GIRGs)~\cite{BriKeuLen19,bringmann2016average}, and Hyperbolic Random Graphs (HRGs)~\cite{BogPapaKriou10, papadopoulos2010greedy, GugPanPet12} (see also \cite{BogPasDia04, SerKriBog08} for earlier versions of the model). 
	These models are closely related, and in fact the results will apply to all of them.
		
	Previous work studying infection processes in these models~\cite{HofKom18,julia-girg} focused on the first passage percolation (FPP) infection process, including the variant in which transmission costs follow an arbitrary probability distribution with support starting at $0$ (see below for a more detailed discussion). Essentially, for any reasonable choice of parameters, either a constant proportion of vertices (for finite models) or infinitely many vertices (for infinite models) may be infected in constant time. This does not match reality. While some processes can indeed spread this fast, there are others which do not, such as the spread of diseases through physical social networks or the spread of behaviours~\cite{centola2010spread}. 
	
	In this paper, we follow the approach of~\cite{feldman2017high,karsai2006nonequilibrium,giuraniuc2006criticality}, and assume that high-weight vertices have higher expected transmission times. This reflects the fact that even large-degree nodes have a limited time budget and cannot scale up their number of contacts per time unit arbitrarily, as has been observed in real-world communication~\cite{miritello2013time} and disease spreading~\cite{nordvik2006number}. By doing so, we will recover the rich variation in behaviour we might expect. We prove a precise phase transition between the case in which infinitely many vertices may be infected in constant time and the case in which they may not, depending on the parameters of the random graph, the distribution of possible transmission times, and the transmission penalty for high-weight vertices.
	
	We will define the GIRG, SFP and HRG models in detail in Sections~\ref{sec:infinite-models} and~\ref{sec:finite-result}. In a nutshell, each vertex $v$ in a graph $G=(V,E)$ following these models has a (possibly random) location in a geometric space and a weight $W_v$ which models its popularity. Then each pair of vertices is connected with a probability that depends on their geometric distance and on their weights. The probability of connecting decreases polynomially with their distance, and increases polynomially with their respective weights; $W_v$ is equal to the expected degree of $v$ up to constant factors.
	
	\subsection{A simple example: Infinite geometric inhomogeneous random graphs (IGIRGs) with symmetric monomial penalties} 
In the IGIRG model, the vertex set $\CV$ is given by a homogeneous Poisson Point Process with intensity $1$ on $\R^d$.
	The weights $(W_v)_{v\in \CV}$ of the vertices are i.i.d.~copies of a random variable $W$ with polynomially decaying tail\footnote{The full model is more flexible, see Definition~\ref{def:GIRG}. E.g., a slowly-varying correction factor is allowed in~\eqref{eq:pl1}.}, 
\be\label{eq:pl1}
\Pv(W \ge x) = 1/x^{\tau-1}.
\ee
Such distributions, for $\tau\in (2,3)$, are called \emph{power laws}\footnote{We will also include the cases $\tau \in(1,2]$ and $\tau \geq 3$ into our analysis, but this may lead to infinite vertex degrees or, in some models, to graphs without giant components.} and $\tau$ is called the \emph{power-law exponent}. Between every two vertices $u$ and $v$ with weights $W_u$ and $W_v$ (respectively), we independently add an edge with probability $1\wedge c (W_uW_v/\|u-v\|^d)^\alpha$ for some $c >0$. Here $\alpha \in (1,\infty)$ is the \textit{long-range parameter} of the model, and governs the prevalence of edges between geographically-distant high-weight vertices; we require $\alpha > 1$ to avoid infinite vertex degrees. The expected degree of a vertex $v$, conditioned on its weight $W_v$, then coincides with $W_v$ up to a constant factor. The constant $c$ governs the edge-density.

Two example graphs of this model\footnote{Or rather, a finite, rescaled version of the model on the unit cube called GIRG, in which the number of vertices is fixed to be $n$, so the density is $n$ instead of $1$, see Definitions~\ref{def:GIRG} and~\ref{def:bgirg}.} are shown in Figure~\ref{fig:GIRG}. 

\begin{figure}[t]\begin{center}
		\includegraphics[width=0.45\columnwidth]{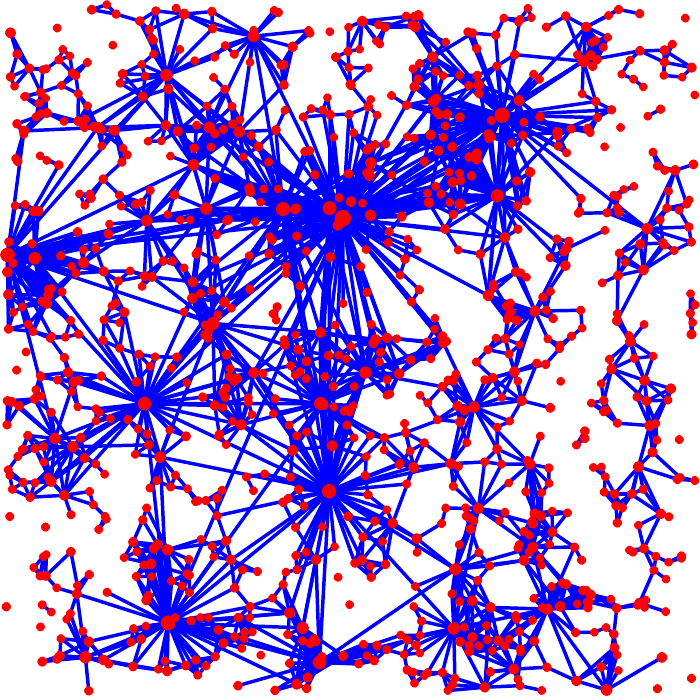}
		\includegraphics[width=0.45\columnwidth]{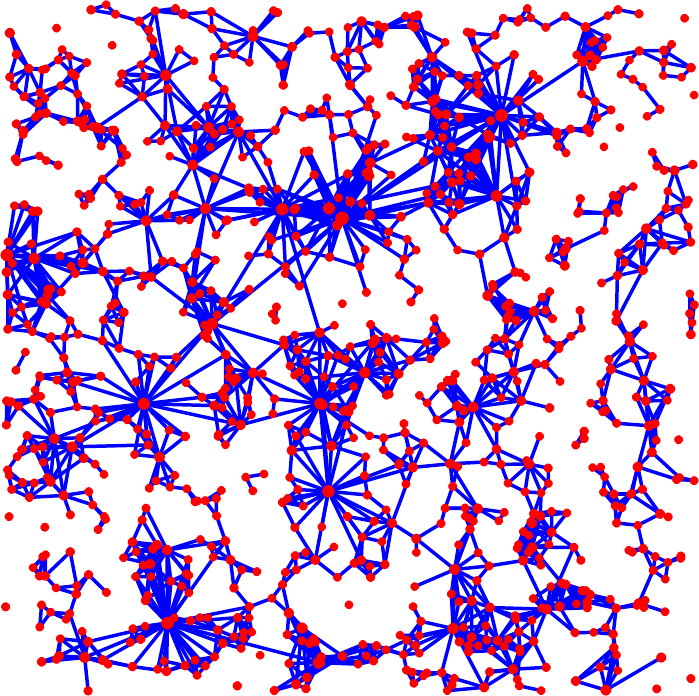}
		\caption{{\small Two examples of Geometric Inhomogeneous Random Graphs (GIRGs). The $n = 1000$ vertices are placed randomly into a unit cube of dimension $d=2$ and each draws a random weight from a power law distribution with exponent $\tau = 2.9$ (on the left) and $\tau=3.5$ (on the right). 
		We used the same vertex set and the same underlying uniform variables to simulate vertex weights in both cases: for a uniform variable $U_v$, we set the weight of vertex $v$ to $W^{(2.9)}_v:=U_v^{-1/1.9}$, for the left picture, while  $W^{(3.5)}_v:=U_v^{-1/2.5}$ for the right picture.  
		Each pair of vertices with positions $x_1,x_2$ and weights $w_1,w_2$, respectively, is connected with probability $p^{(\tau)} = \min\left(1, 0.1 (w_1^{(\tau)} w_2^{(\tau)} |x_1-x_2|^{-d}/n)^\alpha\right)$, where $\alpha = 4$. Connections are again generated in a coupled way, using the same set of uniform variables for the two pictures, thresholded at $p^{(2.9)}$ and $p^{(3.5)}$, respectively.
		The pictures were generated by the open-source software \cite{JorrGIRG2}.}}

		\label{fig:GIRG}
\end{center}\end{figure}

In defining the information- or infection-spreading process, we assume that each edge $e=(u,v)$ has a \emph{transmission cost} (or transmission time) $C_e$, comprised of an i.i.d.\ random component and a deterministic weight penalty. The \emph{random component} is a non-negative random variable $L_e$ associated with the edge $e$; these variables are i.i.d.\ copies
of a non-negative random variable $L$. The \emph{deterministic weight penalty} is a function of the weights $W_u$ and $W_v$ of the endpoints $u$ and $v$ of the edge, and we form the cost by multiplying this by $L_e$. We define the cost along any path as the sum of the costs of all edges on the path, which gives a quasimetric on the graph\footnote{A quasimetric is a distance function that satisfies all axioms of a metric except symmetry. Symmetric penalty functions result in metrics, while non-symmetric penalty functions result in quasimetrics.}. 
For this example, we fix a parameter $\mu > 0$ that we call the \emph{penalty strength}, and set the cost of an edge $e=\{u,v\}$ to be 
\be\label{eq:cost} C_e:=L_e (W_u W_v)^\mu,\ee
resulting in a metric on IGIRG. Later we will generalize the results to arbitrary polynomial weight penalties. 

We consider single-sourced spreading of information and investigate how long it takes for the information to spread to another vertex, and how many vertices are reachable from the source within a given cost $T$.
Since the underlying graph is infinite, it may happen that for some \emph{finite cost} $T<\infty$, the number of vertices reachable from the source with cost at most $T$ is infinite. This phenomenon is called \emph{explosion}, and the infimum of such costs $T$ is called the \emph{explosion time} of the source vertex.

In this paper, we show that for all power law exponents $\tau \in (2,3)$, all long-range parameters $\alpha \in (1,\infty)$, and all penalty strengths $\mu >0$, explosion occurs with positive probability if and only if the cumulative distribution function $F_L$ of $L$ is sufficiently steep at the origin. More formally, we prove a phase transition. Suppose $F_L$ grows polynomially at the origin, so that there exists a small interval $[0,t_0]$ and constants $c_1$, $c_2$ and $\beta > 0$ such that 
\begin{equation}\label{eq:steep}
c_1 t^\beta \le F_L(t) \le c_2 t^\beta \mbox{ for all }t \in [0,t_0].
\end{equation}
Then the main result of this paper implies the following phase transition.
\begin{tcolorbox}
\begin{theorem}\label{thm:toy}[Main Case of Theorem~\ref{thm:mu-classification} for IGIRG graphs]
\end{theorem}
Let $\tau \in (2,3)$, let $\alpha \in (1,\infty)$, and let $L$ be a non-negative random variable satisfying~\eqref{eq:steep}. In IGIRG graphs with degree power-law exponent $\tau$, long-range parameter $\alpha$, and edge weights given by~\eqref{eq:cost}, explosion occurs with positive probability if $\beta < \beta_c(\mu) := (3-\tau)/2\mu$, and almost surely does not occur if $\beta> \beta_c(\mu)$. 
\end{tcolorbox}
Note that explosion only depends on the behaviour of $F_L$ close to $0$, or equivalently the tail behaviour of the random variable $1/L$. Indeed, for all $z\ge 0$, $\Pv(1/L \ge z)=F_L(1/z)$, so~\eqref{eq:steep} is equivalent to the condition that $c_1 z^{-\beta} \le \Pv(1/L \ge z) \le c_2 z^{-\beta}$ for all $z \ge 1/t_0$. Intuitively, this theorem says that explosion is determined entirely by three factors. The first is the tail behaviour of $1/L$ as captured by $\beta$; the second is the strength of the weight penalty $\mu$; and the third is the exponent $\tau$ of the power law of the degree distribution.\medskip

\subsection{Extensions} We defer a formal statement of the main results to Section~\ref{sec:results}, where we generalise the example above. Firstly, the results also apply to related random graph models, including the infinite SFP model and the finite HRG and GIRG models. These finite graphs are typically not connected, but  with high probability they have a single giant (linear-size) component as long as  their degree distributions follows a power law with exponent $\tau \in (2,3)$. In finite graphs, the concept of explosion translates to the cost-distance between two uniformly-chosen vertices in the giant not tending to infinity as the size of the network grows (see Theorem~\ref{thm:GIRG1} and Corollary~\ref{cor:HRG}). We show that explosion follows the same phase transition as in Theorem~\ref{thm:toy}. Moreover, under some additional conditions, that apply, among others, to hyperbolic random graphs, we show that this cost-distance converges in distribution to the sum of two independent copies of the explosion time of a related infinite model (Theorem~\ref{thm:GIRG-fine} and Corollary~\ref{cor:HRG}).

We also consider \emph{other penalty functions}, such as the maximum or the sum of the weights of the incident vertices, or arbitrary (finite) polynomials of those two weights. For all $\mu>0$, the cases $\wit C_e:=L_e \max\{W_u, W_v\}^\mu$, and $\hat C_e:=L_e (W_u+W_v)^\mu$ behave similarly to the product penalty $C_e = (W_uW_v)^\mu$ of~\eqref{eq:cost}; the main difference is that the critical penalty-strength $\beta_c(\mu)$ changes from $(3-\tau)/2\mu$ to $(3-\tau)/\mu$. In general, for an arbitrary polynomial penalty function $f$ with degree $\deg(f)$, the threshold in $\beta$ occurs at $\beta_c(f)=(3-\tau)/\deg(f)$ whenever $\tau\in(2,3)$.

The results also cover the \emph{whole parameter space in $\tau$ and $\alpha$}, not just the ranges $\tau \in (2,3)$ and $\alpha \in (1,\infty)$, but also  the ``$\alpha=\infty$'' (threshold) case when an edge is present when the Euclidean distance between the vertices is less then a threshold value depending on the vertex-weights. For $\tau >3$, it follows from known results that explosion almost surely does not occur in SFP~\cite{HofKom18}, and we expect other models to exhibit the same behaviour.  More interestingly, for the infinite models the results also allow for $\tau \in (1,2]$, when every vertex has infinitely many neighbors almost surely \cite{DeiHofHoo13}. Thus without a weight penalty, explosion happens trivially. Nevertheless, a strong enough penalty factor can prevent explosion even in this case. In fact, if $f$ is a \emph{symmetric} polynomial penalty function, then the threshold in $\beta$ coincides with the $\tau \in (2,3)$ case: $\beta_c(f)=(3-\tau)/\deg(f)$. 

For \emph{asymmetric polynomial penalty} functions, the picture is a little more complicated and we no longer prove a full phase transition in all cases (see Theorem~\ref{thm:poly-classification}). However, for \emph{monomial}  penalty functions, such as $C_e = L_e W_u^\mu W_v^\nu$, we do prove a full phase transition and the critical value of $\beta$ increases to $\beta_c(f)=\max\{(3-\tau)/(\mu + \nu), (2-\tau)/\nu\}$; thus explosion becomes easier. By contrast, when $\alpha \in (0,1]$, explosion occurs almost surely.

Note that considering asymmetric penalty functions such as $C_e = L_e W_u^\mu W_v^\nu$ raises an interesting issue. Explosion, in its original definition, means \emph{outwards} explosion. That is, infinitely many vertices are reachable from a fixed vertex $v$ within some finite cost $T$, with positive probability. However, one can  also consider   \emph{inwards} explosion, in which a fixed vertex $v$  is reachable from infinitely many vertices within cost $T$, with positive probability. The threshold for inwards explosion is the same as the threshold for the reversed penalty function $C_e = L_e W_u^\nu W_v^\mu$, which is $\beta_c^{\text{inwards}}=\beta_c(f_{\nu,\mu})=\max\{(3-\tau)/(\mu+\nu), (2-\tau)/\mu\}$. Interestingly, this implies that there are penalty functions which exhibit inwards explosion but not outwards explosion and vice versa. This could be interpreted as an asymmetry between the two possible directions of information flow: it is much quicker to send out information than to gather it, or vice versa. We emphasise that this phenomenon only arises when $\tau\in(1,2]$.\medskip

\subsection{Comparison to First Passage Percolation} The case that $\mu=0$, i.e., there is no weight penalty and $C_e=L_e$, is also known as \emph{first-passage percolation} (FPP)~\cite{hammersley1965first}. This process has been studied in classical scale-free networks like the configuration model~\cite[Theorem 4]{baroni2017nonuniversality} and~\cite[Theorem 2.4]{adriaans2018weighted}, and also for the networks considered in this paper, SFP, GIRGs, and HRGs~\cite{HofKom18,julia-girg}. 
 When the empirical degree distribution has  exponent $\tau\in(2,3)$, the authors in \cite{julia-girg, HofKom18} showed that explosion happens \emph{if and only if} the random variable $L$ representing the edge-costs is such that 
\be\label{eq:sumL} \mathbf{I}(L):=\sum_{k=1}^{\infty} F_L^{(-1)}\Big(1/\e^{\e^k}\Big)<\infty,\ee
where $F_L^{(-1)}(y)=\inf\{t\in \R: \Pv(L\le t)\ge y\}$ is the generalised inverse of the cdf $F_L$ of $L$. One can check that $\mathbf{I}(L)$ is finite for almost every well-known distribution with support starting at $0$, in fact, $F_L(t)$ has to be \emph{doubly-exponentially flat}\footnote{More precisely, on some small interval $[0, t_0]$ the distribution function  $F_L(t)$  has to satisfy $F_L(t)\le \exp\{-C_1\exp\{C_2/t^{\eta}\}\}$ for some $\eta>1$ and positive constants $C_1, C_2$. This corresponds to $F_L^{(-1)}(y)\ge 1/(\log \log (1/y))^{1/\eta}$ which makes the sum infinite when the sequence $1/\exp(\e^k)$ is substituted for $y$.} around $0$ for $\mathbf{I}(L)$ to be infinite and thus for explosion not to happen. Observe also that the sum $\mathbf{I}(L)$ does not depend on $\tau$, only on $L$. This is a counterintuitive phenomenon as it suggests that explosion does not depend on vertex degrees. The results suggest that allowing $\mu>0$ is a good way to fix this issue; the critical case occurs when $F_L$ is polynomially flat at the origin, rather than doubly-exponentially flat as for FPP.

\subsection{Proof techniques} Assume for simplicity that $F_L(t)$ grows at a polynomial rate around the origin\footnote{We allow for slowly varying function correction terms.}, i.e., $F_L(t)\asymp t^{\beta}$ in some small interval $[0, t_0]$. We show in Section~\ref{sec:expl} that explosion occurs in (continuum and ordinary) SFP and IGIRGs in the regime where $\mu\beta< (3-\tau)/2$ by constructing a path with infinitely many vertices and finite total cost. We do this by constructing an infinite sequence of annuli centred around the source vertex whose volumes grow doubly-exponentially. Each annulus contains doubly-exponentially many vertices that we call `leaders', which have weight that is doubly-exponential in the index of the given annulus. We show that any leader within an annulus is connected to double-exponentially many leaders in the next annulus. We then construct a finite-cost infinite path greedily by repeatedly choosing a least-cost edge to a leader in the next annulus. This construction succeeds when $\mu\beta< (3-\tau)/2$, since in this case the penalty for transmission through an edge between the leaders is not too high compared to the minimum of doubly-exponentially many copies of $L$. This argument is similar to the one used in~\cite{julia-girg}.

We show in Sections~\ref{sec:sideways} and~\ref{sec:cons} that explosion cannot occur in SFP or IGIRGs when $\mu\beta>(3-\tau)/2$. The argument is novel. We have to exclude both \emph{sideways explosion} and \emph{lengthwise explosion}. By sideways explosion, we mean that there are infinitely many vertices within finite cost reachable using finite-length paths; this naturally requires that some vertices have infinite degree. By lengthwise explosion we mean that there is a path of infinitely many edges with finite total cost. We exclude sideways explosion by showing that for any finite cost $t$, each vertex has only finitely many edges attached to with cost less than $t$ when $\mu\beta>(3-\tau)/2$. 
 To exclude lengthwise explosion, we show that if lengthwise explosion can happen at all, then it can happen arbitrarily quickly. That is, writing $T_\mathsf{exp}$ for the explosion time,
 if $T_\mathsf{exp} < \infty$ with positive probability, then for all $t_0 > 0$ we also have $T_\mathsf{exp} < t_0$ with positive probability. A similar phenomenon was previously observed in branching processes, where it arises due to the independence of the subtrees of the root. In spatial random graphs, the proof is more subtle. 
From here we argue by contradiction and show that when $t_0$ is sufficiently small, the probability that there is a vertex within graph distance $k$ and cost-distance $t_0$ of the source decays exponentially in $k$. Hence, almost surely for some $k$ no such vertex exists, and thus explosion does not occur.

Before we extend the results to finite GIRGs and HRGs, in Section~\ref{s:giant} we give a novel proof that these models contain a unique linear-sized giant component (see Theorem~\ref{thm:giant}). This is necessary since we work under milder assumptions on the edge connection probabilities than so far assumed in the literature (e.g.\ in \cite{DeiHofHoo13, BriKeuLen19,bringmann2016average, BodFouMul15, FouMul18, HeyHulJor17}). The argument is based on a bottom-to-top approach: the space is divided into boxes of growing size, and we show that each box contains, independently of each other and with positive probability, a linear-sized ``local giant''.  These local giants are then merged into a single linear-sized largest component via paths through the leaders as used in the explosive case above. Uniqueness is shown by a standard sprinkling argument. 

Finally, in Section~\ref{sec:finite-result} we extend the results to finite GIRGs and HRGs. This process is rather subtle. Because of the polynomial transmission penalties, many of the methods developed in~\cite{julia-girg} break down, and we must develop a new argument for connecting two uniformly chosen vertices within their respective explosion times plus a negligible cost. This argument depends crucially on the fact that the explosion time of any vertex is mostly determined by large-but-finite neighborhoods of vertices with bounded weight. If we carefully maintain independence, edges between high-weight vertices can then be used to establish the necessary low-cost connection between the two neighborhoods.

	\section{Notation}\label{sec:notation}
	We write rhs and lhs for right-hand side and left-hand side respectively, wrt for with respect to, rv for random variable, i.i.d.\! for independent and identically distributed, and cdf for cumulative probability distribution function. 
	Generally, we write $F_X$ for the cdf of an rv $X$, and $F^{(-1)}_X$ for its generalised inverse function, defined as $F^{(-1)}_X(y):=\inf\{t\in\mathbb{R}:F_X(t)\geq y\}$.
	We say that an event $A$ happens almost surely (a.s.) when $\Pv(A)=1$ and a sequence of events $(A_n)_{n\in\mathbb{N}}$ holds with high probability (whp) when $\lim_{n\rightarrow \infty}\Pv(A_n)=1$. 
	We say a sequence of rvs $(X_n)_{n\in\N}$ is \emph{tight} if for every $\eps>0$ there exists a $K_\eps>0$ such that $\Pv(|X_n|> K_\eps)<\eps$ for all $n$. We say that a real function $f$ \emph{varies slowly at infinity} if for all $c>0$, $\lim_{x\to\infty} f(cx)/f(x) = 1$; note in particular that by Potter's bound \cite{BinGolTeu89} this implies that as $x\to\infty$, $f(x)=o(x^\delta)$ and $\omega(x^{-\delta})$ for all $\delta>0$.
	We say that the positive random variable $X$  \emph{has power-law tails} with exponent $\tau$ if for all sufficiently large $x$,
	\be \Pv(X\ge x)=\ell(x)/x^{\tau-1}\ee
	for some function $\ell(x)$ that varies slowly at infinity. 
	
We write $\R^+ = (0,\infty)$ and $\Z^+$ for the set of positive integers. For $n\in\Z^+$, let $[n]:=\{1,\ldots,n\}$. If two functions $f,g$ have range $\R^+$ and any of the domains $\R$, $\Z$, or $[a,\infty)$ for some $a\in \R$, then we use the standard Landau notation $f=O(g)$, $f= o(g)$, $f = \Omega(g)$, $f = \omega(g)$ and $f=\Theta(g)$ as in~\cite[Section 1.2]{JLR}. We also abbreviate $f=\Theta(g)$ with $f \asymp g$. For all $x \in \R$, we denote by $\lfloor x\rfloor$ and $\lceil x\rceil$ the lower and upper integer parts of $x\in\mathbb{R}$, respectively. We write $x\wedge y:=\min\{x,y\}$ and $x\vee y:=\max\{x,y\}$. 
	We denote a graph by $G= (\mathcal V, \mathcal E)$, where $\mathcal V$ is the vertex set and $\mathcal E \subseteq \mathcal{V}^2$ is the edge set. We will assign a geometric position $x_v\in \R^d$ to each vertex $v \in \mathcal V$, and for a subset $A\subseteq \R^d$ we will write $\mathcal{V} \cap A := \{v\in \mathcal{V} \mid x_v \in A\}$, by slight abuse of notation. 
	For two vertices  $u,v$, let $u\leftrightarrow v$ denote the event that $u$ and $v$ are connected by an edge $e=(u,v)$. All these graphs are undirected, i.e., $(u,v) \in \mathcal E$ if and only if $(v,u)\in \mathcal E$, for all $u, v\in \mathcal V$. However, when we wish to consider transmission along an edge, its direction will matter, so we define $e_-:=u, e_+:=v$. If two or more vertices are chosen uniformly at random from a set $S \subseteq \mathcal{V}$, then we say that they are \emph{typical vertices} in $S$. A \emph{walk} in $G$ is a finite sequence of vertices $\pi = (\pi_0,\dots,\pi_k)$ connected by edges $(\pi_i,\pi_{i+1})$, and a \emph{path} is a walk in which all vertices are distinct. We call $|\pi| = k-1$ the \emph{length} of a walk. As usual, the \emph{graph distance} is defined by
	\[
		d_G(A,B) := \inf \big(\{|\pi| \colon \pi = (\pi_0,\dots,\pi_k)\mbox{ is a path with }\pi_0 \in A \mbox{ and } \pi_k \in B\} \cup \{\infty\}\big).
	\]
	If $A$ or $B$ contains only a single vertex, we omit the surrounding braces, writing e.g.\ $d_G(u,v)$ instead of $d_G(\{u\},\{v\})$. We denote balls in this metric by $B^G(v,r) := \{u \in \mathcal{V} \colon d_G(u,v) \le r\}$ for all $v \in \mathcal{V}$ and $r \ge 0$, and we denote their boundaries by $\partial B^G(v,k) := B^G(v,k) \setminus B^G(v,k-1)$ for all integers $k \ge 1$.
	
For an integer $d \ge 1$, we write $\calX_d := [-1/2,1/2]^d, \calX_d(n):=[-n^{1/d}/2,n^{1/d}/2]^d$. We write $\nu_d$ for $d$-dimensional Lebesgue measure, and $\|x\|$ for the Euclidean norm of $x$. We denote Euclidean balls by $B^2(x,r) := \{y \in \R^d \colon \|y-x\| \le r\}$ for all $x \in \R^d$ and all $r \ge 0$.

	\section{Formal definitions and statements of results}\label{sec:results}
	
	\subsection{Definitions of infinite models}\label{sec:infinite-models}
	
	We start by defining the two infinite models, scale-free percolation (SFP) and Infinite Geometric Inhomogeneous Random Graphs (IGIRG). This latter contains, as a special case, continuum SFP \cite{DepWut18}. Later, in Section \ref{sec:finite-result}, we define finite-sized variants and discuss how the results on the infinite models carry through to their finite counterparts. 
	SFP was introduced by Deijfen, van der Hofstad and Hooghiemstra in \cite{DeiHofHoo13} as an extension of long-range percolation~\cite{Trap10, Bisk04}. First passage percolation on SFP was studied in \cite{HofKom18}, and behaviour of random walks on SFP was studied in \cite{HeyHulJor17}. We consider a version of the model which allows more general edge-connection probabilities. 
	
	\begin{definition}[Generalised Scale-free Percolation]\label{def:SFP}Let $h_S\colon \R^d \times \R_+ \times \R_+ \to [0,1]$ be a function, let $W \ge 1$ and $L \ge 0$ be random variables, and let $d\ge 1$ be an integer. For each vertex $v \in \Z^d$ we draw a random vertex weight $W_v$, which is an i.i.d.\ copy of $W$. All pairs of orthogonally adjacent vertices are joined by an edge. Conditioned on $(W_i)_{i\in\Z^d}$, all other edges are present independently with probability
		\be\label{SFP}
		\mathbb{P}\left(u\leftrightarrow v\mid \|u-v\|>1, \left(W_i\right)_{i\in\mathbb{Z}^d}\right)=h_{\text{S}}(u-v, W_u, W_v).		\ee
		Finally, we assign to each present edge $e$ an edge-length $L_e$,  an i.i.d.\ copy of $L$. We denote the resulting random graph on $\Z^d$ by $\mathrm{SFP}_{W,L}$.
	\end{definition}	
	In \cite{DeiHofHoo13} the function $h_{\mathrm{S}}$ was defined, for a \emph{long-range} parameter $\alpha_S>d$ and a \emph{percolation} parameter $\lambda>0$, to be
	\be \label{eq:h-orig}h_{\text{S}}^{\text{orig}}(x, w_1, w_2)=	1-\exp\left(-\lambda w_1 w_2/ \|x\|^{\alpha_S} \right).\ee
	This parametrisation of SFP is not very natural, since a vertex with weight $W_v$ has degree approximately $W_v^{d/\al_S}$ \cite[Proposition 2.3, Proof of Theorem 2.2]{DeiHofHoo13} rather than $W_v$, and the exponent of the degree distribution's power law is different from that of $W$'s power law. To remedy this, we re-parametrise by taking $W^\text{new} = W^{d/\al_S}$ and $\alpha = \alpha_S/d$, so that $h_S(x,w_1,w_2) = 1-\exp(-\lambda (w_1w_2/\|x\|^{d})^\alpha) \asymp 1\wedge (w_1w_2/\|x\|^{d})^\alpha$ and each vertex $v$ has degree approximately $W_v$. We actually allow significantly more general choices of $h_S$, which we will set out momentarily in Assumption~\ref{assu:h}.

The second model we consider is IGIRG. 
The main difference between IGIRG and SFP is that the vertex set of IGIRG is given by a Poisson Point Process on $\R^d$ instead of the grid $\Z^d$. This model is the generalisation of  (finite) Geometric Inhomogeneous Random Graphs \cite{BriKeuLen19} (GIRGs) to infinite space $\R^d$, and contains continuum SFP from \cite{DepWut18} as a special case.
\begin{definition}[Infinite  Geometric Inhomogeneous Random Graphs]\label{def:IGIRG} Let $h_\mathrm{I}:\R^d\times \R_+\times\R_+\to [0,1]$ be a function, let $W\ge1, L\ge 0$ be random variables, let $d \ge 1$ be an integer, and let $\la>0$. We define the infinite random graph model $\Ilambda$ as follows. 
		Let $\CV_\lambda$ be a homogeneous Poisson Point Process (PPP) on $\R^d$ with intensity $\la$, forming the positions of vertices. For each $v\in \CV_\lambda$ draw a random weight $W_v$, an i.i.d. copy of $W$. Then, conditioned on $(z, W_z)_{z\in \CV_\lambda}$,  edges are present  independently with probability 
		\be\label{eq:EGIRG-prob}  \Pv(u \leftrightarrow v \text{ in } \Ilambda \mid (z,W_z)_{z\in \CV_\la}):= h_{\mathrm{I}}( u-v, W_{u}, W_v).  \ee
		Finally, we assign to each present edge $e$ an edge-length $L_e$, an i.i.d.\ copy of a random variable $L\ge 0$.
		We write $(\CV_\la, \CE_\la)$ for the vertex and edge set of the resulting graph, which we denote by $\Ilambda$. \end{definition}
The edge-connectivity functions $h_{\mathrm{S}}$ and $h_{\mathrm{I}}$ as stated in Definitions~\ref{def:SFP} and~\ref{def:IGIRG} are too general, so  we require the following additional assumption.
Write  for some $c_2>0, \gamma\in(0,1)$,
\be\label{eq:ellw-111} l_{c_2,\gamma}(w):=\exp(-c_2 \log (w)^{\gamma}).\ee
\begin{assumption}[Edge-connection bounds]\label{assu:h} We assume that there exist parameters $\al \in (0,\infty]$ and $\gamma\in(0,1)$, and constants $\underline c, \overline c, \underline c_1, \overline c_1, c_2>0$, such that for each $\mathrm{q}\in\{\mathrm{S,I}\}$, $h_{\mathrm{q}}:\R^d\times \R\times \R\mapsto [0,1]$ (considered as a deterministic function) satisfies the following bounds. For $\alpha <\infty$, we require	
\be\label{eq:connection-new}
		\underline c\left(l_{c_2,\gamma}(w_1)l_{c_2,\gamma}(w_2)\wedge \big(w_1 w_2/\|x\|^d\big)^\al \right)\leq h_{\mathrm{q}}(x, w_1, w_2) \leq \overline c\left(1\wedge \big(w_1 w_2/\|x\|^d\big)^\al\right).
		\ee
		For $\al=\infty$,\footnote{In the related graph models GIRG and HRG, this is called the \emph{threshold case}. We refrain from this terminology, to avoid confusion with other thresholds.} we require
\be\label{eq:connection-new-threshold}
		\underline c\left(l_{c_2,\gamma}(w_1)l_{c_2,\gamma}(w_2)\wedge \ind_{\{\underline c_1 w_1 w_2 \geq \|x\|^d\}} \right)\leq h_{\mathrm{q}}(x, w_1, w_2) \leq \overline c\cdot\ind_{\{\overline c_1w_1 w_2 \geq \|x\|^d\}}.
	\ee
	Unless otherwise mentioned, we will also require that $W$ has power-law tails (as defined in Section~\ref{sec:notation}).
		\end{assumption}

In order to formally define explosion, we must first define the cost of a walk. The notation here is analogous to the notation we use for graph distance and Euclidean distance (see Section~\ref{sec:notation}).

	\begin{definition}[Distances and metric balls]\label{def:distances}
		Let $\CG = (V,E)$ be a graph. Let $\{L_e \colon e \in E\}$ be an associated family of \emph{edge lengths}, where $L_e \in (0,\infty)$ for all $e$. Let $\{W_v \colon v \in V\}$ be an associated family of \emph{vertex weights}, where $W_v \in [1,\infty)$ for all $v \in V$. Let $f\colon [1,\infty)^2 \to \R^+$ be a function, which we call the \emph{weight penalty function}.
		For all directed edges $e = (u,v) \in E$, we define the \emph{cost} $C_e$ of $e$ to be $L_e f(W_u,W_v)$. For all walks $\pi = (\pi_0,\dots,\pi_k)$ in $G$, we define the \emph{cost} of $\pi$ to be
		\[
			|\pi|_{f,L} := \sum_{i=1}^k C_{(\pi_{i-1},\pi_i)} = \sum_{i=1}^k f(W_{\pi_{i-1}},W_{\pi_i})L_{(\pi_{i-1},\pi_i)}
		\]
		For all sets $A,B\subseteq V$, we define the \emph{cost-distance} from $A$ to $B$ by
		\[
			d_{f,L}(A,B) := \inf \big(\{|\pi|_{f,L} \colon \pi = (\pi_0,\dots,\pi_k)\mbox{ is a path with }\pi_0 \in A \mbox{ and } \pi_k \in B\} \cup \{\infty\}\big).
		\]
		As with graph distance, we write e.g.\ $d_{f,L}(u,v) := d_{f,L}(\{u\},\{v\})$. We denote balls by $B^{f,L}(v,r) := \{u \in V \colon d_{f,L}(v,u) \le r\}$ for all $v \in V$ and $r \ge 0$. For the special case $f(x,y)=(xy)^{\mu}$ with $\mu \geq 0$, we replace $f$ by $\mu$ in the definitions above, writing $|\pi|_{f,L}=|\pi|_{\mu,L}, d_{\mu,L}\left(A, B\right)=d_{f,L}\left(A, B\right)$, and $B^{\mu,L}\left(v,r\right)=B^{f,L}\left(v,r\right)$.
	\end{definition} 

	Note that $d_{f,L}$ is a metric only if $f$ is symmetric. For asymmetric $f$, we interpret $d_{f,L}(v,u)$ as the transmission cost from $v$ to $u$.
	To define explosion time, we single out a vertex $v\in \CV_\la$ and study the set of vertices that can be reached from $v$ within some cost distance $T$. To simplify the notation, we always condition on the event that an IGIRG's PPP has a vertex at the origin $0$, and then choose $v=0$. Conditioned on $0 \in \CV_\la$, the set $\CV_\la \setminus \{0\}$ is still a PPP of intensity $\lambda$, and by translation invariance, the results generalise to arbitrary fixed vertices $v \in \CV_\la$.
	\begin{definition}[Explosion time]\label{def:explosiontime}
		Consider $\mathrm{IGIRG}_{W,L}(\la)$ from Definition \ref{def:IGIRG}, and let $v\in \CV_\la$.  Let $f$ be a weight penalty function, and define $\sigma_f^{\mathrm{I}}(v,k):=\inf\{t: |B^{f,L}(v,t)|>k \text{ in } \Ilambda\}$. Let the \emph{explosion time} of a vertex $v$ (wrt cost-distance) be defined as the (possibly infinite) limit:
		\be Y_f^{\mathrm{I}}(v):=\lim_{k\to \infty} \sigma_f^{\mathrm{I}}(v,k).\ee 
		The explosion time of the origin, $Y_f^{\mathrm{I}}(0)$, is defined analogously when we condition on $0\in \CV_\la$. We call $\Ilambda$ \emph{explosive} wrt to the cost-distance generated by $f$ and $L$ if $\Pv(Y_f^{\mathrm{I}}(0)< \infty)>0$, otherwise we call it \emph{conservative}. For short, we write that $\Ilambda$ is $(f,L)$-explosive or $(f,L)$-conservative, and for $f=(xy)^\mu$, we write $(\mu, L)$-explosive vs.\ $(\mu, L)$-conservative, respectively.  We call any infinite path $\pi$ with $|\pi|_{f,L}<\infty$ \emph{an explosive path}. 
		We define the same quantities analogously in the model $\SFPWL$, indicating the different model by replacing the superscript $\mathrm I$ with $\mathrm S$: thus we write $\sigma_f^{\mathrm{S}}(v,k)$ and $Y_f^{\mathrm{S}}(v)$.
	\end{definition}
	In other words, $\sigma_f^\mathrm{I}(v,k)$ is the smallest cost $t$ such that $k$ other vertices are reachable within cost $t$ from $v$; similarly, the explosion time $Y^{\mathrm{I}}_f(v)$ is the infimum\footnote{This infimum is a minimum if all degrees are finite.} of all costs $t$ such that the ball $B^{f,L}(v,t)$ contains infinitely many vertices. Thus $Y_f^{\mathrm I}(v)$ is finite if and only if infinitely many vertices are reachable within bounded cost from $v$.
	
 Explosion can either happen as \emph{lengthwise explosion}, in which there is an explosive path from the origin. Or (non-exclusively) there may be \emph{sideways explosion}, in which there is a finite path from the origin to a vertex from which there are infinitely many incident edges of bounded cost. Actually, we will show in Lemma~\ref{lem:opt-real} that whenever sideways explosion occurs in $\Ilambda$ or $\SFPWL$, then with positive probability from the origin itself there are already infinitely many incident edges of bounded cost. 
	
	When the weight penalty function $f$ is asymmetric, so that $d_{f,L}$ is only a quasimetric, we call the phenomenon described in Definition~\ref{def:explosiontime} \emph{outwards} explosion. We define \emph{inwards} explosion analogously, requiring that infinitely many vertices have bounded-cost paths to $0$ instead of the other way around. Thus we say that $\Ilambda$ is \emph{$(f,L)$-inwards explosive} if 
	\[
		\lim_{k\to\infty} \big( \inf \{t \colon |\{u \in \CV_\lambda \colon d_{f,L}(u,0) \le t\}| > k \}\big) < \infty.
	\]
	The corresponding definition for $\SFPWL$ is analogous.

\subsection{Results for infinite models}
 	
	With the definition of explosion at hand, we are now able to state the main results of this paper. 
	We first consider the special case where the weight penalty function is given by $f(w_u,w_v) = (w_uw_v)^\mu$ for some $\mu > 0$. It is already known~\cite{HofKom18} that when the exponent $\tau$ of the vertex weights' power law is greater than $3$, $\SFPWL$ is $(0,L)$-conservative for all distributions $L$ satisfying $\Pv(L=0) = 0$; this implies that $\SFPWL$ is also $(\mu,L)$-conservative for all $\mu > 0$, since increasing $\mu$ only increases the cost of each path. We expect $\Ilambda$ to exhibit the same behaviour, so we focus on the $\tau < 3$ regime.
	
	\begin{theorem}\label{thm:mu-classification}
		Consider the models $\Ilambda$ and $\SFPWL$. 
		Suppose the vertex-weight distribution $W$ has power-law tails with exponent $\tau\in(1,3)$, and that the connection functions $h_{\mathrm{I}}$ and $h_{\mathrm{S}}$ satisfy Assumption \ref{assu:h} for some $\alpha \in (0,\infty]$. Let the weight penalty function be $f(w_1, w_2):=(w_1w_2)^\mu$, for some $\mu>0$. Then the following statements hold.
		\begin{enumerate}[(i)]
			\item Suppose $\alpha \le 1$. Then $\Ilambda$ and $\SFPWL$ are $(\mu,L)$-sideways explosive, and moreover explosion occurs almost surely.
			\item Suppose $\alpha > 1$ and that
			\be\label{eq:beta1} \beta^+ := \limsup_{t\to 0} \big(\log F_L(t)/\log t\big)<(3-\tau)/(2\mu).\ee
			Then $\Ilambda$ and $\SFPWL$ are $(\mu,L)$-lengthwise explosive.
			\item Suppose $\alpha > 1$ and that
			\be\label{eq:beta2} \beta^- := \liminf_{t\to 0} \big(\log F_L(t)/\log t\big)>(3-\tau)/(2\mu).\ee
			Then $\Ilambda$ and $\SFPWL$ are $(\mu,L)$-conservative.
		\end{enumerate}
	\end{theorem}	
	Thus when $\alpha \le 1$ both models are always explosive, and when $\alpha > 1$ and $\beta^- = \beta^+$ the hyperbola curve $\mu\beta = (3-\tau)/2$ is the threshold between the explosive and conservative regimes. Observe that a smaller $\beta$ means steeper cdf $F_L$ at the origin, while a larger $\beta$ means flatter behavior at the origin, hence the criterions \eqref{eq:beta1} and \eqref{eq:beta2} are quite natural. 
We expect $\beta^- = \beta^+$ for any reasonable choice of $L$; in particular, when $L\sim\Exp(1)$ we have $\beta^- = \beta^+ = 1$. Note that $\beta^+ =0$ and $\beta^- = \infty$ are both allowed in Theorem~\ref{thm:mu-classification}. The case $\beta^+ = 0$ occurs when $F_L(t)$ is steeper at $0$ than any polynomial, e.g.\ if $F_L(t) \asymp -1/\log(t)$ as $t\to 0$, and the case $\beta^- = \infty$ occurs when $F_L(t)$ is flatter at $0$ than any polynomial, e.g.\ if $F_L(t) \asymp \exp(-1/t)$ as $t\to 0$. Thus Theorem~\ref{thm:mu-classification} gives a partition of the parameter space into explosive and conservative regimes whenever $F_L$ is suitably well-behaved near the origin.
	
	Note that the parameter regimes $\alpha \le 1$ and $\tau \le 2$ are exceptional in the sense that $W_v$ does not correspond to the expected degree of $v$ up to constant factors. Rather, whenever $\alpha \le 1$ or $\tau \le 2$ then all vertex degrees are a.s.\! infinite in both $\Ilambda$ and $\SFPWL$ (see \cite[Theorem 2.1]{DeiHofHoo13} for the $\tau \le 2$ regime). It is therefore immediate that without a weight penalty (i.e.\ taking $\mu = 0$), explosion occurs in time $\inf \{t \colon F_L(t) > 0\}$ in these regimes. 
	Despite this, Theorem~\ref{thm:mu-classification}(ii) implies that the weight penalty is powerful enough that $\Ilambda$ and $\SFPWL$ may still be conservative when $\tau \le 2$, and indeed that the critical hyperbola is smooth at $\tau = 2$. The reason for this is that in the $\tau \le 2$ regime, the infinite vertex degrees come from edges to high-weight vertices. Thus when \eqref{eq:beta2} holds, the weight penalty ensures that for all costs $K$, there are only finitely many edges from zero with cost at most $K$; thus the cost-distance does not ``see'' the infinite degree. By contrast, when $\alpha \le 1$, the infinite vertex degrees come from low-weight neighbors and so the weight penalty does not matter. In fact, when $\alpha \le 1$ we prove explosion occurs under substantially weaker conditions than stated in Theorem~\ref{thm:mu-classification}(i); see Theorem~\ref{thm:explosive0} for details.
	
	We now comment on the critical case, where $\alpha > 1$ and $\beta^- = \beta^+ = (3-\tau)/2\mu$. In this case, the proof shows that both models are $(\mu,L)$-conservative whenever the moment $\Ev[W^{2-2\mu\beta}]=\Ev[W^{\tau-1}]$ is finite. This occurs when the slowly varying function in the precise tail of the vertex-weight distribution $W$ is sufficiently small, for example if $\Pv(W\ge x)=(\log(x))^{-2}/x^{\tau-1}$. When the moment $\Ev[W^{\tau-1}]$ is infinite, the proof technique we apply breaks down. This case is hard even without a weight penalty, i.e.\ when $\mu = 0$ and $\tau =3$: this threshold regime is not understood even in the simplest case of branching processes~\cite{AmiDev13}.

	Finally, we emphasise that the case $\mu=0$ (where there is no weight penalty and $C_e=L_e$) implies entirely different scaling from Theorem~\ref{thm:mu-classification}; in this case the criterion for both models to be explosive is for the sum $\mathbf{I}(L)$ in \eqref{eq:sumL} to be finite, a result that was established for SFP in \cite{HofKom18} and for IGIRG in \cite{julia-girg}.

{\bf Asymmetric and polynomial penalty functions.} Next, we generalise Theorem~\ref{thm:mu-classification} to (possibly asymmetric) polynomial penalty functions.
Let $f(w_1, w_2)$ be any polynomial of two variables with positive real coefficients and non-negative real exponents. Thus $f$ can be written in the form
\be\label{eq:pol} f(w_1, w_2)=\sum_{i\in \CI} a_i w_1^{\mu_i} w_2^{\nu_i}\ee
for some finite set $\CI$, where $a_i>0$ and $\mu_i, \nu_i\ge 0$ for all $i\in \CI$.
We define the \emph{degree} of $f$ to be 
\be\label{eq:poldeg} \deg(f)=\max_{i\in\CI}(\mu_{i}+\nu_{i}).\ee
Further, for some $\mu > 0$ we define $f_{\vee,\mu}:=(w_1\vee w_2)^\mu$ and $f_{+,\mu}(w_1,w_2):=(w_1+w_2)^\mu$.
As in the $f(w_u,w_v) = (w_uw_v)^\mu$ case, the results of~\cite{HofKom18} directly imply that $\SFPWL$ is conservative for any choice of $f$ when $\tau > 3$, so we focus on the $\tau < 3$ regime. Theorem~\ref{thm:mu-classification} becomes the following:
\begin{theorem}\label{thm:poly-classification}
	Consider the models $\Ilambda$ and $\SFPWL$. 
	Suppose the vertex-weight distribution $W$ has power-law tails with exponent $\tau\in(1,3)$, and that the connection functions $h_{\mathrm{I}}$ and $h_{\mathrm{S}}$ satisfy Assumption \ref{assu:h} for some $\alpha \in (0,\infty]$. Let the weight penalty function $f$ be a polynomial as in~\eqref{eq:pol}. Define $\beta^-$ and $\beta^+$ as in Theorem~\ref{thm:mu-classification}. Then the following statements hold.
	\begin{enumerate}[(i)]
		\item Suppose $\alpha \le 1$ or that for all $i \in \calI$, $\beta^+ < (2-\tau)/\nu_i$. Then $\Ilambda$ and $\SFPWL$ are $(f,L)$-sideways explosive, and moreover explosion occurs almost surely.
		\item Suppose $\alpha > 1$ and that $\beta^+ < (3-\tau)/\deg(f)$. Then $\Ilambda$ and $\SFPWL$ are $(f,L)$-lengthwise explosive.
		\item Suppose $\alpha > 1$, that $\beta^- > (3-\tau)/\deg(f)$, and that for some $i \in \calI$ with $\mu_i + \nu_i = \deg(f)$ we have $\beta^- > (2-\tau)/\nu_i$. Then $\Ilambda$ and $\SFPWL$ are $(f,L)$-conservative.
	\end{enumerate}	
\end{theorem}
If $f$ is a general asymmetric polynomial, then while the bounds of Theorem~\ref{thm:poly-classification}  still apply it need not give a partition of the entire parameter space. For example, Theorem~\ref{thm:poly-classification} does not apply (and the proof technique breaks down) if $\alpha > 1$, $\tau = 3/2$, $\beta^- = \beta^+ =: \beta$, and $f(w_u,w_v) = w_u^{7/4\beta} + w_v^{3/4\beta}$. However, we do recover a partition in many special cases:
\begin{itemize}
	\item When $\tau > 2$, the condition $\beta^- > (2-\tau)/\nu_i$ is automatically satisfied for all $i \in \calI$, so it can be dropped.
	\item If $f$ is a monomial, then conditions (i)--(iii) cover the whole of the parameter space.
	\item If $f$ is symmetric and $\tau \in (1,2]$, then there exists $i \in \calI$ with $\mu_i \le \nu_i$ and $\mu_i + \nu_i = \deg(f)$; thus
	\be \label{eq:compare-conditions}
		\frac{3-\tau}{\deg(f)} = \frac{3-\tau}{\mu_i + \nu_i} \ge \frac{2(2-\tau)}{2\nu_i} = \frac{2-\tau}{\nu_i}.
	\ee
	It follows that the condition $\beta^- > (2-\tau)/\nu_i$ can be safely removed from condition (iii).
\end{itemize}

In particular, Theorem~\ref{thm:poly-classification} yields a partition of the parameter space for $f_{*,\mu}(w_1,w_2) := w_1^\mu + w_2^\mu$. Moreover, since $(w_1^\mu +w_2^\mu)/2 \le (w_1 \vee w_2)^\mu \le (w_1 + w_2)^\mu \le 2^\mu(w_1^\mu + w_2^\mu)$, for any walk $\pi$ we have $|\pi|_{f_{*,\mu},L}/2 \le |\pi|_{f_{\vee,\mu},L} \le |\pi|_{f_{+,\mu},L} \le 2^\mu|\pi|_{f_{*,\mu},L}$. Hence any given instance of $\Ilambda$ or $\SFPWL$ explodes under $f_{\vee,\mu}$ and $f_{+,\mu}$ if and only if it explodes under $f_{*,\mu}$, and the same partition applies.

Finally, we note that Theorem~\ref{thm:poly-classification} implies --- perhaps surprisingly --- that an asymmetric penalty function can yield asymmetric explosive behaviour even in undirected models such as $\Ilambda$ and $\SFPWL$. For example, taking $\tau = 3/2$ and $\beta^+ = \beta^- = 1$, Theorem~\ref{thm:poly-classification}(i) and (iii) imply that $f(w_u,w_v) = w_u^3w_v^{1/4}$ explodes almost surely, but the reverse function $f^{\text{rev}}(w_u,w_v) = w_u^{1/4}w_v^3$ does not. Equivalently, $f$ demonstrates outwards explosion but not inwards explosion, and $f^{\text{rev}}$ demonstrates inwards explosion but not outwards explosion. We emphasise that this behaviour is only possible when $\tau \le 2$.
\subsection{Finite models and results}\label{sec:finite-result}
	In this section we define the finite version of the IGIRG model, as well as the Hyperbolic Random Graph model, and explain how the results carry over to these finite versions. We start with Geometric Inhomogeneous Random Graphs (GIRGs), introduced in~\cite{BriKeuLen15,BriKeuLen19}. Various aspects of the GIRG model have been studied, including average distances~\cite{bringmann2016average}, greedy routing~\cite{BriKeuLenYan17}, bootstrap percolation~\cite{KocLen16}, first passage percolation~\cite{julia-girg}, and how to sample from the graph model efficiently~\cite{BlaFriKat19}. Extensions to non-metric geometries were studied in~\cite{lengler2017existence}.
	
	\begin{definition}[Geometric Inhomogeneous Random Graph]\label{def:GIRG}  Let $n \in \Z^+$, and let $W^{(n)}\ge 1,\ L\ge 0$ be random variables.
		Let $V:=\left[n\right]$, and consider $\mathcal{X}=[-1/2,1/2]^d$ equipped with the Lebesgue measure $\nu$. Assign to each vertex $i\in [n]$  an i.i.d.\ position vector $x_i\in \mathcal{X}$ sampled from $\nu$, and a vertex-weight $W_i^{(n)}$,  an i.i.d.\ copy of   $W^{(n)}$. Then, conditioned on $(x_i, W_i^{(n)})_{i\in [n]}$  edges are present  independently. For any $u,v\in [n]$, we denote
		\be\label{eq:gn-intro}
		\mathbb{P}\big(u\leftrightarrow v \text{ in } \mathrm{GIRG}_{W,L}(n)\mid (x_i, W_i^{(n)})_{i\in [n]}\big)=:g_n^{u,v}\big(x_u,x_v,(W_i^{(n)})_{i\in [n]}\big),
		\ee
		and we require $g_n^{u,v}: \mathcal{X} \times \mathcal{X}\times (\mathbb{R}^+)^n \rightarrow [0,1]$ to be $\nu$-measurable. 
		Finally, assign to each present edge $e$ an edge-length $L_e$, an i.i.d.\ copy of $L$.
		We denote the resulting graph by $\mathrm{GIRG}_{W,L}(n)$.
	\end{definition} 

	We next set out the properties we assume for $g_n^{u,v}$ (Assumption~\ref{assu:mild-connect}) and for $W^{(n)}$ (Assumption~\ref{assu:mild-vertex}). In \cite{BriKeuLen19}, the authors assumed that
	there is a parameter $\alpha\in (1,\infty]$, and $0<\underline c, \overline c, \underline c_1, \overline c_1 < \infty$, such that for all $n$  and all $u,v\in[n]$, the edge-connectivity function $g_n^{u,v}$ satisfies, as a \emph{deterministic} function from $\mathcal{X}\times \mathcal{X}\times (\R^+)^n$ to $[0,1]$, that 
	\be
	\label{eq:GIRG-orig}
	\underline c  \le\frac{g_n^{u,v}\big(x_u,x_v,(w_i)_{i\in [n]}\big)}{1\wedge (w_u w_v/(\|x_u-x_v\|^d\sum_{i=1}^n w_i))^\alpha} \le \overline c \quad \text{for } 1<\alpha <\infty; 
	\ee
	while for  $\alpha=\infty$,
	\be
	\label{eq:GIRG-orig-threshold}
	\underline c \ind_{\{\underline c_1 w_u w_v \geq \|x_u-x_v\|^{d}\sum_{i=1}^n w_i \}} \le g_n^{u,v}\big(x_u,x_v,(w_i)_{i\in [n]}\big) \le \overline c \ind_{\{\overline c_1 w_u w_v \geq \|x_u-x_v\|^{d}\sum_{i=1}^n w_i\}}.
	\ee
The reason for the restriction $\alpha >1$ was that under the above conditions there are constants $c,C, n_0>0$ such that $\E[\deg(v) \mid W_i^{(n)} = w] \in [cw,Cw]$ holds in $\mathrm{GIRG}_{W,L}(n)$ for all $n\geq n_0$, $v\in [n]$ and $w \geq 1$. The same would not hold for $\alpha \leq 1$, where the expected degrees grow with $n$. 

Observe that conditions \eqref{eq:GIRG-orig} and~\eqref{eq:GIRG-orig-threshold} are very similar to Assumption \ref{assu:h}. There are two main differences. One is the factor $\sum_{i=1}^n w_i$ in the denominator. In the finite GIRG model, the weights are i.i.d.\! random  variables $W_i^{(n)}$. Writing $S:= \sum_{i=1}^n W_i^{(n)}$, whenever $\Ev[W^{(n)}]<\infty$ (including the power-law case with exponent $\tau >2$, considered in~\cite{BriKeuLen19}), there are constants $c,C > 0$ such that whp $S\in(c n, C n )$. Hence it is natural to replace the sum in \eqref{eq:GIRG-orig} and \eqref{eq:GIRG-orig-threshold} by $n$ and compensate by changing the constant prefactors. As we will see later, this factor of $n$ simply reflects the different scaling of the models: the infinite model uses a Poisson point process of intensity one, while the finite model places $n$ points in a cube of volume one. Thus the factor of $\sum_{i=1}^n w_i$ does not constitute a major difference from Assumption \ref{assu:h} if $\Ev[W^{(n)}]<\infty$.
	
The second difference is that the lower bound in Assumption \ref{assu:h} is milder: Assumption~\ref{assu:h} allows for a correction term $\exp(-c_2 (\log w_u)^\gamma)\exp(-c_2 (\log w_v)^\gamma)$. The weaker Assumption~\ref{assu:mild-connect}, stated below, incorporates this correction term. It does not change any of the qualitative behaviour of the model, but it will be important in proving our results, as we can discard edges from a GIRG satisfying Assumption~\ref{assu:mild-connect} independently at random and still recover a GIRG satisfying Assumption~\ref{assu:mild-connect}. This will allow us to use weight-dependent percolation, passing to a GIRG containing only low-cost edges in order to connect two high-weight vertices with a low-cost path.
		Let $l(w):=l_{c_2, \gamma}(w)$ from \eqref{eq:ellw-111}. 
	\begin{assumption}\label{assu:mild-connect} Consider $\GIRG$ in Definition \ref{def:GIRG}. 	We assume there exist parameters $\al \in (1,\infty]$ and $\gamma\in(0,1)$, and constants $0<\underline c \leq \overline c <\infty$ and $c_2>0$, such that for all $n\in\Z^+$, all $u,v \in [n]$, all sequences $(x_i)_{i \in [n]}$ in $\R^d$, and all sequences $(w_i)_{i \in [n]}$ in $[1,\infty)$, the function $g_n^{u,v}$ in \eqref{eq:gn-intro} satisfies the following. If $1<\alpha <\infty$, we require
		\begin{equation}\label{GIRGgeneral}\ba 
		\un c \cdot \bigg( l(w_u)l(w_v)\wedge \Big( \frac{w_u w_v}{n \|x_u-x_v\|^{d}}\Big)^\alpha\bigg) \le g_n^{u,v}\big(x_u,x_v, (w_i)_{i\in [n]}\big) \le   \overline c \cdot \bigg( 1\wedge \Big( \frac{w_u w_v}{n \|x_u-x_v\|^{d}}\Big)^\alpha\bigg).
		\ea\end{equation}
		If $\alpha = \infty$ then we require that for some constants $\un c_1, \overline c_1\in(0, \infty)$,

		\begin{equation}\label{GIRGgeneral-threshold}
		\un c \cdot \Big( l(w_u)l(w_v)\wedge\ind_{\{\un c_1 w_u w_v \geq n\|x_u-x_v\|^{d}  \}} \Big)\le g_n^{u,v}\big(x_u,x_v,(w_i)_{i\in [n]}\big) \le   \overline c \cdot \ind_{\{\overline c_1 w_u w_v \geq \{ n\|x_u-x_v\|^{d} \}}.
\end{equation}
	\end{assumption}
	
	Note that we allow the weight distribution $W^{(n)}$ to depend on $n$ in Definition~\ref{def:GIRG}. This is not generality for its own sake -- it will later allow us to extend the results to hyperbolic random graphs.
	In this paper, we make the following assumption on $(W^{(n)})_{n\ge 0}$; it is milder than assuming, for instance, i.i.d.\ power-law weights, and it is satisfied by HRG~\cite{julia-girg}.
	\begin{assumption}\label{assu:mild-vertex}	There exists $\tau> 1$ such that the following holds.
 Write $\ell^{(n)}(x):=\Pv(W^{(n)}\ge x) /x^{-(\tau-1)}$. Then there exists a sequence $(M_n\colon n\ge 1)$ of positive reals such that $\Pr(W^{(n)} > M_n) = o(1/n)$ as $n\to \infty$, and functions $\llow,\lhigh\colon\R^+\to\R^+$ varying slowly at infinity, such that $\llow(x) \le \ell^{(n)}(x) \le \lhigh(x)$ for all $n$ and all $x \in [1,M_n]$. 
\end{assumption}
In words, we assume that $W^{(n)}$ does not vary too severely from a power-law with an exponent $\tau> 1$ that does not depend on $n$. (These variations are captured by the functions $\ell^{(n)}$.)
		Note that this assumption trivially holds if $W^{(n)}\equiv W$ does not depend on $n$ and $W$ follows a power law. 
		
		In order to formulate the main results properly, we must first be sure that a linear-sized \textit{giant} component exists with high probability in the models we study under Assumptions~\ref{assu:mild-connect} and~\ref{assu:mild-vertex}, assuming $2 < \tau < 3$. An analogous result is already known for HRG~\cite{BodFouMul15, FouMul18} (see Definition~\ref{def:HRG}), which satisfies the stronger assumption~\eqref{eq:GIRG-orig} on its connection probabilities, and for SFP the question does not arise as the model has only a single component. For GIRG, the result was proved in~\cite{BriKeuLen19} under \eqref{eq:GIRG-orig},~\eqref{eq:GIRG-orig-threshold}, and Assumption~\ref{assu:mild-vertex}; recall that \eqref{eq:GIRG-orig},~\eqref{eq:GIRG-orig-threshold} are stronger than Assumption~\ref{assu:mild-connect}.
		
		\begin{theorem} \label{thm:giant}
		Consider $\mathrm{GIRG}_W(n)$ as in Definition \ref{def:GIRG}, with edge-connectivity functions $g_n^{u,v}$ satisfying Assumption \ref{assu:mild-connect} and weight distribution satisfying Assumption \ref{assu:mild-vertex} with $2 < \tau <3$. Then whp there exists a unique linear-sized giant component $\CC_{\max}$ in $\mathrm{GIRG}_W(n)$. 
		\end{theorem}
		We prove Theorem \ref{thm:giant} in Section \ref{s:giant}. The proof is interesting in its own right, since it is novel and reveals the hierarchical structure of the graph. We sketch the core idea here: we call an arbitrary vertex $u$ \textit{successful} if it is connected by a path to a `reasonably' high-weight vertex $\wit u$ that is nearby (within a box that we specify). We show that a vertex is successful with strictly positive probability. We then show that starting from $\tilde{u}$, whp we can construct a path of vertices of increasing weight leading up to the highest-weight vertices in the graph. The graph induced by these highest-weight vertices is dominated below by an Erd\H{o}s-R\'{e}nyi random graph, as the minima in~\eqref{GIRGgeneral} and~\eqref{GIRGgeneral-threshold} remove all position-dependent terms from their respective lower bounds. It follows that all successful vertices lie in the same component.

		We then use a boxing structure: we call a box \textit{successful} if it contains linearly many successful vertices, and  spatial independence ensures that the number of boxes that are successful is linear, hence establishing the presence of the giant.  
		
		 A similar hierarchy was described in a top-to-bottom fashion for scale-free percolation in \cite{HeyHulJor17}. However, in scale free percolation, the connection probability gets arbitrarily close to one when $w_u w_v/\|u-v\|\gg 1$, and hence almost every hub is adjacent to every other hub that satisfies $w_u w_v/\|u-v\|\gg 1$. 

		 This fact is crucial for~\cite{HeyHulJor17}, and it fails when the weaker lower bound in \eqref{GIRGgeneral} is applied. Indeed, when two pre-selected hubs are no longer adjacent whp, a top-to-bottom hierarchy is hard to describe as we can say very little about an individual hub; this motivates the bottom-to-top approach used in the proof of Theorem~\ref{thm:giant}.
		 
	With Theorem~\ref{thm:giant} in place, we arrive at the first result on finite-sized models:
		\begin{theorem}[Cost-distances in GIRG]\label{thm:GIRG1}
		Consider $\text{GIRG}_{W,L}(n)$, satisfying Assumptions \ref{assu:mild-connect} and \ref{assu:mild-vertex} for some $\tau\in(2,3)$, and let $f$ be a polynomial as in~\eqref{eq:pol}. Let $v_n^1, v_n^2$  be two typical vertices in the giant component $\CC_{\max}$. Let $\beta^+$ and $\beta^-$ be defined as in Theorem~\ref{thm:mu-classification}.
				\begin{enumerate}
		\item Suppose the edge weight distribution $F_L$ satisfies $\beta^+ < (3-\tau)/\deg(f)$.
		Then
		$\big(d_{f,L}(v_n^1, v_n^2)\big)_{n\ge 1}$
	is a tight sequence of random variables. 
	\item Suppose the edge weight distribution $F_L$ satisfies $\beta^- > (3-\tau)/\deg(f)$. 
	Then 
	\be\label{eq:tight-2} 
		(d_{f,L}\left(v_n^1, v_n^2\right))_{n\ge 1} \toinp \infty.
		\ee
		\end{enumerate}
	\end{theorem}
The meaning of Part 1 of Theorem \ref{thm:GIRG1}  is that when $F_L(t)$ is sufficiently steep close to $0$, the typical cost-distance does not grow with the network size. This implies the following: For every $\ve>0$ and $p<1$ one can find a \emph{constant} $K_{\ve,p}$, depending on $\ve$ and $p$ but not on $n$, such that for sufficiently large $n$, with probability at least $p$ all but an $\ve$-proportion of vertices within the giant component $\CC_{\max}$ are within cost-distance $K_{\ve,p}$ from the (uniformly chosen) source vertex. This is the analogue of explosion in finite models. Part 2  tells us that when $F_L(t)$ is flatter at the origin, then the typical cost-distance does grow with the network size; this is the analogue of the conservative case. We remark that for power-law exponents $\tau \geq 3$, a giant (linear size) component need not exist when the edge-density is low, so we cannot hope for an analogue of Theorem~\ref{thm:GIRG1} for $\tau \geq 3$.

With some extra assumptions, we obtain a finer result in the explosive case: distributional convergence of the typical cost-distance.  Since the model is not projective (i.e., the GIRG model with $n+1$ vertices is not an extension of the model with $n$ vertices), this is best possible --- one cannot hope for e.g.\ almost sure convergence. Even for distributional convergence, one needs quite a few extra assumptions: we need that the edge-connection probabilities $g_n^{u,v}$ converge uniformly in $u,v$ to some limiting function $h$ satisfying Assumption \ref{assu:h} when the Euclidean distance between the two vertices $x_u, x_v$ under consideration is of order $n^{1/d}$. Moreover, the distributions of the sequence of vertex-weights $W^{(n)}$ must converge to a limiting distribution. The exact assumptions are rather technical since we want them to be general enough to include hyperbolic random graphs. Since the transfer from infinite models to finite models closely follows the proof in~\cite{julia-girg}, we omit the full details of the assumptions on the convergence of $W^{(n)}$ and $g_n^{u,v}$ and refer the reader to~\cite[Assumptions 2.4 and 2.5]{julia-girg}.  We  provide  the proof for the following theorem in Section~\ref{proof-finite}.

	\begin{theorem}[Cost-distances in GIRG, explosive case]\label{thm:GIRG-fine}
		Let $n \in \Z^+$, let $f$ be a polynomial as in~\eqref{eq:pol}, and consider $\text{GIRG}_{W,L}(n)$, satisfying~\cite[Assumptions 2.4 and 2.5]{julia-girg} with some $\tau\in(2,3)$. Let $v_n^1, v_n^2$  be two typical vertices  in the giant component $\CC_{\max}$. Suppose the edge weight distribution $F_L$ satisfies $\beta^+ < (3-\tau)/\deg(f)$, where $\beta^+$ is defined as in Theorem~\ref{thm:mu-classification}. 
		
		Let $\Ione$ be the corresponding infinite model, with connection probability function and weight distribution given by the limiting probability function and limiting weight distribution of $\text{GIRG}_{W,L}(n)$, respectively. Let $Y^{(1)}$ and $Y^{(2)}$ be two i.i.d.\ copies of the explosion time $Y_f^{\mathrm{I}}(0)$ of $\Ione$ (see Definition~\ref{def:explosiontime}), conditioned on $Y_f^{\mathrm{I}}(0) < \infty$. Then
		\be\label{eq:convergence}
		d_{f,L}(v_n^1, v_n^2)\toindis Y^{(1)}+Y^{(2)}.
		\ee
	\end{theorem} 
	
\subsection*{Hyperbolic Random Graphs} As mentioned before, the $\text{GIRG}_{W,L}(n)$ model contains Hyperbolic Random Graphs (HRGs) as a special case. We first summarise some related literature. The model originates from 
a hidden variable model, introduced by Bogun\'a and Pastor-Satorras in \cite{BogPas03}. Inhomogeneous random graphs were studied slightly afterwards by Bollob\'as, Janson and Riordan in \cite{BolSvaRio07}. Space was then introduced with latent variables by Bogun\'a in \cite{BogPasDia04}, and  the pre-hyperbolic latent space paper by Serrano, Krioukov and Bogun\'a  in \cite{SerKriBog08}.  The embedding into hyperbolic space first appeared in  \cite{BogPapaKriou10, krioukov2010hyperbolic}. This is when the model became popular, and gave rise to a sequence of papers, studying e.g.: degrees and clustering in \cite{GugPanPet12}, the size of the giant component in \cite{BodFouMul15, FouMul18}, the clustering coefficient and bootstrap percolation in \cite{CanFou14, CanFou16}, competing First Passage Percolation in \cite{CanSta18}, typical distances in the scale-free regime in~\cite{AbdBodFou16}, and the spectral gap in~\cite{KiwMit18}. 
 
We now give the model's formal definition. Let us denote by $(\phi_v,r_v)$ the (hyperbolic) angle and radius of a vertex $v$ within a disk of radius $R$. Then the hyperbolic distance $d_H^{(n)}(u,v)$ between two points $(\phi_u, r_u ), (\phi_v, r_v)$ is defined by the equation
	\be \cosh(d_H^{(n)}(u,v)) := \cosh(r_u) \cosh(r_v) - \sinh(r_u) \sinh(r_v) \cos(\phi_u - \phi_v).\ee
	\begin{definition}[Hyperbolic Random Graphs]\label{def:HRG}
	For parameters $C_H, \al_H, T_H>0$, let us set $R_n$ $=2\log n +C_H$, and sample $n$ vertices independently from a circle of radius $R_n$ so that for each $v\in[n]$, $\phi_v$ is uniform in $[0, 2\pi]$, and $r_v\in[0,R_n]$ follows a density $f_n(r):=\al_H \sinh(\al_H r)/(\cosh(\al_H R_n)-1) $, \emph{independently} of $\phi_v$. In \emph{threshold} hyperbolic random graphs, two vertices $u$ and $v$ are connected whenever $d_H^{(n)}(u,v)\le R_n$, while in a parametrised version \cite[Section VI]{krioukov2010hyperbolic} they are connected independently of everything else, with probability 
	\be \label{eq:phdh} p_H^{(n)}(d_H^{(n)}(u,v)):=\big( 1+ \exp\{ (d_H^{(n)}(u,v)-R_n)/ 2T_H\}\big)^{-1}.\ee 
	We denote the resulting random graphs by $\mathrm{HG}_{\al_H, C_H, T_H}(n)$ when \eqref{eq:phdh} applies and $\mathrm{HG}_{\al_H, C_H}(n)$ when the threshold $d_H^{(n)}(u,v)\le R_n$ is applied. 
\end{definition}
	The connection to GIRGs is derived as follows: set $d:=1, \mathcal X_1:=[-1/2, 1/2]$. For each vertex $v=(\phi_v,r_v)$, let
	\be\label{mapping} x_v:=(\phi_v-\pi)/(2\pi),\quad W_v^{(n)}:=\exp\{(R_n-r_v)/2\}. \ee
	In \cite[Sections 8, 9]{julia-girg} the authors show that with this transformation, Hyperbolic Random Graphs become GIRGs satisfying~\cite[Assumptions 2.4 and 2.5]{julia-girg} with the following limiting parameters.
	The limiting weight distribution $W\ge1$ is described by its tail,
	\be\label{eq:w-HRG} \Pv(W\ge x)=x^{-2\al_H};\ee
	that is, $\tau=2\al_H+1$. In the parametrised case~\eqref{eq:phdh}, the limiting connection probability function $h$ is given by
	\be\label{eq:h-HRG} h_{\mathrm{H}}(\Delta, w_u, w_v)= \Big(1+(\e^{C_H/2}|\Delta|\pi /(w_u w_v) )^{1/T_H} \Big)^{-1},\ee 
	implying that in this case $\al=1/T_H$. In the threshold case, $h$ is given by
	\be\label{eq:h-THRG}  h_{\mathrm{T}}(\Delta, w_u, w_v) =\ind\{ |\Delta| \le \e^{-C_H/2} w_uw_v/\pi\},\ee
which corresponds to $\alpha = \infty$. Theorems~\ref{thm:GIRG1} and~\ref{thm:GIRG-fine} therefore carry over to HRGs.
\begin{corollary}\label{cor:HRG}
Consider $\mathrm{HG}_{\al_H, C_H, T_H}(n)$ or $\mathrm{HG}_{\al_H, C_H}(n)$ with $\al_H\in (1/2, 1)$, and equip every existing edge with $L_e$, an i.i.d.\ copy of a random variable $L\ge 0$. Let the penalty function $f$ be a polynomial as in \eqref{eq:pol}, i.e., the cost of edge $e=(u,v)$ is 
\[L_e f(W_u^{(n)}, W_v^{(n)}) = L_e f(\exp\big[(R_n-r_u)/2\big],\exp\big[(R_n-r_v)/2\big]).\] Then, Theorems \ref{thm:GIRG1} and \ref{thm:GIRG-fine} stay valid with $\tau := 2\al_H+1$.

 More precisely, for every $n \geq 1$, let $v_n^1$, $v_n^2$ be typical vertices in the giant component of $\mathrm{HG}_{\al_H, C_H, T_H}(n)$ or that of $\mathrm{HG}_{\al_H, C_H}(n)$. Then their cost distances satisfy $d_{f,L}(v_n^1, v_n^2) \to \infty$ almost surely if $\beta^->(2-2\al_H)/\deg(f)$. 
On the other hand,  if $ \beta^+  <(2-2\al_H)/\deg(f)$, 
then $(d_{f,L}(v_n^1, v_n^2))_{n\ge 1}$ converges in distribution to the sum of two i.i.d.\ copies of the explosion time of the origin in a one-dimensional $\Ione$ with weights from distribution \eqref{eq:w-HRG} and $h_{\mathrm I}=h_\mathrm{H}$ from \eqref{eq:h-HRG} for  $\mathrm{HG}_{\alpha_H, C_H,T_H}(n)$, or $h_{\mathrm I}=h_\mathrm{T}$ in \eqref{eq:h-THRG} for $\mathrm{HG}_{\alpha_H, C_H}(n)$, respectively.
\end{corollary}

	\section{Explosive  greedy paths}\label{sec:expl}
	In this section we prove Theorems \ref{thm:mu-classification}(ii) and~\ref{thm:poly-classification}(ii).
	We will show that explosion occurs by constructing an infinite path with finite total cost. As mentioned below Definition \ref{def:explosiontime}, the existence of  such a path implies  that $\sigma_f^{\mathrm I}(v,k)$ stays bounded, implying that $Y_f^{\mathrm I}(v)$ is finite. 
	
	We consider expanding boxes (i.e.\ balls in the $L_\infty$ metric) around the origin, such that the $k$th box has doubly-exponential volume $\e^{MDC^k}$ for some suitably chosen $C,D>1$ and arbitrary $M>0$. We then partition the $k$th annulus into roughly $\e^{M(D-1)C^k}$ disjoint sub-boxes, each of volume $\e^{MC^k}$. In each sub-box we find the vertex of maximum weight, which we call the \emph{leader} of the sub-box. We construct a path to infinity greedily as follows: suppose we have exposed the $k$th annulus and reached some leader vertex $v_k$ therein. Then expose the contents of the $(k+1)$st annulus, and choose the edge $v_kv_{k+1}$ from $v_k$ to a leader vertex of some sub-box of the $(k+1)$st annulus such that the assigned $L_e$ is minimal. To prove explosion, it suffices to show that this path has finite cost almost surely. 
	 
	 We start by describing the sequence of expanding  boxes.  For constants $C,D>1$, to be defined shortly, and an arbitrary parameter $M$, let us  define a \emph{boxing system} centered at $u\in \R^d$, by defining for $k\geq0$,
	\be\label{eq:boxing}
	\ba
	\text{Box}_k(u)&:=\left\{x\in\mathbb{R}^d\;:\;\|x-u\|_\infty\leq \e^{M D C^k/d}/2\right\},\\
	\Gamma_k(u)&:=\text{Box}_k\left(u\right)\backslash \text{Box}_{k-1}\left(u\right)\mbox{ for }k\geq1,\qquad \Gamma_0\left(u\right):=\text{Box}_0(u),\\
	\ea
	\ee
	We `pack' each annulus $\Gamma_k(u)$ with as many disjoint \emph{sub-boxes}
	\[
		\mathrm{SB}_{k,i}(u) = \left\{x\in\mathbb{R}^d \;:\; \|x-z_i\|_\infty\leq \e^{MC^k/d}/2\right\}
	\]
	of volume $\e^{MC^k}$ as possible; here the $z_i$'s are appropriately-chosen points in $\text{Box}_k(u)$. The exact choice of $z_i$'s will not matter to us, but note that in general the side length of a sub-box will not divide the side length of an annulus so there will be some volume left over. Let $b_k$ denote the number of sub-boxes in $\Gamma_k(u)$, and  order the sub-boxes  arbitrarily from $1$ to $b_k$ within each annulus. The ratio of the volumes of $\text{Box}_k(u)$ and $\SB_{k,i}(u)$ is $\e^{M(D-1)C^k}$. Hence, for sufficiently large $M$,
	\be\label{eq:boundbk}  
	\e^{M(D-1)C^k}/2\leq b_k\leq \e^{M(D-1)C^k} \quad \text{for } \Ilambda,
	\ee
	Within each sub-box $\SB_{k,i}(u)$ which contains at least one vertex, we define the \emph{leader vertex} $c_{k,i}$ to be the vertex with the highest weight, i.e., $\cki:=\arg\max_{v\in \mathrm{SB}_{k,i}\left(u\right)}\left\{W_v\right\}$. We say that $\SB_{k,i}(u)$ is \emph{$\delta$-good} if it has a leader vertex and this leader vertex has weight
	\be\label{eq:delta-good} W_{\cki}\in (\e^{(1-\delta)MC^k/(\tau-1)},\e^{(1+\delta)MC^k/(\tau-1)}].\ee We will also say that the leader vertex itself is \emph{$\delta$-good}.
	We will see in Lemma~\ref{lem:bipartite} that for suitable choices of $C$, $D$ and $\delta$, with high probability there are many $\delta$-good sub-boxes in each $\Gamma_k(u)$. Moreover, again with high probability, each $\delta$-good leader vertex $\cki$ in $\Gamma_k(u)$ is connected to many $\delta$-good leader vertices in $\Gamma_{k+1}(u)$. Since edge weights are chosen independently, it will follow that with high probability there is a low-cost edge from $\cki$ to a $\delta$-good leader in $\Gamma_{k+1}(u)$, and we will use this to greedily construct an infinite path with finite cost-distance. The key to the proof of Lemma~\ref{lem:bipartite} is that the weights $w_1$ and $w_2$ of two $\delta$-good leader vertices are so high relative to their Euclidean distance (which is bounded above by the diameter of $\text{Box}_{k+1}(u)$) that their connection probability is bounded below by $l_{c_2,\gamma}(w_1)l_{c_2,\gamma}(w_2)$, by Assumption~\ref{assu:h}. A similar boxing scheme was used in~\cite{julia-girg}, and we have adapted Lemma~\ref{lem:bipartite} from Lemma~6.3 of that paper.

Even though it will only become relevant later, in Section \ref{s:giant}, we note here that the same boxing method remains valid when we consider $\GIRG$ instead of $\Ilambda$. To keep the box sizes the same in the two models, we blow up the original $\GIRG$ model as follows.
		\begin{definition}[Blown-up-GIRG]\label{def:bgirg}
	Consider a realisation of a $\GIRG$ from Definition \ref{def:GIRG}, with vertices  $(x_v)_{v\in[n]}$. Map each vertex-location to $\wit x_v:=n^{1/d}x_v$. We denote the resulting model by $\BGIRG$. Let $\CV_B(n):=[n]$, and the edge set by $\CE_B(n):=\{(v,w)\in[n]^2: v\leftrightarrow w \in \GIRG$, and its underlying state space by $\calX_d(n):=[-n^{1/d}/2, n^{1/d}/2]^d$.
	\end{definition}
	Note that $\BGIRG$ is the \emph{same graph} as $\GIRG$. The two models differ only in the location of points, observe that by blowing the model up, the density of points in $\BGIRG$ is constant (namely $1$), while the density of points in $\GIRG$ is $n$. The additional notation $\CV_B(n) = [n]$ seems superfluous, but it allows us to use the slightly abusive notation $\CV_B(n) \cap A := \{v\in[n] \mid \tilde x_v\in A\}$ for $A\subseteq \R^d$ without ambiguity. 
	We can actually realise the edges of $\GIRG, \BGIRG$ by working with the new locations. 
	This is convenient since by Assumption \ref{assu:mild-connect}, using the blown-up locations instead of the original ones, two vertices with blown-up locations $\wit x_u, \wit x_v$ are connected with probability at least 
	\be \label{eq:blown-h}  g_n^{u,v}(\wit x_u, \wit x_v, (w_i)_{i\le n} )\ge 
	\begin{cases}
		\un c \cdot \bigg( l(w_u)l(w_v)\wedge \Big( \frac{w_u w_v}{ \|\wit x_u-\wit x_v\|^{d}}\Big)^\alpha\bigg) & \mbox{ if }\alpha<\infty,\\
		\un c \cdot \Big( l(w_u)l(w_v)\wedge \ind_{\{\un c_1 w_uw_v \ge \|x_u-x_v\|^d\}}\Big) & \mbox{ otherwise,}
	\end{cases}
	\ee
	where we wrote $l(w):=l_{c_2,\gamma}(w)$ from \eqref{eq:ellw-111}.
		Observe the factor of $n$ disappears from the denominator, and the bound becomes the same as the bound in Assumption \ref{assu:h} for $\Ilambda$. Since $\GIRG$ and $\BGIRG$ are equivalent, from now on we work with the blown-up model instead of the original $\GIRG$, and use \eqref{eq:blown-h} instead of the lower bound in Assumption \ref{assu:h}. We construct boxing systems for $\BGIRG$ in precisely the same way as for IGIRG, except that we require all sub-boxes to fix within $\calX_d(n)$. 
	For this reason, we define 
		\be\label{eq:kstar} k^\star=k^\star(n,M):=\max\{k\in\mathbb{N} \;|\; \e^{MDC^k/d}\le n^{1/d}\},\ee
to be the largest $k$ such that $\text{Box}_{k}(0)$ in \eqref{eq:boxing} fits within $\calX_d(n)$, and thus, for any $u \in \calX_d(n)$,  at least\footnote{In case $u$ is in the corner of a sub-box.} a $2^{-d}$ fraction of $\text{Box}_{k^\star}(u)$ fits into $\calX_d(n)$.
	Observe that 
\be\label{eq:boundbk-girg} 	\e^{M(D-1)C^k}/2^{d+1} \leq b_k\leq \e^{M(D-1)C^k} \quad \text{for } \BGIRG, \ee
where the factor $2^{-d}$ in the lower bound comes from the fact that not all sub-boxes might be part of $\mathcal X_d(n)$, but at least a $1/2^{-d}$ proportion of them are, if their centers are suitably chosen.

	We use the following standard Chernoff bound.
	\begin{lemma}{{\cite[Corollary 2.3]{JLR}}}\label{lem:chernoff}
		Let $X$ be a binomial r.v.\ with mean $\mu$. Then for all $0 < \eps \le 3/2$, 
		\[\Pr(|X-\mu|\ge\eps\mu)\le 2\e^{-\eps^2\mu/3}\,.\]
	\end{lemma}
	The next lemma is crucial to show explosion and also relevant to showing the existence of the unique giant component in the finite case. As mentioned above, it shows that every $\delta$-good leader has many $\delta$-good leader neighbors in the next annulus. This guarantees the existence of infinite paths, and enables the greedy construction of low-cost paths.  Recall that we denote by $c_{k,i}$ the vertex with the highest weight in sub-box $\SB_{k,i}(u)$, for $i\le b_k$. 
	\begin{lemma}[Weights and subgraph of centers]\label{lem:bipartite}
		Consider $\Ilambda$ with parameters $d\geq1$, $\tau\in (1,3)$, $\alpha \in(0,\infty]$, and $\lambda > 0$. Let $C,D > 1$ and $0 < \delta < 1$ satisfy
		\begin{equation}\label{eq:cd}
			\frac{1-\delta}{\tau-1}(1+C)-DC>0.
		\end{equation}
		For every $\eps >0$ there exists $M_0 >0$ such that the following holds for all $M \ge M_0$. 
		Let $u\in\mathbb{R}^d$, and consider the boxing system centered at $u$ with parameters $C$, $D$ and $M$ as described in \eqref{eq:boxing}. Define $N_{j}(\cki)$ to be the number of $\delta$-good leader vertices in $\Gamma_{j}(u)$ that are adjacent to $\cki$,  and define the events
		\be\ba\label{boundsweight}
		F_k^{(1)}&:=\left\{\big|\{i \in [b_k] \colon \SB_{k,i}\textnormal{ is $\delta$-good}\}\big| \ge b_k/2\right\},\\
		F_k^{(2)} := F_k^{(2)}(\eps)&:=\left\{ \forall i \in [b_k]\textnormal{ such that $\SB_{k,i}$ is $\delta$-good}: N_{k+1}(\cki) \ge \e^{(1-\eps) M C^{k+1} (D-1) } \right\}.
		\ea\ee
		Then
		\be\label{boundsweight-error}
		\Pv\big(\neg \cap_{k\ge 0} (F_k^{(1)} \cap F_k^{(2)})\big) \le  3\exp\Big({-} \lambda\e^{M((D-1)\wedge 1)(1-\ve)}2^{-d}/75\Big)
 =:p_M. \ee
		
The same result holds for $\SFPWL$, taking $\lambda=1$. It also holds for $\BGIRG$ when Assumptions~\ref{assu:mild-connect} and~\ref{assu:mild-vertex} are satisfied and $\lambda = 1$, replacing the intersection $\cap_{k\ge 0}$ on the lhs of \eqref{boundsweight-error} by $\cap_{k\le k^\star(M,n)}$ and requiring that $n$ is sufficiently large and $u \in \calX_d(n)$. (Here we take $b_k$ to be the number of sub-boxes contained in $\Gamma_k(u) \cap \calX_d(n)$ rather than in $\Gamma_k(u)$, as discussed above.)
	\end{lemma}

	We remark that we will not use the assumption $\tau < 3$ explicitly in the proof of Lemma~\ref{lem:bipartite}. However, if $\tau \ge 3$ then there is no choice of $C,D>1$ and $0<\delta<1$ satisfying~\eqref{eq:cd}. For $\tau < 3$ there is always such a choice, as we shall see in Claim~\ref{cl:parameters} below.
	
\begin{proof}
We first prove the result for $\BGIRG$, then discuss how to adjust the proof for $\SFPWL$ and $\Ilambda$. For this reason, we will keep $\lambda$ explicit in the calculation, even though for $\BGIRG$ we always have $\lambda =1$. 
We first bound the probability that a given sub-box is $\delta$-good from below. Recall that $n$ vertices are uniformly distributed in $\mathcal X_d(n)$ (which has volume $n$),   
that $\SB_{k,i}$ has volume $\e^{MC^k}$, and that vertex weights follow an approximate power law as set out in Assumption~\ref{assu:mild-vertex}. We condition throughout on the event that every vertex has weight at most $M_n$; by Assumption~\ref{assu:mild-vertex}, this event occurs whp and implies that the weight of every vertex independently follows a distribution $\Pr(W \ge x) = \ell^{(n)}(x)x^{-(\tau-1)}$, where $\llow(x) \le \ell^{(n)}(x) \le \lhigh(x)$ for some functions $\llow$ and $\lhigh$ which vary slowly at infinity. We require $n$ to be large enough that $\text{Box}_1(0) \subseteq [-n^{1/d}/2,n^{1/d}/2]^d$.

First we exclude the event that some sub-box has too many or too few vertices, then we study the maximal weight of vertices in each sub-box. The number of vertices $V_{k,i}$ in each sub-box is binomial with parameters $n$ and $\mathrm{Vol}(\mathrm{SB_{k,i}})/n$, so it has mean $\lambda\e^{MC^k/2}$.
Hence, by the Chernoff bound of Lemma \ref{lem:chernoff},
\be\ba \label{eq:e111}\Pv(\neg\CE^1_{k}):=\Pv(\exists i\le b_k: V_{k,i} \notin [\la\e^{MC^k}/2, 2\la\e^{MC^k}]) &\le 2b_k \exp(- \la \e^{MC^k}/12 )\\
&\le \exp(-\la \e^{MC^k}/24)\ea\ee
for all sufficiently large $M$, since the second factor is doubly exponentially small in $MC^k$ while $2b_k$ is only exponential in $MC^k$ by \eqref{eq:boundbk-girg}. 
For any tuple $(n_{k,i})_{i=1}^{b_k}$ such that $\lambda\e^{MC^k}/2 \le n_{k,i} \le 2\lambda\e^{MC^k}$, let $\CE_{n_{k,1}, \dots, n_{k,b_k}} \subset \CE^1_k$ be the event that $\{\forall i \le b_k:\ V_{k,i}=n_{k,i} \}$. 
Then, for any $\CE_{n_{k,1}, \dots, n_{k,b_k}}\subseteq \CE^1_k$, we have
\begin{align}\nonumber
\Pv\Big( \max_{v \in \SB_{k,i} \cap \mathcal{V}_B(n)} W_v^{(n)} \le  y \mid \CE_{n_{k,1}, \dots, n_{k,b_k}} \Big) 
&=\big(1- \Pv(W^{(n)} >  y)\big)^{ n_{k,i}}\le \big(1- \llow(y)y^{-(\tau-1)}\big)^{ \lambda\e^{MC^k}/2}\\
\label{eqn:bip-weights-a}
&\le \exp\Big(-\llow(y)y^{-(\tau-1)} \cdot \lambda \e^{MC^k}/2 \Big).
\end{align}
Recall that since $\llow$ varies slowly at infinity, Potter's bound implies that for all $\eta>0$, we have $\llow(y) = o(y^{\eta})$ and $\llow(y) = \omega(y^{-\eta})$ as $y\to \infty$. Thus when $M$ is sufficiently large, taking $y$ in~\eqref{eqn:bip-weights-a} to be the lower bound in the definition of $\delta$-goodness in \eqref{eq:delta-good}, we obtain
\begin{align}\nonumber
\Pv\Big( \max_{v \in \SB_{k,i} \cap \mathcal{V}_B(n)} W_v^{(n)} \le  \e^{\frac{1-\delta}{\tau-1}MC^k} \mid \CE_{n_{k,1}, \dots, n_{k,b_k}} \Big) 
&\le \exp\Big({-}\llow(\e^{\frac{1-\delta}{\tau-1}MC^k}) \e^{{-}(1-\delta)MC^k + MC^k}\lambda/4 \Big)\\\label{eqn:bip-weight-up}
&\le \exp\Big({-} \la\e^{\delta MC^k/2} \Big)\,,
\end{align}
where we have applied Potter's bound to obtain the last line, and absorbed the factor of $4$ in the same step. 
We now bound the maximum weight above. By a union bound, for all $y>0$,
\begin{equation*}
\ba \Pv\Big( \max_{v \in \SB_{k,i} \cap \mathcal{V}_B(n)} W_v^{(n)} > y \mid \CE_{n_{k,1}, \dots, n_{k,b_k}}\Big) 
&\le \sum_{v \in \SB_{k,i} \cap \CV_B(n)} \Pv\big(W_v^{(n)} > y \mid \CE_{n_{k,1}, \dots, n_{k,b_k}}\big)\\
&\le \lhigh(y) y^{-(\tau-1)}2\la \e^{MC^k}.\ea
\end{equation*}
Since $\ell$ varies slowly at infinity, when $M$ is sufficiently large, taking $y$ to be the upper bound in the definition of $\delta$-goodness and applying Potter's bound yields
\begin{equation}\label{eqn:bip-weight-low}\ba
\Pv\Big( \max_{v \in \SB_{k,i} \cap \mathcal{V}_B(n)} W_v^{(n)} > \e^{\frac{1+\delta}{\tau-1}MC^k}\mid \CE_{n_{k,1}, \dots, n_{k,b_k}} \Big)& 
\le \lhigh( \e^{\frac{1+\delta}{\tau-1}MC^k})  \e^{-(1+\delta)MC^k} 2\la \e^{MC^k} \\
&\le 2\la \e^{-\delta MC^k/2}.\ea
\end{equation}
Combining~\eqref{eqn:bip-weight-up} with the much weaker bound~\eqref{eqn:bip-weight-low}, when $M$ is sufficiently large we see that in $\BGIRG$ (where $\la=1$),
\begin{equation}\label{eqn:bip-good}
\Pv\big(\SB_{k,i}\mbox{ is not $\delta$-good} \mid \CE_{n_{k,1}, \dots, n_{k,b_k}} \big) \le 3\la \e^{-\delta MC^k/2}
\end{equation}
holds uniformly over all $\CE_{n_{k,1}, \dots, n_{k,b_k}}\subseteq \CE_k^1$.
Note that for all $k$ and $i$, the event of $\SB_{k,i}$ being $\delta$-good depends only on the number of vertices and their (i.i.d.) weights in $\SB_{k,i} \cap \CV_B(n)$, So, conditioned on any of the events $\CE_{n_{k,1}, \dots, n_{k,b_k}}$, these events are mutually independent. Thus~\eqref{eqn:bip-good} implies that, conditioned on any  $\CE_{n_{k,1}, \dots, n_{k,b_k}}\subseteq \CE_k^1$, the number of $\delta$-good sub-boxes in $\Gamma_k(u)$ is dominated below by a binomial random variable with parameters $b_k$ and $1-3\la \e^{-\delta MC^k/2}\ge 3/4$. It follows by a standard Chernoff bound (namely Lemma~\ref{lem:chernoff} with $\eps = 1/3$),~\eqref{eq:e111}, and~\eqref{eq:boundbk-girg} that
\be\ba
	\Pv(\neg F_k^{(1)}) &\le \Pv(\neg \CE_k^1)+  \sum_{\CE_{n_{k,1}, \dots, n_{k,b_k}}\subseteq \CE_k^1 } \Pv(\neg F_k^{(1)}\mid \CE_{n_{k,1}, \dots, n_{k,b_k}}) \Pv(\CE_{n_{k,1}, \dots, n_{k,b_k}}) \\
	& \le  \exp(- \la \e^{MC^k}/24)+  2\e^{-b_k/36}   \le 2\exp\Big({-}\lambda\e^{MC^k((D-1)\wedge1)}2^{-d}/72\Big).
\ea\ee
Hence by a union bound over $k$, when $M$ is sufficiently large we have
\begin{equation}\label{eqn:bip-F1}\ba
	\Pv\Big(\cap_{k\ge 0}F_k^{(1)}\Big) &\ge 1 - 2\sum_{k\ge 0} \exp\Big({-}\la\e^{MC^k((D-1)\wedge1)}2^{-d}/72\Big)\\& \ge 1- 2\exp\Big({-}\la\e^{M((D-1)\wedge1)}2^{-d}/75\Big).\ea
\end{equation}
since the sum is dominated by its first term and decays faster then a geometric sum.

We now turn to the events $F_k^{(2)}$. We condition on $\cap_{k\ge 0}F_k^{(1)}$, and expose $\CV_B(n)$. 
We will first study the connection probability between  any $\delta$-good leader vertex $\cki$ in $\Gamma_k(u)$  to any given $\delta$-good leader vertex in $\Gamma_{k+1}(u)$. (This is where we will use~\eqref{eq:cd}.) We will then dominate the number of such vertices it is adjacent to, $N_{k+1}(\cki)$, below by a binomial variable and use a Chernoff bound to show that $N_{k+1}(\cki)$ is likely to be large. We will then use a union bound to show that $\cap_{k\ge 0}F_k^{(2)}$ is likely to occur, proving the result.

Let $\cki$ be a $\delta$-good leader vertex in $\Gamma_k(u)$, and let $c_{k+1,j}$ be a $\delta$-good leader vertex in $\Gamma_{k+1}(u)$. Write $w_1$ and $w_2$ for the weights of $\cki$ and $c_{k+1,j}$ respectively, and write $\|x_1-x_2\|$ for the Euclidean distance between them. Recall $l(w):=l_{c_2, \gamma}(w)$ from \eqref{eq:ellw-111}. By Assumption~\ref{assu:mild-connect},~\eqref{eq:blown-h} holds, so the probability that $\cki$ and $c_{k+1,j}$ are adjacent is at least 
\begin{equation}\label{eqn:bip-recall}
	\begin{cases}
		\underline c\left(l(w_1)l(w_2)\wedge \big(w_1 w_2/\|x_1-x_2\|^d\big)^\al \right) & \mbox{ if }\alpha < \infty,\\
		\underline c\left(l(w_1)l(w_2)\wedge \ind_{\{\underline c_1 w_1 w_2 \geq \|x_1-x_2\|^d\}} \right) & \mbox{ otherwise.}
	\end{cases}
\end{equation}
Since $\cki$ and $c_{k+1,j}$ both lie in $\mbox{Box}_{k+1}(u)$, we have $\|x_1-x_2\| \le d
\e^{MDC^{k+1}/d}$. Since both vertices are $\delta$-good, it follows that
\begin{align*}
	\frac{w_1w_2}{\|x_1-x_2\|^d} 
	&\ge \frac{1}{d^d}\exp\Big( \frac{1-\delta}{\tau-1}MC^k + \frac{1-\delta}{\tau-1}MC^{k+1} - MDC^{k+1}\Big)\\
	&= \frac{1}{d^d}\exp\Big(MC^k\Big(\frac{1-\delta}{\tau-1}(1+C) - DC\Big)\Big)\,.
\end{align*}
By~\eqref{eq:cd}, the exponent of the rhs is positive, so when $M$ is sufficiently large we have $w_1w_2 \ge \|x_1-x_2\|^d$ and $\un c_1w_1w_2 \ge \|x_1-x_2\|^d$. Thus by~\eqref{eqn:bip-recall}, whatever the value of $\alpha$, whenever $\mathcal W_n := (x_v,W_v^{(n)})_{v \in \CV_B(n)}$ is such that $\cki$ and $c_{k+1,j}$ are $\delta$-good, we have 
\begin{equation}\label{eqn:bip-conn-prob}
\Pv\big(\cki \leftrightarrow c_{k+1,j} \mid \mathcal W_n\big) 
\ge \underline{c}\, l(w_1)l(w_2)\,.
\end{equation}
Recall that $l(w) = \e^{-c_2\log^\gamma w}$. Thus since $\cki$ and $c_{k+1,j}$ are $\delta$-good and $C>1$, using the upper bound on their weights in \eqref{eq:delta-good},
\[
	l(w_1)l(w_2) 
	\ge \exp\bigg({-}c_2\Big(\frac{1+\delta}{\tau-1}MC^k\Big)^\gamma - c_2\Big(\frac{1+\delta}{\tau-1}MC^{k+1} \Big)^\gamma \bigg)
	\ge \exp\bigg({-}2c_2 \Big(\frac{1+\delta}{\tau-1}MC^{k+1}\Big)^\gamma \bigg).
\]
Since $\gamma \in (0,1)$, when $M_0$ is sufficiently large we can upper bound the absolute value of the exponent by $\eps M (D-1) C^{k+1}/4$, where we include the factor $\eps (D-1)/4$ to prepare for the upcoming calculations. Thus it follows from~\eqref{eqn:bip-conn-prob} that whenever $\mathcal W_n := (x_v,W_v^{(n)})_{v \in \CV_B(n)}$ is such that $\cki$ and $c_{k+1,j}$ are $\delta$-good,
\begin{equation}\label{eqn:bip-conn-prob-2}
\Pv\big(\cki \leftrightarrow c_{k+1,j} \mid \mathcal W_n\big) 
\ge \underline{c}\exp\Big({-} \eps M(D-1)C^{k+1}/4\Big)\,.
\end{equation}
Now, conditioned on $\mathcal W_n$ as above, edges between $\delta$-good leaders are present independently. In the following, we fix a $\mathcal W_n$ that implies $\cap_{\ell}F_{\ell}^{(1)}$. Then there are at least $b_{k+1}/2 \ge \exp (M(D-1)C^{k+1})/2^{d+2}$ good leaders in $\Gamma_{k+1}(u)$ by~\eqref{eq:boundbk-girg}. Thus by~\eqref{eqn:bip-conn-prob-2}, $N_{k+1}(\cki)$ is dominated below by a binomial random variable with mean $\underline{c}/2^{d+2} \cdot \exp((1-\eps/4)M(D-1)C^{k+1})$. By a standard Chernoff bound (Lemma~\ref{lem:chernoff} with the $\eps$ of that Lemma chosen as $1/2$), it follows that if $\cki$ is $\delta$-good,
\[\ba
\Pv\Big(N_{k+1}(\cki) \le \frac{\underline{c}}{2^{d+3}}\e^{(1-\eps/4)M(D-1)C^{k+1}} &\,\big|\, \mathcal W_n\Big)\le 2\exp\Big({-}\frac{\underline{c}}{3\cdot 2^{d+4}}\e^{(1-\eps/4)M(D-1)C^{k+1}} \Big)\,.
\ea\]
If $M$ is sufficiently large then the above bound on $N_{k+1}(\cki)$ is stronger than the bound required by~$F_k^{(2)}$. Hence, by a union bound over all $\delta$-good $i \in [b_k]$, it follows that when $M$ is sufficiently large,
\begin{align*}
\Pv\big(\neg F_k^{(2)} \mid \mathcal W_n\big) 
\le 2b_k\exp\Big({-}\frac{\underline{c}}{3\cdot 2^{d+4}}\e^{(1-\eps/4)M(D-1)C^{k+1}} \Big)\le \exp\Big({-}\lambda\e^{(1-\eps/2)M(D-1)C^{k+1}}\Big),
\end{align*}
where we have used the upper bound on $b_k$ from \eqref{eq:boundbk-girg}.
By a union bound over all $k \ge 0$, it follows that when $M$ is sufficiently large,
\begin{align}
\Pv\big(\neg\cap_{k\ge 0}F_k^{(2)} \mid \mathcal W_n\big)
\le \sum_{k=0}^\infty \exp\Big({-}\lambda\e^{(1-\eps/2)M(D-1)C^{k+1}}\Big) \label{eqn:bip-F2}
\le \exp\big({-}\lambda\e^{(1-\eps)M(D-1)}\big)\,.
\end{align}
Recall that this is the case for any $Y$ for which $\cap_{k\ge 0} F_k^{(1)}$ holds. The result therefore follows from~\eqref{eqn:bip-F1},~\eqref{eqn:bip-F2} and a union bound, when $M$ is sufficiently large.
Note that we may assume $\eps <1/4$, so the bound we obtain can be simplified to, with a different $\ve$,
\[ \Pv\big(\neg \cap_{k\ge 0} (F_k^{(1)} \cap F_k^{(2)})\big) \le 3\exp(-\la\e^{M (1-\ve)((D-1)\wedge1)}2^{-d}/75),\]
obtaining \eqref{boundsweight-error}.

It remains only to discuss the necessary changes to the proof for $\SFPWL$ and $\Ilambda$. For $\SFPWL$, the only difference is that the number $n_{k,i}$ of vertices in any given sub-box $\SB_{k,i}$ is now deterministic, with $\e^{MC^k}/2 \le n_{k,i}\le \e^{MC^k}$ as long as $M$ is sufficiently large. Thus there is no need for the event $\CE_k^1$ and~\eqref{eqn:bip-weights-a} and~\eqref{eqn:bip-weight-up},  \eqref{eqn:bip-weight-low} hold without conditioning.

For $\Ilambda$, the derivation of~\eqref{eq:e111} now follows from the version of Lemma~\ref{lem:chernoff} which applies to Poisson variables rather than binomial variables~\cite[Remark~2.6]{JLR} (the statement is otherwise identical). The number of vertices in the sub-boxes $\SB_{k,i}$ are also now independent of each other, so some of the conditioning becomes unnecessary. We also use Assumption~\ref{assu:h} in place of Assumptions~\ref{assu:mild-connect} (with~\eqref{eq:blown-h}) and~\ref{assu:mild-vertex}; Assumption~\ref{assu:h} is always equivalent or stronger, so this does not cause issues. Finally, all our calculations up to and including the final bound in~\eqref{eqn:bip-good} remain valid when $\la \ne 1$.
\end{proof}

In the next lemma we specify the choice of parameters $C, D >1$ and $\delta$ in the boxing scheme in \eqref{eq:boxing}, so that they satisfy \eqref{eq:cd}, and another set of inequalities that will ensure that a constructed greedy path has finite total cost.  The introduction of the extra parameter $s$ will be relevant in the proof of Theorem \ref{thm:GIRG-fine}.

\begin{claim}\label{cl:parameters} Let $\tau\in(1,3)$, let $\mu,\nu,\beta^+ \ge 0$, and suppose $(\mu+\nu)\beta^+<3-\tau$. For all sufficiently small $\delta>0$, the following interval is non-empty:
\be\label{eq:idelta} \mathcal I_{\delta}:=\Big(1+\frac{(\mu+\nu)\beta^+}{\tau-1}\cdot \frac{1+\delta}{(1-\delta)^2},\ \  \frac{2}{\tau-1} \cdot \frac{(1-\delta)}{1+\delta}\Big). \ee
We fix $\delta>0$ with $I_\delta\neq \emptyset$, and choose parameters 
\begin{equation}\label{eq:parameters}
		C: = 1+\delta,\qquad 
		D \in  \mathcal I_\delta.
\end{equation}
Then $D, C>1$ and the following inequalities all hold for all $s\in[0,1]$:
 \begin{align}  \frac{1-\delta}{\tau-1} 2 -C^s D &>0\label{eq:cdx2}\\
 (\mu  +\nu C^{s}) \frac{1+\delta}{\tau-1} - \frac{(D-1)C^{s}(1-\delta)^2}{\beta^+}  &<0.\label{eq:mubeta-3}
\end{align}
 \end{claim}
 Before we come to the proof, observe that ~\eqref{eq:cdx2} for $s=1$ is a strictly stronger condition than~\eqref{eq:cd}, since $C > 1$. Hence the parameters from Claim~\ref{cl:parameters} automatically satisfy the conditions of Lemma~\ref{lem:bipartite}.
 \begin{proof}
 First we fix $\delta>0$ small, set $C:=1+\delta$, and show that the set of solutions for $D>1$ that satisfy \eqref{eq:cdx2} and \eqref{eq:mubeta-3} for all $s\in[0,1]$ is precisely $\mathcal I_\delta$.  Elementary calculation yields that when $D>1$,~\eqref{eq:cdx2} is satisfied if and only if
	\be \label{eq:d-2}D \in \Big(1, \frac{1-\delta}{\tau-1} \frac{2}{C^s}\Big). \ee
	Moreover,~\eqref{eq:mubeta-3} is satisfied if and only if
	\be \label{eq:d-3} D> 1+  \beta^+\frac{(\mu  +\nu C^{s})}{C^{s}} \frac{1+\delta}{(\tau-1)(1-\delta)^2} .\ee
	The upper end of the interval on the rhs of \eqref{eq:d-2} is minimised when $s=1$, giving the upper end of $I_\delta$, while the rhs of \eqref{eq:d-3} is maximised when  $s=0$, giving the lower end of $\mathcal I_\delta$. Thus for all $s \in [0,1]$ and all $D \in I_\delta$, \eqref{eq:cdx2} and \eqref{eq:mubeta-3} are satisfied.
	
	We have yet to show that $I_\delta$ is non-empty for sufficiently small $\delta>0$. 
For this, observe that the lower end of $\mathcal I_\delta$ is monotone decreasing as $\delta\downarrow 0$, while its upper end is monotone increasing, and 
\[ \mathcal I_0=\lim_{\delta\downarrow 0} \mathcal I_\delta = \Big( 1+\frac{(\mu+\nu)\beta^+}{\tau-1},\  \frac{2}{\tau-1} \Big), \] which is non-empty by the assumptions that $(\mu +\nu)\beta^+<3-\tau$ and $\tau \in (1,3)$. Hence, for sufficiently  small $\delta>0$, $\mathcal I_\delta$ will be non-empty.
	\end{proof}
	\cbb
Before the proof of Theorem \ref{thm:mu-classification} (ii), conditioned on the event $\cap_{k\ge 0} (F_k^{(1)}\cap F_k^{(2)})$, we define a \emph{greedy path} emanating from some $\delta$-good leader, and analyse its cost. 
\begin{definition}[Greedy path between $\delta$-good leaders]\label{def:greedy}
Consider a boxing system centered around $u\in \R^d$. Condition on $\cap_{k\ge 0} (F_k^{(1)}\cap F_k^{(2)})$, and let $c_{0}$ be any $\delta$-good leader in $\Gamma_{0}(u)$.  We greedily extend this vertex into an infinite path $\pi^{greedy} = c_{0},c_{1},\dots $ as follows. Suppose we are given $c_{0},\dots ,c_{k}$ for some $k \ge 0$, and that $c_{k}$ is a $\delta$-good leader. Since $F_{k}^{(2)}$ occurs, there is at least one $\delta$-good leader in $\Gamma_{k+1}(u)$ adjacent to $c_{k}$. We then choose $c_{k+1}$ to be (one of) the  $\delta$-good leaders that minimises $L_{(c_{k},c_{k+1})}$. 
\end{definition}	
The next lemma analyses the cost of the greedy path.
\begin{claim}\label{lem:cost-greedy} Let $C,D,\delta,\eps$ and $M_0$ be as in Lemma~\ref{lem:bipartite}, and let $(\zeta_{0}, \zeta_{1}, \dots )$ be any infinite sequence with positive entries. Then for every $M\geq M_0$, with the boxing system from~\eqref{eq:boxing}, the cost of the greedy path starting in a leader $c_{0}$ of $\Gamma_{0}$ wrt the penalty function $f(w_u, w_v)=w_u^\mu w_v^\nu$ is
\be\label{eq:cost-greedy} |\pi^{greedy}|_{f,L}\le  \sum_{k=0}^\infty  \Big(\e^{MC^k\tfrac{1+\delta}{\tau-1}}\Big)^{\mu} \Big(\e^{MC^{k+1}\tfrac{1+\delta}{\tau-1}}\Big)^{\nu} F_L^{(-1)}(\zeta_{k} \e^{-(1-\eps)MC^{k+1}(D-1)})\ee
 with probability at least $1-\sum_{k\ge 0} \e^{-\zeta_{k}}$ conditioned on $\CV_\la$, $\{W_v \colon v \in \CV_\la\}$, $\cap_{k\ge 0}(F_k^{(1)} \cap F_k^{(2)})$, and the (unweighted) edge set of the graph.
\end{claim}
\begin{proof}
Recall that the cost of an edge $(u,v)$ is $C_{(u,v)}=f(W_u, W_v) L_{(u,v)}=W_u^\mu W_v^\nu L_{(u,v)}$. 
To estimate the cost of the greedy path, we recall the upper bound on the weights of the $\delta$-good leaders $c_{k}$ in \eqref{eq:delta-good}, and the lower bound on the number $N_{k+1}(c_{k})$ of $\delta$-good leader-neighbors they have from $F_{k}^{(2)}$ (see \eqref{boundsweight}). The total cost of $\pi^{greedy}$ is therefore bounded above by
	\be\label{explosivesum}
	|\pi^{greedy}|_{\mu,L} \le  \sum_{k=0}^\infty  \Big(\e^{MC^{k}\tfrac{1+\delta}{\tau-1}}\Big)^{\mu} \Big(\e^{MC^{k+1}\tfrac{1+\delta}{\tau-1}}\Big)^{\nu} \min\left\{L_{k,1},L_{k,2},\ldots,L_{k,d_{k}}\right\},
	\ee
where $d_{j}=\ceil{\exp((1-\eps)M C^{j+1} (D-1))}$ and $(L_{k,i})_{i\le d_{k}}$ are the lengths of the first $d_{k}$ edges from $c_{k}$ to $\delta$-good leaders in $\Gamma_{k+1}(0)$ (ordered arbitrarily). These are i.i.d.\ copies of the random variable $L$.
For any $N \in \N$ and $\zeta>0$, for i.i.d.\ copies $L_1,\dots,L_N$ of $L$, we have
\be\label{eq:lower-min} 
\Pv\big(\min_{j\leq N} L_j > F_L^{(-1)}(\zeta/N)\big) = \big(1-F_L(F_L^{(-1)}(\zeta/N))\big)^N \le (1 -\zeta/N)^N\le \e^{-\zeta}
\ee 
since $F_L(F_L^{-1}(x))\ge x$ by the right-continuity of the cdf. Hence taking $\zeta:=\zeta_k$ and $N := d_{k}$, with probability at least $1-\e^{-\zeta_{k}}$,
	\begin{equation}\label{eq:expl-2}
	\min\left\{L_{k,1},L_{k,2},\ldots,L_{k,d_{k}}\right\} \le F_L^{(-1)}(\zeta_{k}/d_{k}) \le F_L^{(-1)}\big(\zeta_{k}\e^{-(1-\eps)MC^{k+1}(D-1)}\big)\,.
\end{equation}
A union bound implies that these events all happen with probability at least $1-\sum_{k\ge 0}\e^{-\zeta_k}$, so~\eqref{eq:cost-greedy} follows from~\eqref{explosivesum}.
\end{proof}

\begin{proof}[Proof of Theorem \ref{thm:mu-classification}(ii)]
	We only prove the result for $\Ilambda$, since the same proof works for $\SFPWL$ (but can be simplified using nearest--neighbor edges for the start of the path). Condition on the origin lying in the vertex set $\CV_\la$. Then with Lemma \ref{lem:bipartite} at hand, we construct a greedy path with finite total cost from the origin. 
	First we find a boxing system for which the event $\cap_{k\ge 0} (F_k^{(1)}\cap F_k^{(2)})$ of Lemma \ref{lem:bipartite} occurs. 
	Since $\tau \in (1,3)$ and $2\mu\beta^+ < 3-\tau$ by hypothesis, Claim~\ref{cl:parameters} (applied with $\nu=\mu$) implies we can choose $\delta>0$ such that $\mathcal I_\delta$ is non-empty. Then taking $C:=1+\delta$ and $D\in \mathcal{I}_\delta$, equation~\eqref{eq:cd} is satisfied (by~\eqref{eq:cdx2}), thus satisfying the conditions of Lemma \ref{lem:bipartite}.
	We then apply Lemma~\ref{lem:bipartite}, choosing $\eps := \delta$ in the lemma, and let $M_0$ be as in the lemma statement. We then define $M_i = M_0 + i$ for all $i > 0$, and construct infinitely many boxing systems around $0$ with parameters $C$, $D$ and $M_i$. Note that taking $M=M_i$ in~\eqref{boundsweight-error}, the rhs is summable; thus by Lemma~\ref{lem:bipartite} and the Borel-Cantelli lemma, there exists $i_0$ such that for all $i \ge i_0$, $\cap_{k\ge 0}(F_k^{(1)} \cap F_k^{(2)})$ occurs. 

	From now on we only consider a boxing system with parameters $C$, $D$ and $M \ge M_{i_0}$ which is sufficiently large for~\eqref{eq:fini} below to hold. Note that the event $\cap_{k\ge 0}(F_k^{(1)} \cap F_k^{(2)})$ depends only on $\CV_\lambda$, $\{W_v\colon v \in \CV_\lambda\}$, and on the set of edges between $\delta$-good leaders. Exposing these variables, and letting $c_0$ be an arbitrary $\delta$-good leader in $\Gamma_0(0)$ (which exists since $F_0^{(1)}$ occurs), we see that with positive probability $p$, either $c_0=0$ or there is an edge from $0$ to $c_0$. Suppose there is an edge from $0$ to $c_0$; the $c_0=0$ case is essentially identical. Conditioned on this event, we 
	use the greedy path $\pi^{greedy}$ constructed in Definition \ref{def:greedy}  with initial vertex $c_0$, and set $\pi^0:=(0, \pi_{greedy})$. 
The bound in Claim \ref{lem:cost-greedy} holds with probability $1-\sum_{k\ge0}\e^{-\zeta_k}$, and we choose $\zeta_k:=\log(1/p)+k+1$ so that $\sum_k e^{-\zeta_k} < p$. Thus by a union bound, with positive probability,~\eqref{eq:expl-2} holds.

Hence, 
the cost of the constructed path, with positive probability, is at most
	\be\label{eq:explosivesum}\ba 
	|\pi^0|_{\mu,L} &\le  C_{(0,c_0)} + \sum_{k=0}^\infty  \Big(\e^{MC^k\tfrac{1+\delta}{\tau-1}}\Big)^{\mu} \Big(\e^{MC^{k+1}\tfrac{1+\delta}{\tau-1}}\Big)^{\mu} F_L^{(-1)}(\zeta_k \e^{-(1-\delta)MC^{k+1}(D-1)}).
	\ea
	\ee
	The first term on the rhs is an a.s. finite random variable. Hence, 
	to show explosion it suffices to prove that the last sum is finite for our choice of $C$, $D$, $M$ and $\delta$. For this we use the definition of $\beta^+$ in \eqref{eq:beta2}, which implies that for all sufficiently small $x>0$, $\log F_L(x)/\log x \le \beta^+/(1-\delta)$, or equivalently that $F_L(x)\ge x^{\beta^+/(1-\delta)}$. This in turn implies that $F_L^{(-1)}(y) \le y^{(1-\delta)/\beta^+}$  holds for all sufficiently small $y>0$. Hence when $M$ is sufficiently large,
	\be\label{eq:fini}\ba
		\sum_{k=0}^\infty&\e^{\mu MC^k \tfrac{1+\delta}{\tau-1}(1+C)}F_L^{(-1)}\big(\zeta_k \e^{-(1-\delta)MC^{k+1}(D-1)}\big) \\
		&\le \sum_{k=0}^\infty \zeta_k^{(1-\delta)/\beta^+} \exp\Big(\mu MC^k \tfrac{1+\delta}{\tau-1}(1+C) - \tfrac{(1-\delta)^2}{\beta^+}MC^{k+1}(D-1)\Big).\ea
	\ee
	This sum is finite if and only if the exponent is negative, i.e., if and only if
	\be
		\mu(1+C) \frac{1+\delta}{\tau-1} - \frac{(D-1)C(1-\delta)^2}{\beta^+} < 0.
	\ee
	Since $D \in \mathcal{I}_\delta$, this is true by~\eqref{eq:mubeta-3} of Claim~\ref{cl:parameters} (taking $\nu=\mu$ and $s=1$). Thus by~\eqref{eq:explosivesum}, $|\pi^0|_{\mu,L}$ is finite with positive probability as required.\end{proof}
	
Next we discuss how the above method can be modified to work with more general penalty functions, proving Theorem~\ref{thm:poly-classification}(ii). 
\begin{proof}[Proof of Theorem~\ref{thm:poly-classification}(ii)]

The proof of Theorem~\ref{thm:poly-classification}(ii) is very similar to that of Theorem~\ref{thm:mu-classification}(ii), so we only describe where it should be modified. We may assume that $f$ is of the form $f= w_1^{\mu}w_2^\nu$, since for any given path $\pi$, a polynomial penalty function with non-negative coefficients yields finite cost if and only if each of its monomials yields finite cost.
Since $\tau \in (1,3)$ and $(\mu+\nu)\beta^+ < 3-\tau$ by hypothesis, Claim \ref{cl:parameters} implies we can choose $\delta>0$ such that $\mathcal{I}_\delta$ is non-empty. We then take $C:=1+\delta$ and $D \in \mathcal{I}_\delta$, and choose $M$, apply Lemma~\ref{lem:bipartite}, and construct $\pi^0$ in exactly the same way as in the proof of Theorem~\ref{thm:mu-classification}(ii). 

The almost sure bound in\eqref{eq:explosivesum} on the total cost of $\pi^0$ now becomes
\be\label{explosivesum-max}\ba
|\pi^0|_{f,L} &\le  C_{(0,c_0)}+\sum_{k=0}^\infty \Big(\e^{MC^{k}\tfrac{1+\delta}{\tau-1}}\Big)^\mu\Big(\e^{MC^{k+1}\tfrac{1+\delta}{\tau-1}}\Big)^\nu F_L^{(-1)}(\zeta_k \e^{-(1- \delta)MC^{k+1}(D-1)}).
\ea\ee	
Bounding the $F_L^{(-1)}$ term above as in \eqref{eq:fini}, we see that $\pi^0$ has finite cost if
\be \label{eq:explo-sum-bound}  \sum_{k=0}^\infty \zeta_k^{(1+\delta)/\beta^+}\exp\Big((\mu + \nu C) MC^{k}\frac{1+\delta}{\tau-1} - \frac{(1-\delta)^2}{\beta^+}MC^{k+1}(D-1)\Big)<\infty,\ee
which is the case if the exponent is negative. Since $D \in \mathcal{I}_\delta$, this holds by~\eqref{eq:mubeta-3} of Claim~\ref{cl:parameters} (taking $s=1$) as in the proof of Theorem~\ref{thm:mu-classification}(ii).
\end{proof}
\subsection{Extensions of the boxing system}\label{s:box-extend}

In the proof of Theorems~\ref{thm:mu-classification}(ii) and~\ref{thm:poly-classification}(ii), it was enough to say that any vertex $u$ is connected to a $\delta$-good leader vertex $c$ with positive probability. To show the existence of a giant component in Theorem~\ref{thm:giant}, we will need an upper bound on the failure probability (when $u$ has suitably high constant weight). To generalise the result to the finite model in Theorem~\ref{thm:GIRG-fine}, we will also need an upper bound on the cost of the path from $u$ to $c$. In this section, we present some additional definitions and lemmas for this purpose; all proofs are deferred to Appendix~\ref{s:app-cost}. The reader might wish to skip this section for now and return to it when Theorems~\ref{thm:giant} and \ref{thm:GIRG-fine} are proved in Sections~\ref{s:giant} and~\ref{proof-finite}.

Fix some $M>0$ as in Lemma \ref{lem:bipartite} above, and let $\CV_{n,M}:=\{i\in [n] \mid W_i \geq \e^M\}$ be the set of vertices in $[n]$ with weight at least $\e^M$. 
Consider a vertex $u\in \CV_{n,M}$, and start a boxing system centered at its position $x_u\in \R^d$ with parameters $\delta, \ve, C,D$ as in Lemma \ref{lem:bipartite} (given in~ \eqref{eq:parameters}).  
Recall that in BGIRG, we require all sub-boxes to fit into $\calX_d(n)$. Recall from~\eqref{eq:kstar} that $k^\star=k^\star(n,M)$
is the largest $k$ such that at least a $2^{-d}$ fraction of $\text{Box}_{k^\star}(u)$ fits into $\calX_d(n)$, and recall the resulting bounds on the number of sub-boxes $b_k$ from~\eqref{eq:boundbk-girg}.
Recall the definitions of $\delta$-good leaders and sub-boxes from before \eqref{eq:delta-good}, and recall the definitions of $F_k^{(1)}$, $F_k^{(2)}$ and $p_M$ from Lemma~\ref{lem:bipartite}. 

We call a path $u,v_1,...,v_s$ \emph{box-increasing} if there exists $0\leq k_0 \leq k^\star-s$ such that $v_i$ is the leader of a $\delta$-good sub-box in the $(k_0+i)$th annulus $\Gamma_{k_0+i} \cap \calX_d(n)$, for all $1 \leq i \leq s$. Define the event $\CS_u$ that $u$ is \emph{successful} by
\be\label{eq:successful}
\CS_u:=\big\{ \text{There is a box-increasing path from $u$ to a $\delta$-good leader in $\Gamma_{k^\star}(u)$ }\big\}\cap F_{k^\star}^{(1)}(u).
\ee

\begin{lemma}\label{lem:highw-connect} Consider a boxing system with parameters $\delta$, $C=C(\delta)$, $D=D(\delta)$ satisfying~\eqref{eq:parameters} (and hence~\eqref{eq:cd}) centered around the location $x_u$ of a vertex $u$ with given weight $\e^M \le w_u\le \exp(MC^{k^\star}\tfrac{1-\delta}{\tau-1})$.  Let $\eta(\delta):=((D-1)\wedge 1)(\tau-1)$.
Then there exists 
$M_0(\de) > 0$  such that if $M\geq M_0(\de)$ \cbb and $n$ is sufficiently large, 
the following holds in $\BGIRG$:
\be\label{eq:su-bound-1}  \Pv( \neg \CS_u \mid W_u = w_u) \leq 5\exp\big( -(\underline c \wedge 1) 2^{-d-7}   w_u^{\eta(\delta) }\big).\ee
\end{lemma}

The next lemma is a version of Lemma~\ref{lem:highw-connect} that also bounds the cost of the path from $u$, which will be useful for the proof of Theorem \ref{thm:GIRG-fine}, at the cost of a significantly increased failure probability. Importantly, the path from $u$ will only use vertices with weight strictly larger than $w_u$; later, this will allow us to safely condition on the path's existence.
For a choice of $\delta, C, D$ as in \eqref{eq:parameters}, set 
\be \ba \label{eq:xi-rho} \xi(\delta):=-(\mu + \nu C) \frac{1+\delta}{\tau-1} +\frac{(1-\delta)^2}{\beta^+}C(D-1)>0, \\
\rho(\delta):= - \frac{\tau-1}{1-\delta}\Big(\frac{\mu+C\nu}{\tau-1} - \frac{(1-\delta)^2(D-1)}{\beta^+} \Big)>0.\ea\ee
The positivity of $\xi(\delta), \rho(\delta)$ follows from \eqref{eq:mubeta-3}, recalling that $C=1+\delta$.
\begin{lemma}\label{lem:cheap-connect}
Consider the penalty function $f(w_1,w_2):=w_1^\mu w_2^\nu$ with $(\mu+\nu)\beta^+<3-\tau$. There is a constant $K<\infty$, such that for a vertex $u$ with given weight $w_u\in [K, n^{(1-\delta)/(D(\tau-1))})$ in $\BGIRG$, the following holds. Construct a boxing system with parameter $\delta$, $C$ and $D$ as in \eqref{eq:idelta}, and $M_{w_u}:=\tfrac{\tau-1}{1-\delta}\log w_u$.  Then, with failure probability at most 
\be \label{eq:xi-error} \Xi(w_u):= 4w_u^{-\tfrac{\tau-1}{1-\delta} \tfrac{\beta^+}{1+\delta} \xi(\delta)/2},\ee 
there exists a box-increasing path $\pi^u$ from $u$ to a $\delta$-good leader vertex $\pi^u_{\mathrm{end}}$ in $\Gamma_{k^\star(M_{w_u},n)}(u)$ with total cost at most
\be \label{eq:xi-cost} |\pi^u|_{f,L}  \le w_u^{-\rho(\delta)/2} + 2w_u^{-\tfrac{\tau-1}{1-\delta}\xi(\delta)/2}. \ee
Moreover, the only vertex in $\pi^u$ with weight at most $w_u$ is $u$ itself.
 Finally, the other end vertex $\pi^u_{\mathrm{end}}$ of $\pi^u$ has weight in the interval
 \be\label{eq:weight-end}  W_{c_{\text{end}}} \in \Big( n^{(1-\delta)/(D C(\tau-1))}, n^{(1+\delta)/(D(\tau-1))}\Big]. \ee
 \end{lemma}
The last lemma connects two of these end-vertices with a low-cost path.
\begin{lemma}[Cost between end-vertices of greedy paths]\label{lem:core}
Consider the setting of Lemma \ref{lem:cheap-connect}, but for an arbitrary polynomial weight function as in \eqref{eq:pol} such that $\deg(f) < (3-\tau)/\beta^+$. 
Then, there exist constants $\zeta=\zeta(f, L)>0$ and $K=K(\zeta)<\infty$ such that whp, all pairs of vertices $u_1,u_2$ with weights in the interval
\be\label{eq:iend}  I_{\mathrm{end}}:=[ n^{(1- \delta)/(DC(\tau-1))}, n^{(1+ \delta)/(D(\tau-1))} ] \ee
have cost-distance at most
\be\label{eq:core-cost} d_{f,L}(u_1,u_2) \le K(\zeta)n^{-\zeta}.  \ee
\end{lemma}

\section{Sideways explosion}\label{sec:sideways}
In this short section we prove Theorems~\ref{thm:mu-classification}(i) and \ref{thm:poly-classification}(i).
We start with a lemma.
\begin{lemma}\label{lem:inf1}
Consider $\Ilambda$ or $\SFPWL$ satisfying Assumption \ref{assu:h} with some $\alpha\le 1$. Let $[c,d]$ be an interval with $\Pv(W\in [c,d])>0$. Then almost surely, every vertex has infinitely many neighbors with weight in $[c,d]$. 
\end{lemma}
\begin{proof}
We abbreviate $l(w):=l_{c_2, \gamma}(w)$. By translation invariance of the models, it is enough to show that the statement holds for the origin. 
Let $w_0 \in [1,\infty)$; we condition on the vertex-weight of the origin being $w_0$. Let us denote the number of neighbors of $0$ whose weights lie in the interval $[c,d]$ by $D_0[c,d]$. 
First consider $\Ilambda$. Conditioned on $W_0 = w_0$, $D_0[c,d]=\sum_{v_1\in \CV_\la} \ind\{0 \leftrightarrow v_1 \cap W_{v_1} \in [c,d]\}$ is a sum of independent indicator variables with expectation 
\[ \Ev\big[D_0[c,d]\mid W_0=w_0\big] =\Ev\Big[\sum_{v_1\in \CV_\la} \Ev\big[h_\mathrm{I}(v_1, w_0, W_{v_1})\ind\{W_{v_1}\in [c,d]\}\mid \CV_\la \big]\Big]. \]
Taking the conditional expectation with respect to $W_{v_1}$, using the lower bound on $h_{\mathrm I}$ in Assumption \ref{assu:h}, and denoting Lebesgue-Stieltjes integration wrt the measure of $F_W$ by $F_W(\mathrm dw_1)$, we obtain
\be \label{eq:d0cd}\Ev\big[D_0[c,d]\mid W_0=w_0\big]  \ge \underline c \int_{w_1\in[c,d]}  \Ev\Bigg[\sum_{v_1\in \CV_\la}  \Big(l(w_0) l(w_1) \wedge  \Big(\frac{w_0w_1}{\| v_1\|^d}\Big)^\al\Big)\Bigg]  F_W(\mathrm dw_1).
\ee
Whenever $\|v_1\|^d \ge w_0 w_1/(l(w_0)l(w_1))^{1/\al}:=K_{w_0, w_1}$, the minimum is attained at the second term. Hence the inner expectation, which we abbreviate to $T(w_0,w_1)$, can be bounded below by 
\be\label{eq:t01}  T(w_0, w_1)\ge  \Ev\Bigg[\sum_{v_1\in \CV_\la: \| v_1\|^d \ge K_{w_0, w_1}}    \Big(\frac{w_0w_1}{\| v_1\|^d}\Big)^\al\Bigg]
=\la (w_0w_1)^\al \int_{x\in \R^d: \| x\|^d\ge  K_{w_0,w_1} } \| x\|^{-d\al} \mathrm d\nu_d,
\ee
where $\mathrm d\nu_d$ denotes integration wrt the Lebesgue measure on $\R^d$. 
Changing variables yields 
\be\label{eq:t012}T(w_0, w_1) \ge \la (w_0w_1)^\al  \int_{r^d\ge K_{w_0,w_1}} r^{d-1}r^{-d\al} \mathrm dr =\infty. \ee
Since $\al\le 1$,  the integrand in \eqref{eq:d0cd} is infinite, and since we assumed that $\Pv(W\in [c,d])>0$, 
we obtain that $\Ev[D_0[c,d]\,|\,W_0=w_0]=\infty$. Conditioned on $W_0 = w_0$, $D_0[c,d]$ is a sum of independent indicators, so $D_0[c,d]$ itself is also infinite a.s.\ by the second Borel-Cantelli lemma.
For $\SFPWL$, the same argument applies, except that the integral in~\eqref{eq:t01} is replaced by a sum over all lattice points with $\| x\|^d\ge  K_{w_0,w_1}$. This also yields an infinite integrand in \eqref{eq:d0cd}.
 \end{proof}
 
 We use Lemma~\ref{lem:inf1} to prove the following stronger version of Theorem~\ref{thm:mu-classification}(i).
 
 \begin{theorem}\label{thm:explosive0} Consider the models $\Ilambda$ and $\SFPWL$ with $d\ge 1$, $\al\in(0,1]$, $\tau>1$, arbitrary vertex-weight distribution $W\ge 1$, and connection functions $h_{\mathrm{I}}, h_{\mathrm{S}}$ satisfying Assumption \ref{assu:h}. Consider an arbitrary penalty function $f(w_1, w_2)$, with the following property: there exist intervals $[a,b],[c,d]$ with $\Pv(W\in[a,b])\Pv(W\in[c,d])>0$ and $K<\infty$ such that $f(w_1, w_2)<K$ whenever $w_1\in[a,b], w_2 \in [c,d]$. 
 Then $\Ilambda$ and $\SFPWL$ are $(f,L)$-explosive for arbitrary edge-weight distribution $L$. Moreover, with positive probability, each vertex has a.s.\ infinitely many \emph{neighbors} within bounded cost. 
 \end{theorem}
 
 \begin{proof}
 	The result will follow from Lemma~\ref{lem:inf1}. Let us denote the set of neighbors of the origin with weight in $[c,d]$ by $\CN[c,d]$, and note that there exists $x < \infty$ such that $\Pv(L \le x) > 0$. Let $\calE_1$ be the event that the origin receives weight in $[a,b]$; let $\calE_2$ be the event that $|\CN[c,d]| = \infty$; and let $\calE_3$ be the event that infinitely many edges between $0$ and $v$ receive cost at most $K$. Then $\Pr(\calE_1) > 0$; Lemma~\ref{lem:inf1} implies that $\Pr(\calE_2 \mid \calE_1) = 1$; since the edge-weight variables $L_e$ are mutually independent; the second Borel-Cantelli lemma implies that $\Pr(\calE_3 \mid \calE_1,\,\calE_2) = 1$; and $\calE_3$'s occurrence implies sideways explosion. Thus the result follows.
\end{proof}

We now use Theorem~\ref{thm:explosive0} to prove Theorem~\ref{thm:poly-classification}(i).

\begin{proof}[Proof of Theorem \ref{thm:poly-classification}(i)]
The $\alpha \le 1$ case is immediate from Theorem~\ref{thm:explosive0}, so we may assume $\alpha > 1$. As before, we will prove the result for $\Ilambda$, and the proof for $\SFPWL$ will be analogous. We may bound the penalty function $f(w_1,w_2)$ above by $aw_1^{\mu}w_2^\nu$, where $\mu = \max\{\mu_i \mid i\in \CI\}$, $\nu = \max\{\nu_i\mid i\in\CI\}$, and $a>0$ is a suitable constant. 

Let $N_1^t:=|\{ (0,u) \in \CE_\la: C_{(0,u)}\le t \}|$ be the number of neighbors of the origin such that the edges leading to these neighbors have cost at most $t$. 
Let $w_0 > 0$ be arbitrary; we will show that conditioned on the weight $W_0$ of the origin being equal to $w_0$, $N_1^t = \infty$ almost surely, for all $t>0$. Thus explosion does not simply occur with positive probability --- it almost surely occurs instantaneously.
Observe that $N_1^{t}$ is again a sum of independent indicators: conditioned on $W_0 = w_0$, we have
\be\label{eq:n1t-222} N_1^{t} = \sum_{v_1\in \CV_\la} \ind\{0 \leftrightarrow v_1\}\ind\{f(w_0,W_{v_1}) L_{(0, v_1)}\le t\} \ge \sum_{v_1\in \CV_\la} \ind\{0 \leftrightarrow v_1\}\ind\{aw_0^\mu W_{v_1}^\nu L_{(0, v_1)}\le t\}. \ee
We first investigate the case when $\nu>0$. Using the law of total probability on the value of $W_{v_1}=:w_1$  as well as using the lower bound on $h_{\mathrm I}$ in Assumption \ref{assu:h} results in
\be \ba\label{eq:k=1}
		 \Ev\big[N_1^{t} \mid W_0=w_0\big]&\ge \int\limits_{w_1} F_L(ta^{-1}w_{0}^{-\mu} w_1^{-\nu})  \E\left[\sum_{v_1 \in \CV_\la}  \underline c \bigg(l(w_0)l(w_1)\wedge\frac{w_{0}w_1}{\norm{v_1}^d} \bigg)^\alpha\right] F_W(\mathrm dw_1).
		 \ea\ee
Exactly as in~\eqref{eq:d0cd}--\eqref{eq:t012} from the proof of Lemma~\ref{lem:inf1}, denoting the inner expectation by $T(w_0,w_1)$, we have
\[
	T(w_0,w_1) \ge \lambda(w_0w_1)^\alpha \int_{r^d \ge K_{w_0,w_1}} r^{d-1-d\alpha}\mathrm{d}r,
\]
where $K_{w_0,w_1} = w_0w_1/(l(w_0)l(w_1))^{1/\alpha}$. Unlike in~\eqref{eq:t012}, the rhs is integrable (since $\alpha>1$); for some constant $c$, we have
\be\label{eq:t-lower}  T(w_0, w_1) \ge \la c(w_0w_1)^\alpha K_{w_0,w_1}^{1-\alpha} = \la c w_0 w_1(l(w_0)l(w_1))^{(\al-1)/\al}.\ee
	By \eqref{eq:beta1},  for any $\ve>0$ there is a $t_0=t_0(\ve)$ such that the bound $F_L(x)\ge x^{\beta^++\ve}$ holds in the interval $[0, t_0)$. Let us write $b_\ve^+:=\beta^++\ve$. The argument of $F_L$ in \eqref{eq:k=1} is at most $t_0$ when $w_1\ge (t/(at_0))^{1/\nu}w_0^{-\mu/\nu}$, yielding the lower bound
	\be\ba  \label{eq:k=13} \Ev\big[N_1^{t}\mid W_0 = w_0\big]&\ge\la c \!\!\!\! \int\limits_{w_1\ge (t/(at_0))^{1/\nu}w_0^{-\mu/\nu}\vee 1}\!\!\!\! t^{b_\ve^+}a^{-b_\ve^+} w_0^{-\mu b^+_\ve} w_1^{-\nu b^+_\ve}w_0w_1(l(w_0)l(w_1))^{(\al-1)/\al} F_W(\mathrm dw_1)
	\\&=\la c (t/a)^{b_\ve^+}w_0^{1-\mu b_\ve^+} l(w_0)^{\frac{\al-1}{\al}} \Ev\Big[W_1^{1-\nu b_\ve^+} l(W_1)^{(\al-1)/\al}\ind_{\{ W_1\ge (t/t_0)^{1/\nu}w_0^{-\mu/\nu}\vee 1  \}}\Big]
	\ea
	\ee	 
We claim that for $\ve>0$ small enough, the last expectation is infinite. Indeed, $l(W_1)$ is varies slowly at infinity, so when $W_1$ is sufficiently large, $l(W_1)^{(\alpha-1)/\alpha} \ge W_1^{-\ve}$ by Potter's bound.  Then applying Karamata's theorem~\cite[Proposition~1.5.10]{BinGolTeu89} for $1-F_W(w)=l_W(w)w^{1-\tau}$, for some constant $c$ and for any sufficiently large constant $K$, we obtain
\[\ba \Ev[W^{1-\nu b_\ve^+} l(W_1)^{(\al-1)/\al}\ind\{W_1\ge K\}] &\ge \Ev[W_1^{1-\nu b_\ve^+-\ve} \ind\{W_1\ge K\}] &\\
&\ge c\int_K^{\infty} w^{-\ve-\nu b_\ve^++(1-\tau)} l_W(w) \mathrm d w=\infty\ea \]
as long as $\ve$ is so small that $-\ve-\nu (\beta^++\ve)+(1-\tau)>-1$. Such an $\eps$ exists whenever $\beta^+<(2-\tau)/\nu$, which is satisfied by our assumption on $\{\nu_i\colon i \in \calI\}$.
 To finish, given $w_0$, the presence of edges going out of the origin are conditionally independent, hence the second Borel-Cantelli lemma ensures that $N_1^t \mid W_0=w_0$ is a.s.\ infinite (regardless of the value $w_0$).

 Next we consider the case when $\nu=0$.
 In this case we return to \eqref{eq:n1t-222} and observe that the last indicator does not depend on $W_{v_1}$. 
 This implies that a factor $F_L(ta^{-1}w_{0}^{-\mu})$ can be taken out of the integral on the rhs of \eqref{eq:k=1}, which, combined with the lower bound on $T(w_0,w_1)$ in \eqref{eq:t-lower}, results in
 \be\ba \label{eq:N1tlower}
		 \Ev\big[N_1^{t}\mid W_0 = w_0\big]&\ge F_L(tw_{0}^{-\mu}) \int\limits_{w_1}   \la   c w_0 w_1(l(w_0)l(w_1))^{(\al-1)/\al} F_W(\mathrm dw_1).\\
		& = F_L(tw_{0}^{-\mu}) \la   c w_0l(w_0)^{(\al-1)/\al}\Ev[ W_1l(W_1)^{(\al-1)/\al}].
		  \ea\ee
		  Since $\alpha>1$, the latter expectation is infinite whenever $\tau\in(1,2)$.
\end{proof}
	\section{Understanding explosion to show conservativeness}\label{sec:cons}
	
In this section we prove Theorem~\ref{thm:mu-classification}(iii). Somewhat counter-intuitively, to be able to show that a model is conservative, we need to better understand the ways in which a model can explode. 
We start with two general lemmas about explosion.   
	Recall the definition of the explosion time from Definition \ref{def:explosiontime}, and that there are two non-exclusive ways for a vertex $v$ to have finite explosion time. In \emph{sideways explosion}, we can reach infinitely many vertices via paths with bounded cost \emph{and} bounded length from $v$. It is not hard to see that this is equivalent to the presence of a (possibly trivial) finite-cost path from $v$ to some vertex $w$ that has infinitely many neighbors via bounded-cost edges.  Formally, we modify the notation introduced in the previous section: $N_{1}^t(v):=|\{ (v,u) \in \CE_\la: C_{(v,u)}\le t \}|$. Then sideways explosion is the event that 	\be\label{eq:sideways1} \text{there is $t>0$ and a finite path from $v$ to a vertex $v'$ such that $N_1^t(v') = \infty$.} \ee
	We remark that sideways explosion does \emph{not} just mean that there are infinitely many vertices which have both bounded graph distance and bounded cost-distance from $v$. For such vertices, even if the graph distance from $v$ is bounded, this does not imply that the paths attaining the minimal cost-distances are also unbounded.
	
	The second possibility is \emph{lengthwise explosion}, in which there is an infinite path of vertices $\pi=(\pi_0=v, \pi_1, \pi_2\dots)$ with total cost $|\pi|_{f,L}<\infty$, i.e. 
	\be \wit Y_f^{\mathrm I}(v):= \inf_{\pi: \pi_0=v, |\pi|=\infty}\{|\pi|_{f,L}\} < \infty. \ee

	We will use the next lemma to rule out sideways explosion in SFP and IGIRG in the situation of Theorem~\ref{thm:mu-classification}(iii), and to find a specific path $\pi$ attaining the explosion time $\wit Y_f^{\mathrm I}(v)=Y_f^{\mathrm I}(v)$. The assumption $\tau\in (1,3)$ in the second statement ensures that there is an infinite component $\calC_\infty$ (this can be seen by Lemma~\ref{lem:bipartite}). We mention that the infinite component is unique; for scale-free percolation, uniqueness of an infinite component in any $d\ge1$ follows from the main result in \cite{GanKeaNew92}, and for IGIRG (i.e., continuum percolation), it follows from \cite{MeeRoy96}, see also \cite{DeiHofHoo13, DepWut18}. 
	\begin{lemma}\label{lem:opt-real}Consider $\Ilambda$ or $\SFPWL$ with parameters $d\ge 1, \tau >1, \al\in(1,\infty]$. Consider a penalty function $f$, and edge-length distribution $L \ge 0$ such that for all $t<\infty$, $\Pv(N_1^t(0) <\infty)=1$ holds. Then sideways explosion almost surely does not happen. 
	Moreover, if $\tau \in (1,3)$, then for any vertex $v$ in the infinite component, $\wit Y_f^{\mathrm I}(v)= Y_f^{\mathrm I}(v)$ is \emph{realised} via (at least one) infinite path $\pi_{\mathrm{opt}}(v)$. 
	\end{lemma}
	We defer the proof to Appendix \ref{s:app-exp}, but we make some comments. The path $\pi_{\mathrm{opt}}(v)$ may not be unique -- this might occur e.g.\ when $L,W$ are not absolutely continuous distributions, or if $Y_f^{\mathrm I}(v) = \infty$: then, any infinite path can be chosen since they all have infinite cost.
	Second, the negation of the condition ``$\forall t<\infty, \ \Pv(N_1^t(0)<\infty)=1$'' is that for some $t<\infty$ we have $\Pv(N_1^t(0)=\infty)>0$, so sideways explosion happens with positive probability at the origin.
	
The next lemma shows that if lengthwise explosion may happen at all, then it may happen arbitrarily~fast. A similar result for age-dependent branching processes was proved by Grey~\cite{grey1974}. However, that proof relies on an independent subtree decomposition which is not applicable to spatial random graphs, so we need to use different methods.
	\begin{lemma}\label{lem:short-expl}
		Consider $\Ilambda$ or $\SFPWL$, with a penalty function $f$ and edge-length distribution $L \ge 0$, such that explosion occurs with positive probability, but that for all $t<\infty$, $  \Pv(N_1^t(0)<\infty)=1$. Then for all constant $t_0>0$, with positive probability there is an infinite path from the origin with total cost at most $t_0$. 
	\end{lemma}
	\begin{proof}
	We first prove the lemma for $\Ilambda$, and then we discuss how to modify the proof for $\SFPWL$. 
	For brevity we write $Y:=Y_{f}^{\mathrm{I}}(0)$ for the (possibly infinite) explosion time of the origin. 
		Let $p_Y := \Pr[Y < \infty]$, and note that $p_Y >0$ by hypothesis. We first show that with positive probability, there exists some vertex $v(t_0/2)$ in a suitably large Euclidean ball with explosion time at most $t_0/2$. We then show that conditioned on this event, again with positive probability, the origin is joined to $v(t_0/2)$ with a path of cost at most $t_0/2$; we thereby obtain a path to infinity of cost at most $t_0$, as required.

		We first find the required ball. For all $r, x>0$, let $\calA_r^{x}$ be the event that there exists a vertex $v=v(x)$ in the Euclidean ball $B_r^2(0)$
		from which there is an infinite path of cost at most $x$. Note that we may assume that $v(x)$ is the only vertex on this path that lies in $B_r^2(0)$, since a.s. there are only finitely many vertices in $B_r^2(0)$ and thus we may truncate the beginning of the path if necessary. We will prove that:
		\begin{align}\label{eq:short-expl1}
		\text{For all $x >0 $, there exists $r(x)<\infty$ such that $\Pr[\calA_{r(x)}^{x}] \geq p_Y/2$.}
		\end{align}
		To prove~\eqref{eq:short-expl1}, we first define a random variable $R(x)$ as follows. If $Y = \infty$, then we define $R(x) := \infty$ also. Otherwise, by Lemma~\ref{lem:opt-real}, there a.s.\ exists at least one path from $0$ with cost $Y < \infty$. We wish to choose one such path in a well-defined way, so we order them lexicographically according to the Euclidean norms of their vertices, and take $\pi_{\mathrm{opt}}(0)$ to be the first path in the order. (Note that a.s.\ every vertex in $\CV_\la$ has a different norm, so the order is a.s.\ unique.) Let $v_i = v_i(x)$ be the (random) first vertex on $\pi_{\mathrm{opt}}(0)$ such that the subpath $(v_i(x),v_{i+1},\ldots)$ of $\pi_{\mathrm{opt}}(0)$ has cost at most $x$. Then we define $R(x) := \|v_i(x)\|$.
		For all $x>0$, $R(x)$ is a well-defined random variable, with $\Pr[R(x) < \infty] = \Pr[Y <\infty] = p_Y >0$. Hence there exists $r(x)>0$ such that $\Pr[R(x) \leq r(x)] \geq p_Y/2$. 
		 Finally, on the event $\{R(x) \leq r(x)\}$, we set $v(x)$ to be the last vertex on $(v_i(x),v_{i+1},\dots)$ that is still within $B^2_{r(x)}(0)$.
Since the event $\{R(x) \leq r(x)\}$ contains the event $\calA_{r(x)}^{x}$, this proves~\eqref{eq:short-expl1}.
		
		Now fix $t_0 > 0$, and for brevity write $r_0 := r(t_0/2)$ and $\calA := \calA_{r_0}^{t_0/2}$; thus $\Pv(\calA)\ge p_Y/2$ by \eqref{eq:short-expl1}.
		Now, let $\calB$ be the event that the origin is connected to every vertex in $B^2(0,r_0)$ via a path of cost at most $t_0/2$. (Since the number of vertices in $B^2(0,r_0)$ follows a Poisson distribution with mean $\nu(B^2(0, r_0))$, it is a.s. finite.) Observe that  if $\calA \cap \calB$ occurs, then by combining the low-cost path from the origin to $v(t_0/2)$ with $\pi_\mathrm{opt}(t_0/2)$, we obtain an infinite path from the origin with total cost at most $t_0$. Thus it suffices to prove that $\Pr(\calA \cap \calB) > 0$. 
		
		Let $F\Ext$ be the set of all edges not internal to $B^2(0,r_0)$, let $L\Ext := \{L_e \colon e \in F\Ext\}$, and observe that the event $\calA$ is \emph{determined} by the variables $\CV_\lambda, (W_v)_{v\in\CV_\la}, F\Ext$, and $L\Ext$. If $\alpha < \infty$, then the connection probability $h_{\mathrm I}$ is nonzero for all its arguments, and so almost surely,
		\[ q(\CV_\lambda,\,(W_v)_{v\in\CV_\la},\,F\Ext,\,L\Ext) := \Pr(\calB \mid \CV_\lambda,\,(W_v)_{v\in\CV_\la},\,F\Ext,\,L\Ext) > 0.\] In particular, there is a measurable set $S$ of values of $\CV_\lambda,\,(W_v)_{v\in\CV_\la},\,F\Ext,\,L\Ext$ that has positive probability and such that $S$ implies $\calA$ and $q>0$ on $S$. Hence, $\Pr[\calA \cap \calB] \geq \int_S q > 0$, and we are done. However, if $\alpha = \infty$, then it is no longer true that $q >0$ almost surely. For example, if the vertices in $B^2(0,r_0)$ have low weight, and there is a vertex $w \in B^2(0,r_0)$ which is far from any other vertices in Euclidean space, then it is a.s.\ isolated in $B^2(0,r_0)$. To deal with this issue, we pass to a thinned model to expose $\CV_\lambda \cap B^2(0,r_0)$ in two rounds.
		
		Let $\calG \sim \Ilambda$. We form $\calG\thin$ from $\calG$ by discarding \emph{vertices} in $\CV_\la \cap B^2(0,r_0)$ independently with probability $1/2$, and let $V\thin := V(\calG\thin)$ and $\mathcal W\thin := \{W_v \colon v \in V\thin\}$. Thus $\calG\thin$ is distributed as a thinned model of $\Ilambda$ in which the density of the Poisson point process is reduced to $\lambda/2$ within $B^2(0,r_0)$. Let $\calA\thin$ be the analogue of $\calA$ in $\calG\thin$, and let $\calB\thin$ be the event that in $\calG$, the origin is connected to every vertex in $V\thin \cap B^2(0,r_0)$ via a path of cost at most $t_0/2$. Observe that for all possible values of $\CV_\la$, $(W_v)_{v\in\CV_\la}$, $F\Ext$ and $L\Ext$ with $|\CV_\la \cap B^2(0,r_0)| < \infty$, we have 
		\[
			\Pr(\calG\thin = \calG \mid \CV_\la = V,\, (W_v)_{v\in\CV_\la},\,F\Ext,\,L\Ext) = 2^{-|V \cap B^2(0,r_0)|} > 0;
		\] 
		Thus, 
		\begin{equation}\label{eqn:short-expl2-2}
		\Pr(\calA\thin) \ge \Pr(\calA \mbox{ and }\calG\thin = \calG) > 0.
		\end{equation}
Now define $F\thin$ to be the set of all edges of $\calG\thin$ not internal to $B^2(0,r_0)$, and let $L\thin := \{L_e \colon e \in F\thin\}$. Then $\calA\thin$ is determined by $V\thin$, $\mathcal W\thin$, $F\thin$ and $L\thin$. Conditioned on arbitrary values of these variables, $V(\calG) \setminus V\thin$ is distributed as a Poisson point process on $B^2(0,r_0)$ with intensity $\lambda/2$, and the edges between $V(\calG) \setminus V\thin$ and $V\thin$ and their costs are distributed as usual in the $\Ilambda$ model. Thus, almost surely these variables take values such that $\Pr(\calB\thin \mid V\thin,\,\mathcal W\thin,\,F\thin,\,L\thin) > 0$. (For example, $\calB\thin$ occurs whenever $V(\calG) \setminus V\thin$ contains a suitably dense net for $B^2(0,r_0)$ which connects the origin to all vertices in $V\thin$ via suitably cheap edges.) Together with~\eqref{eqn:short-expl2-2}, this implies $\Pr(\calA\thin \cap \calB\thin) > 0$, just as before in the case $\alpha < \infty$. Since $\calA\thin \cap \calB\thin$ implies $Y \le t_0$, this concludes the proof for the IGIRG model.
		
		In the SFP model, the argument proceeds essentially as in the $\alpha < \infty$ case of the IGIRG proof. As before, there exists $r_0$ such that with positive probability, there is an infinite path from some vertex in $B^2(0,r_0)$ to infinity with total cost at most $t_0/2$; call this event $\calA$. Let $\calB$ be the event that every nearest-neighbor edge in $B^2(0,r_0)$ has cost at most $t_0/2dr_0$. Whenever $\calB$ occurs, the origin is joined to every vertex in $B^2(0,r_0)$ via a path of cost at most $t_0/2$, so 
		\[
			\Pr(Y \le t_0) \ge \Pr(\calA \cap \calB).
		\]
		As in the proof for IGIRG, $\calA$ is determined by $( W_{v})_{v\in\Z^d}$, $F\Ext$ and $L\Ext$. The values of these variables are almost surely such that $\Pr(\calB\mid\, (W_{v})_{v\in\Z^d},F\Ext,\,L\Ext) > 0$, even if $\alpha = \infty$, so $\Pr(\calA \cap  \calB) > 0$ and the result follows.
	\end{proof}
 We move towards the proof of Theorem~\ref{thm:mu-classification}(iii). To show this theorem, we count certain paths that we define now. 
 Let us define a subgraph $\CG(t_0)$ of $\Ilambda$ by keeping only the  edges with cost at most $t_0$. To structure the paths emanating from the origin $0=:v_0$ within $\CG_0$, we define the \emph{self-avoiding walk tree}, $\saw$ of $\CG(t_0)$ as follows. 
 
 The root of $\saw$ is the trivial path $\pi_0:=(v_0)$. The direct children of the root are paths of length $1$ of the form $\pi_1=(v_0, v_1)$; where we set the cost of the edge between $\pi_0$ and $\pi_1$ to be $C_{(v_0, v_1)}$. Generally, vertices of $\saw$ are the finite simple paths in $\CG(t_0)$ emanating from $v_0=0$, where a path $\pi_k=(v_0,\dots, v_k)$ in the $k$th level of the tree  is connected to a path $\pi'_{k+1}=(v_0',\dots, v'_{k+1})$ in the $k+1$th level if and only if $v_i = v_i'$ for all $i \le k$, (that is, $\pi'_{k+1}$ is the continuation of $\pi_k$). The cost of the edge between $\pi_k$ and $\pi_{k+1}$ is then set to  $C_{(v_k,v'_{k+1})}$. Observe that the cost-distance of any path $\pi_k$ from $\pi_0$ within $\saw$ equals, by construction, $|\pi_k|_{f,L}$, the cost of the path itself in $\CG(t_0)$ and in $\Ilambda$.
For $k \ge 1$, let $N_k^{t_0}(0)$ be the number of vertices in the $k$'th level of $\saw$ (i.e., the number of paths of length $k$ emanating from $0$).
	The next lemma provides an exponentially decaying bound on the expected number of such paths.
\begin{lemma}\label{lemma:decay}
Consider $\Ilambda$ under the conditions of Theorem~\ref{thm:mu-classification}(iii). Let $\beta^-$ be as in~\eqref{eq:beta2}. Let $\ve>0$ be such that $\Ev[W^{2-2\mu (\beta^--\ve)}]<\infty$. Let $t_0$ be such that 
 \begin{equation}\label{eqn:cons-t}
		\mbox{for all $x \in [0,t_0]$, $F_L(x) \le x^{\beta^--\eps}$.}
		\end{equation}
		Then, with $b_\ve:=\beta^--\ve$, for some constant $C_2<\infty$
	\be\label{eq:cons-altogether-2}
		\Ev[N_k^{t_0}(0) \mid W_0 = w_0] 
		\le w_0^{1-\mu b_\ve} ( \la C_2)^k  t_0^{ b_\ve k} \E[W^{2-2\mu b_\ve}]^{k-1}\E[W^{1-\mu b_\ve}].
	\ee
	The same remains true for $\SFPWL$ if we omit the factor $\lambda^k$ in~\eqref{eq:cons-altogether-2}.
	\end{lemma}	
Before the proof, we explain why there exist $\eps$ and $t_0$ satisfying the lemma's assumptions. By the definition of $\beta^-$, for all sufficiently small $t<1$, the inequality $\log F_L(t)/\log t\ge \beta^--\ve$ holds; 	equivalently, $F_L(t)\le t^{\beta^--\ve} $, so~\eqref{eqn:cons-t} is satisfied. Further, we argue now that $\Ev[W^{2-2\mu b_\ve}]<\infty$ for small enough $\ve>0$.
By the power-law assumption on $W$, let us write  $\Pr(W \ge x) =\ell_W(x) x^{1-\tau}$ for some slowly-varying function $\ell_W$. After integration by parts, by Karamata's theorem (see \cite[Proposition 1.5.10]{BinGolTeu89}), 
\begin{align}\label{eq:karamata}
	\E[W^{2-2\mu b_\ve}] &= \int_1^\infty w^{2-2\mu b_\ve}F_W(\mathrm{d}w) \nonumber\\
	&= \Big[{-}w^{2-2\mu b_\ve}\frac{\ell_W(w)}{w^{\tau-1}}\Big]_1^\infty + \int_1^\infty (2-2\mu b_\ve)w^{1-2\mu b_\ve} \frac{{-}\ell_W(w)}{w^{\tau-1}} \mathrm dw < \infty,
\end{align}
whenever	$2- 2\mu b_\ve-\tau  < -1$, that is, $\mu b_\ve>(3-\tau)/2$ with $b_\ve=\beta^--\ve$. So the condition that $\Ev[W^{2-2\mu (\beta^--\ve)}]<\infty$ can be fulfilled by the condition \eqref{eq:beta2} of Theorem~\ref{thm:mu-classification}(iii), by choosing $\eps$ sufficiently small relative to $\mu$, $\tau$ and $\beta^-$. Further, the finiteness of the moment $\Ev[W^{1-\mu b_\ve}]$ follows from this condition as well, since then $1-\mu b_\ve-\tau = \tfrac12(2-2\mu b_\ve-\tau) -\tfrac{\tau}{2} < -\tfrac{1}{2}-\tfrac{\tau}{2} < -1$.	\begin{proof}[Proof of Lemma \ref{lemma:decay}]	
	We shall calculate $n_k^{t_0}(0,w_0) := \Ev[N_k^{t_0}(0) \mid W_0 = w_0]$ for $\Ilambda$. For this, let $\CV_\la^{(k)} := \{(v_i)_{1\leq i\le k} \in \CV_\la^k \mid v_i \neq v_j \text{ for } 1 \leq i < j \leq k\}$ be the set of all $k$-tuples of distinct points of the Poisson point process $\CV_\la$. Then we sum over all such $k$-tuples, and use the law of total probability to integrate over their possible weights as follows:
		\be\ba \label{eq:nkt}
		 n_k^{t_0}(0,w_0)= \E\!\Bigg[\sum_{(v_i)_{i\le k} \in \CV_\la^{(k)}} \ \int\limits_{(w_i)_{i\le k}}\!\!\!\! \ind\{\forall i\!\le\! k\!:  W_{v_i}\!\!\in\! [w_i,w_i\!+\!\mathrm dw_i], v_i \leftrightarrow v_{i-1}, C_{(v_{i-1},v_{i})}\! \!\in\! [0,t_0]\}\Bigg],
	\ea\ee
	where we consider $\mathrm d w_i$ to be infinitesimal. Observe that the event in the indicator is
	\[ \calE_1(v_1, \dots, v_k) \cap \calE_2(v_1, \dots, v_k)\cap \calE_3(v_1, \dots, v_k),\]
where, for any distinct fixed points $v_1, \dots, v_k \in \R^d$,
		
\be \ba		 \calE_1(v_1, \dots, v_k)&:=\{W_{v_i} \in [w_i,w_i+\mathrm dw_i]\ \forall  i \in [k]\};\\
 \calE_2(v_1, \dots, v_k)&:=\{ v_{i}\leftrightarrow v_{i-1} \ \forall i \in [k]\};\\
\calE_3(v_1, \dots, v_k)&:=\{C_{(v_{i-1},v_i)} \in [0,t_0]  \ \forall i \in [k]\}.
	\ea\ee
	Then as $\mathrm dw_1, \dots, \mathrm dw_k \to 0$, using $F_W(\mathrm dw_i)$ to denote the Lebesgue-Stieltjes integral with respect to the cdf $F_W$ as before,
	\begin{align}
		\Pr(\calE_1(v_1, \dots, v_k)) &\to  \prod_{i=1}^k F_W(\mathrm dw_i),\nonumber\\
		\Pr(\calE_2(v_1, \dots, v_k) \mid \calE_1(v_1, \dots, v_k)) &\to \prod_{i=1}^k h_{\mathrm I}(\|v_i-v_{i-1}\|, w_{i-1}, w_i ),\label{eqn:cons-factors}\\
		\Pr(\calE_3(v_1, \dots, v_k) \mid \calE_1(v_1, \dots, v_k) \cap \calE_2(v_1, \dots, v_k)) &\to \prod_{i=1}^k F_L(t_0w_{i-1}^{-\mu} w_i^{-\mu}).\nonumber
	\end{align}
	Now we use the upper bound on $h_{\mathrm I}$ in Assumption \eqref{assu:h}. The upper bound for the case $\alpha = \infty$ is stronger (i.e., less) than the upper bound for any $\alpha< \infty$. Therefore, we may assume that $\alpha <\infty$, and any upper bound that we obtain for $n_k^{t_0}(0,w_0)$ in this way will also hold for $\alpha =  \infty$. Hence, we pick $\alpha < \infty$ and bound 
	\be \ba \label{eqn:cons-bigeq}
		 n_k^{t_0}(0,w_0)&\leq \E\Bigg[\sum_{(v_i)_{i\le k} \in \CV_\la^{(k)}} \int\limits_{(w_i)_{i\le k}} \prod_{i=1}^k \bigg(\overline c\Big(1\wedge\frac{w_{i-1}w_i}{\norm{v_i-v_{i-1}}^d} \Big)^\alpha F_L(t_0w_{i-1}^{-\mu} w_i^{-\mu}) F_W(\mathrm dw_i) \bigg)\Bigg]\\
		&=\,\int\limits_{(w_i)_{i\le k}} \prod_{i=1}^k F_L(t_0w_{i-1}^{-\mu} w_i^{-\mu})  \E\Bigg[\sum_{(v_i)_{i\le k} \in \CV_\la^{(k)}} \prod_{i=1}^k \overline c\Big(1\wedge\frac{w_{i-1}w_i}{\norm{v_i-v_{i-1}}^d} \Big)^\alpha\Bigg]\prod_{i=1}^k F_W(\mathrm dw_i).
\ea\ee
Note that the sum within the expectation is over distinct points in $\R^d$ (that are part of the PPP $\CV_\la$):
Using Campbell's formula (see e.g.\ \cite{Matt96}), this expectation can be written as 
	\be \label{eq:t1}
		T_1:=\E\Bigg[\sum_{(v_i)_{i\le k} \in \CV_\la^{(k)}} \prod_{i=1}^k \overline c\Big(1\wedge\frac{w_{i-1}w_i}{\norm{v_i-v_{i-1}}^d} \Big)^\alpha\Bigg] = \int_{(v_i)_{i\le k}}\prod_{i=1}^k\overline c\Big(1\wedge\frac{w_{i-1}w_i}{\norm{v_i-v_{i-1}}^d} \Big)^\alpha dM_k,
	\ee
	where $M_k$ is the $k$th factorial moment measure of the point process. Writing $\nu$ 
	for the standard measure of the point process, i.e.\ Lebesgue measure on $\R^d$, $M_k$ is dominated from above\footnote{In other words, conditioning on points being present only decreases the local density.} by $\la^k \nu^k$, and the term in the integral is non-negative. Thus
	\be
		T_1\le 
		\int\limits _{(v_i)_{i\le k}}\prod_{i=1}^k\overline c\Big(1\wedge\frac{w_{i-1}w_i}{\norm{v_i-v_{i-1}}^d} \Big)^\alpha\la^k \mathrm d\nu^k
		= \la^k\overline c^k\prod_{i=1}^k\bigg( \int_{v_i\in \R^d}\Big(1\wedge\frac{w_{i-1}w_i}{\norm{v_i}^d} \Big)^\alpha\,\mathrm d\nu\bigg),\label{eqn:cons-bigeq2} 
\ee
where we used the translation invariance of the Lebesgue measure to obtain the last step.	Observe that here $w_{i-1}$ and $w_i$ are constants, so the $i$th factor $T_{1i}$ on the rhs can be calculated by splitting the integral according to the value of the minimum. With $V_d$ denoting the volume of the Euclidean ball of radius $1$ in $\R^d$, for all $i \in [k]$,
	\be \ba
		T_{1i}&\le  \int_{\norm{v}^d \le w_{i-1}w_i} 1\, \mathrm d\nu + \int_{\norm{v}^d > w_{i-1}w_i} (w_{i-1}w_i)^\al/\norm{v}^{d\alpha} \mathrm d\nu \\ \label{eqn:cons-bigeq3}
		&\le V_d w_{i-1}w_i + (w_{i-1}w_i)^\al \int_{r> (w_{i-1}w_i)^{1/d}} r^{-d\alpha} r^{d-1}\mathrm dr= \big(V_d+\tfrac{1}{d(\al-1)}\big) w_{i-1}w_i,
	\ea\ee
	where we have used that $\al>1$ and hence the integral is finite.
	Using this value for the $i$'th term in~\eqref{eqn:cons-bigeq2}, with $C_2:= \overline c(V_d+1/d(\al-1))$, we obtain that $T_1$ is at most
	\be\label{eq:t1bound}
		T_1\le \la^k C_2^k \prod_{i=1}^k(w_{i-1}w_i) =( \la C_2)^k w_0\cdot \prod_{i=1}^{k-1}w_i^2\cdot w_k.
	\ee
	We continue bounding ~\eqref{eqn:cons-bigeq}. Observe that by~\eqref{eqn:cons-t}, the first product  in \eqref{eqn:cons-bigeq} is bounded by 
	\be\label{eq:fl}
		\prod_{i=1}^k F_L(t_0w_{i-1}^{-\mu}w_i^{-\mu}) \le \prod_{i=1}^k \big(t_0w_{i-1}^{-\mu}w_i^{-\mu}\big)^{b_\ve} = t_0^{kb_\ve}w_0^{-\mu b_\ve} \cdot \prod_{i=1}^{k-1} w_i^{-2\mu b_\ve} \cdot w_k^{-\mu b_\ve}.
	\ee
	Combining this with the bound on $T_1$ from \eqref{eq:t1bound}, we arrive at
	\begin{align}\nonumber
		n_k^{t_0}(0,w_0) &\le   ( \la C_2)^k t_0^{k b_\ve}w_0^{1-\mu b_\ve}\int_{(w_i)_{i\le k}} w_k^{1-\mu b_\ve}\prod_{i=1}^{k-1} w_i^{2-2\mu b_\ve}\cdot \prod_{i=1}^k F_W(\mathrm dw_i)\\\label{eq:cons-altogether3}
		&=( \la C_2)^kw_0^{1-\mu b_\ve}  t_0^{k b_\ve} \E[W^{2-2\mu b_\ve}]^{k-1}\E[W^{1-\mu b_\ve}].
	\end{align}
	This is precisely the required bound, finishing the proof of the lemma for $\Ilambda$. 
	
		For $\SFPWL$, we need to replace the integral over $k$-tuples of distinct points of a Poisson point process by the integral over $k$-tuples of distinct points over the Dirac measure of the grid. However, the calculations remain the same, except that the factor $\lambda^k$ disappears. In particular, we upper-bound the integral over $k$-tuples of \emph{pairwise} distinct points by the integral over any $k$-tuples such that any two \emph{consecutive} points are distinct. In this way,~\eqref{eqn:cons-bigeq2} becomes
		\be
		T_1\le  \overline c^k\prod_{i=1}^k\bigg( \int_{v_i \neq 0}\Big(1\wedge\frac{w_{i-1}w_i}{\norm{v_i}^d} \Big)^\alpha\,\mathrm d\nu'\bigg),\label{eqn:cons-bigeq4} 
\ee
where $\nu'$ denotes the Dirac measure of the grid. This gives the same bound as before, up to constant factors, and up to the factor $\la^k$.
	\end{proof}
The proof of Lemma  \ref{lemma:decay} makes it possible to show that the condition $\Pv(N_1^t<\infty)=1$ in Lemmas~\ref{lem:opt-real} and~\ref{lem:short-expl} is satisfied. 
		\begin{claim}[Truncated cost-degree is finite a.s.]\label{claim:costdeg} Consider $\Ilambda$ or $\SFPWL$ under the conditions of Theorems \ref{thm:mu-classification}(iii) and \ref{thm:poly-classification}(iii). For the penalty function $f_\mu=(w_1w_2)^\mu$,  $\Pv(N_1^t(0) <\infty)=1$ for all $t\ge 0$ whenever $\beta^->(3-\tau)/2\mu$. For the general polynomial penalty function $f$ of \eqref{eq:pol}, $\Pv(N_1^t(0) <\infty)=1$ holds  for all $t\ge 0$, when there is an $i\in \CI$ such that $\beta^->(2-\tau)/\nu_i$. 		
		\end{claim}

		\begin{proof}
		We show the statement by showing that $\Ev[N_1^t(0) ]<\infty$.		Observe that to show that $N_1^t(0)$ is finite for a given penalty function $f$, it is enough to show the same for a function $g\le f$. For $f_\mu$ we use the lower bounding function $w_2^\nu$ with $\nu := \mu$, and we observe that $\beta^- > (3-\tau)/2\mu > (2-\tau)/\nu$ by~\eqref{eq:compare-conditions}. For a general polynomial function $f$, we set $\nu:=\nu_i$, where $i$ is the index satisfying $\beta^->(2-\tau)/\nu_i$, and again we use the lower bounding function $w_2^\nu$, by observing that constant pre-factors do not change the qualitative behavior.
We will modify the proof of Lemma \ref{lemma:decay}. The difference from the proof of Lemma \ref{lemma:decay} is that we cannot assume that $t\le t_0$, where $t_0$ is from \eqref{eqn:cons-t}.
		Nevertheless most of the calculations carry through; only $F_L(t w_0^{-\mu} w_1^{-\mu})$ should be replaced by $F_L(t w_1^{-\nu})$. Recall that $n_1^{t}(0,w_0)=\Ev[ N_1^{t}(0) | W_0=w_0)]$.
		Following the calculations from \eqref{eq:nkt} to \eqref{eqn:cons-bigeq}, we see that the calculation for $k=1$ simplifies to
		\be \ba\label{eq:k=1trunc}
		 n_1^{t}(0,w_0)&\le \int\limits_{w_1} F_L(t w_1^{-\nu})  \E\Bigg[\sum_{v_1 \in \CV_\la}  \overline c\Big(1\wedge\frac{w_{0}w_1}{\norm{v_1}^d} \Big)^\alpha\Bigg] F_W(\mathrm dw_1).
		 \ea\ee
		The inner expectation, denoted by $T_1$, is bounded from above in \eqref{eq:t1} -- \eqref{eqn:cons-bigeq3} by 
		 \be \label{eq:k=12} \E\left[\sum_{v_1 \in \CV_\la}  \overline c\bigg(1\wedge\frac{w_{0}w_1}{\norm{v_1}^d} \bigg)^\alpha\right]\le \la \overline c \Big(V_d + \frac{1}{d(\al-1)}\Big) w_0w_1=:\la C_2 w_0 w_1.\ee	
		  Now we need to deviate somewhat from the calculation done for the arbitrary-$k$ case, since we can only apply the bound $F_L(x)\le x^{b_\ve}$ of \eqref{eqn:cons-t} when $x \le t_0$. So for $w_1 > (t/t_0)^{1/\nu}$ we can still apply the bound of~\eqref{eqn:cons-t}, and when $1 \le w_1 \le (t/t_0)^{1/\nu}$ we simply bound $F_L(t w_1^{-\nu}) \le 1$. Thus from~\eqref{eq:k=1trunc} and~\eqref{eq:k=12} we obtain
		\be  \label{eq:k=13trunc} n_1^{t}(0,w_0)\le \la C_2 w_0 \int\limits_{1 \le w_1\le (t/t_0)^{1/\nu}} w_1 F_W(\mathrm dw_1)
	 + \la C_2 w_0 \int\limits_{w_1\ge (t/t_0)^{1/\nu} \vee1} t^{b_\ve}  w_1^{-\nu b_\ve}w_1 F_W(\mathrm dw_1)
	\ee	 
	The first term on the rhs is bounded from above by $\la C_2 w_0(t/t_0)^{1/\nu}$, while the integral in the second term is bounded from above by the integral on the whole of $[1, \infty)$, yielding the bound $\la C_2 w_0t^{b_\ve}\Ev[W_1^{1-\nu b_\ve}]$. Observe that this moment is finite when $1-\nu b_\ve<\tau-1$, yielding the condition $\beta^->(2-\tau)/\nu$.
Thus $N_1^{t}(0)$ is a.s.\ finite conditioned on $W_0=w_0$. Since conditionally on $W_0=w_0$, the presence of the edges are independent indicators, by the Borel-Cantelli lemma, this means that $N_1^{t}(0)\mid W_0=w_0$ is an a.s.\! finite variable for each $w_0$. 
	 \end{proof}
			
			 \begin{proof}[Proof of Theorem~\ref{thm:mu-classification}(iii)]
	The same proof works for both $\Ilambda$ and $\SFPWL$. By Claim \ref{claim:costdeg} the condition $\Pv(N_1^{t}(0) <\infty)=1$ for all $t\ge 0$ holds. This implies by Lemma \ref{lem:opt-real} that the explosion time is realised via infinite paths, i.e., sideways explosion is excluded. Further, by Lemma \ref{lem:short-expl}, when the model is (lengthwise) explosive, for all $t_0>0$, with strictly positive probability there is an infinite path with total cost at most $t_0$.  
	So, to show that a model is conservative, it suffices to show that for a suitably-chosen $t_0 < 1$, the probability of having an infinite path with total cost in the interval $[0,t_0]$ is zero. 
	For this latter statement, it is enough to show the stronger statement that a.s.\ there is no infinite path $\pi$ starting from $0$ that uses only edges $e$ with $C_{e}\le t_0$. 
	Hence, to show that the model is conservative, it is enough to show that a.s.\ there is no infinite path in $\CG(t_0)$ (defined before Lemma~\ref{lemma:decay}). 
	
	Recall that  $N_k^{t_0}$ counts the number of $k$-edge paths in $\CG(t_0)$ emanating from $0=v_0$,
	and recall the bound~\eqref{eq:cons-altogether-2}.
	Choose $t_0:=(2 \la  C_2 \Ev[ W^{2-2\mu b_\ve}])^{-1/b_\ve}$, so that~\eqref{eq:cons-altogether-2} implies $\E[N_k^{t_0}]$ is exponentially decaying in $k$, 
	and define the events $E_k:=\{N_k^{t_0}(0)\ge 1 \}$. For any $w_0 \geq 1$ we can apply Markov's inequality to see that $\sum_{k\ge 1}\Pv(E_k \mid W_0=w_0) \leq C\sum_{k} 2^{-k}< \infty$, where $C>0$ denotes some constant. Then the Borel-Cantelli lemma tells us that a.s.\ there exists $k_0$ such that for all $k \ge k_0$, $E_k$ does not occur. This means that $N_k^{t_0}(0)=0$ a.s.\ for all $k\ge k_0$, and hence there is a.s.\ no infinite path in $\CG({t_0})$. Consequently, the model is a.s.\ conservative. 
		\end{proof}

\subsection{Extension to other penalty functions: conservative case}\label{sec:extension-conservative}
In this section we prove Theorem~\ref{thm:poly-classification}(iii). Recall that $f(w_1, w_2)$ stands for a polynomial of two variables (see \eqref{eq:pol}), with degree $\deg(f)$ defined in \eqref{eq:poldeg}. 

\begin{proof}[Proof of Theorem~\ref{thm:poly-classification}(iii)]
We may assume that the index $i=1$ is the one for which $\deg(f)$ is achieved and such that $\beta^- \nu_i>2-\tau$. Then there exists a constant $a>0$ such that for all inputs $w_1, w_2$, 
\[ f(w_1, w_2)\ge a w_1^{\mu_{1}} w_2^{\nu_{1}}=:a f_2(w_1, w_2). \]
Hence it is enough to show conservativeness for the penalty function $f_2$. 
We first describe how the proof of Lemma~\ref{lemma:decay} needs to be changed. In the last line of \eqref{eqn:cons-factors}, the two exponents $\mu$ should be replaced by $\mu_1$ and $\nu_1$, i.e., we obtain
\be
		\Pr(\calE_3(v_1, \dots, v_k) \mid \calE_1(v_1, \dots, v_k) \cap \calE_2(v_1, \dots, v_k)) \to \prod_{i=1}^k F_L(t_0w_{i-1}^{-\mu_1} w_i^{-\nu_1}).
\ee
This carries through Equations~\eqref{eqn:cons-bigeq} and \eqref{eq:fl}, where the rhs of the latter becomes  $t_0^{kb_\ve}w_0^{-\mu_1 b_\ve} \cdot \prod_{i=1}^{k-1} w_i^{-(\mu_1+\nu_1) b_\ve} \cdot w_k^{-\nu_1 b_\ve}$. Therefore,~\eqref{eq:cons-altogether3} becomes
	\begin{align*}
		n_k^{t_0}(0, w_0) \le ( \la C_2)^kw_0^{1-\mu_1 b_\ve}  t_0^{k b_\ve} \E[W^{2-(\mu_1+\nu_1) b_\ve}]^{k-1}\E[W^{1-\nu_1 b_\ve}].
	\end{align*}
The condition $\deg(f) = \mu_1+\nu_1 > (3-\tau)/\beta^-$ ensures that $2-(\mu_1+\nu_1)b_\ve <\tau-1$, so as before $\E[W^{2-(\mu_1+\nu_1) b_\ve}] <\infty$. Further, $\Ev[W^{1-\nu_1b_\ve}]<\infty$ holds when $1-\nu_1 b_\ve<\tau-1$, corresponding to the assumption that $\nu_1 \beta^- >2-\tau$ for $\tau<2$, and $W^{1-\nu_1 b_{\ve}} \leq W$ has finite expectation whenever $\tau>2$. So both expectations are finite under the conditions of Theorem~\ref{thm:poly-classification}(iii). Now the same argument as in the proof of Theorem~\ref{thm:mu-classification}(iii) shows that if $t_0$ is sufficiently small then a.s.\ $N_k^{t_0}(0) = 0$ for all $k \geq k_0$. Hence, the model is a.s.\ conservative.
\end{proof}

	\section{Existence of the giant component}\label{s:giant}
	In this section we focus on the model $\GIRG$, and provide a proof that under Assumptions \ref{assu:mild-connect} and \ref{assu:mild-vertex} it has a unique linear-sized giant component $\CC_{\max}$ (Theorem \ref{thm:giant}). We emphasise again that  Assumption \ref{assu:mild-connect} is a weaker assumption than what was assumed in the literature before \cite{BodFouMul15, BriKeuLen19, FouMul18, HeyHulJor17}, hence earlier techniques do not carry through. We will use the scaled version $\BGIRG$ introduced in Definition~\ref{def:bgirg}. 

 The first step of the proof of Theorem \ref{thm:giant} is to recall box-increasing paths from before Lemma \ref{lem:highw-connect}. In particular, 
recall that $\CV_{n,M}:=\{i\in [n] \mid W_i \geq \e^M\}$, we say that $u\in \CV_{n,M}$ is \emph{successful} when there is a box-increasing path from $u$ to a $\delta$-good leader in $\Gamma_{k^\star}(u) $ and $F_{k^\star}^{(1)}(u)$ also holds (see \eqref{eq:successful}), and we write $\CS_u$ for the event that $u$ is successful.
Lemma \ref{lem:highw-connect} shows that for each $u \in \CV_{n,M}$, the probability that $u$ is not successful is at most $5\exp( -(\underline c \wedge 1) 2^{-d-7}   \e^{M\eta(\delta) })$ for some $\eta(\delta)>0$. This probability is at most $1/2$ if we choose $M$ large enough. Thus, in expectation, linearly many vertices in $\CV_{n,M}$ will be successful, and hence reach a $\delta$-good leader in $\Gamma_{k^\star(n,M)}$. By $\eqref{eq:kstar}$, and the definition of being $\delta$-good (see \eqref{eq:delta-good}), the weight of these leaders, for some $s\in[0,1)$, falls in the interval 
\be \label{eq:interval-k} \Big(n^{C^{-s}\frac{1}{D} \frac{1-\delta}{\tau-1} }, n^{C^{-s}\frac{1}{D} \frac{1+\delta}{\tau-1} }\Big]\subseteq \Big(n^{\frac{1}{CD} \frac{1-\delta}{\tau-1} }, n^{ \frac{1+\delta}{\tau-1} }\Big]=I_{\mathrm{end}} ,\ee 
see \eqref{eq:iend}, and these are among the highest weight vertices in the whole box $\mathcal X_d(n)$.
 In the next step we study the graph formed by the vertices of highest weight in a box. The next lemma asserts that the probability that these vertices form a connected graph tends to one as the expected number of vertices in the box tends to infinity.

Let $\mathrm{ER}_{n,p}$ denote an Erd{\H o}s-R\'enyi random graph on $n$ vertices with edge probability $p$. Recall $l(w)=\exp(-c_2 (\log w)^\gamma)$, where $c_2>0, \gamma\in(0,1)$ are taken from Assumption \ref{assu:h}.
 \begin{claim}\label{claim:ER}
Consider the model $\BGIRG$, satisfying Assumptions~\ref{assu:mild-connect} and~\ref{assu:mild-vertex} and with $2 < \tau < 3$. 
Fix $M>0$ and define constants $C,D,\delta$ as in~\eqref{eq:parameters}, with the additional condition $\delta < (3-\tau)/(\tau +1)$. Let $\mathrm{B} \subseteq \calX_d(n)$ be a box of side length $r$. Let $\mathrm{Core}_{\mathrm{B}} =(\CV_{\mathrm{Core}}(\mathrm{B}), \CE_{\mathrm{Core}}(\mathrm{B}))$ be the graph spanned by vertices in $\mathrm{B}$ with weights in the interval 
\be \label{eq:ir} I_r := \Big((r^d)^{(1- \delta)/(DC(\tau-1))},(r^d)^{(1+ \delta)/(\tau-1)}\Big].\ee Then, for all large enough $r$, 
\be\label{eq:ER-domin} \mathrm{Core}_{\mathrm B}\  {\buildrel d \over \ge }\  \mathrm{ER}_{|\CV_{\mathrm{Core}}(\mathrm{B})|, q_r}, \ee
with $q_r:=\exp(-2 c_2 (\log r^d)^\gamma (\tfrac{1+\delta}{\tau-1})^\gamma)$. 
Moreover,  $\mathrm{Core}_B$ is whp connected as $n$ and $r$ tend to infinity; this remains true conditioned on $|\CV_n\cap \mathrm{B}|$, as long as the conditioned value is within a constant factor of $\mathrm{Vol}(\mathrm B)$.
\end{claim}
We defer the proof to Appendix \ref{s:app}, since its proof uses the same method as that of Lemma \ref{lem:bipartite}.
 \begin{claim}\label{claim:giant}
 Consider the model $\BGIRG$, satisfying Assumptions~\ref{assu:mild-connect} and~\ref{assu:mild-vertex} and with $2 < \tau < 3$.
 Fix a constant $M>0$ and define constants $C,D,\delta$ as in~\eqref{eq:parameters}; if $M$ is sufficiently large and $\delta$ is sufficiently small, the following holds. 
Let $\CV_{n,M,\CS}$ be the set of successful vertices $u\in \CV_{n,M}$, i.e.,
\[ \CV_{n,M,\CS}:=\{u \in \CV_{n}: W_u\ge \e^M, \ind_{\CS_u}=1\}.\]
Whp, the graph induced by $\CV_{n,M,\CS}$ is connected. Moreover, $\E[|\CV_{n,M,\CS}|]\ge c_M n$, for some constant $c_M>0$. 
 Consequently, with probability at least a positive constant, the component of $\BGIRG$ containing $\CV_{n,M,\mathcal{S}}$ has linear size. 
\end{claim}
 \begin{proof}
 We start by observing that whp the high-weight vertices in $\BGIRG$ span a connected subgraph.
Indeed, apply Claim \ref{claim:ER} to the whole space, taking $\mathrm{B} = \calX_d(n)$ and $r = n^{1/d}$, to see that the graph $\mathrm{Core}_{\mathcal X_d(n)}$ is whp connected. Whp there are no vertices of weight larger than $n^{(1+\delta)/(\tau-1)}$ by Assumption~\ref{assu:mild-vertex}, so $\mathrm{Core}_{\mathcal X_d(n)}$ contains precisely the vertices with weight larger than $n^{(1-\delta)/(DC(\tau-1))}$. 

Now let $u, v$ be two vertices with $\ind_{\CS_u}=\ind_{\CS_v}=1$.
Then, by \eqref{eq:interval-k}, all the $\delta$-good leaders in $\Gamma_{k^\star}(u)$ and $\Gamma_{k^\star}(v)$ are contained in $\mathrm{Core}_{\mathcal X_d(n)}$.
 Hence, all these leaders are in the same connected component.  On the other hand, since the events $\CS_u$ and $\CS_v$ occur, there is a path from $u$ to a $\delta$-good leader in $\Gamma_{k^\star}(u)$, and from $v $ to a $\delta$-good leader in $\Gamma_{k^\star}(v)$, so $u$ and $v$ are also in the same connected component.  
This shows connectedness of the graph induced by $\CV_{n,M,\CS}$.

 We next show that $\E[|\CV_{n,M,\CS}|]$ is linear in $n$. 
 By Lemma \ref{lem:highw-connect}, each vertex with weight in $[\exp(M), \exp(MC^{k^\star}\tfrac{1-\delta}{\tau-1})]$ is successful with probability at least $1-5\exp\big( -(\underline c \wedge 1) 2^{-d-7}   \e^{M\eta(\delta) }\big) \ge 1/2$ if $M$ is sufficiently large. 
So by linearity of expectation, \[ \Ev[|\CV_{n,M,\CS}|]\ge n\Pv\Big(W_u^{(n)} \in [\exp(M), \exp(MC^{k^\star}\tfrac{1-\delta}{\tau-1})]\Big)/2=:nq_2.\]
This is linear in $n$, since 
\[ \ba  \Pv\big(W_u^{(n)} \in [\,\exp(M), \exp(MC^{k^\star}\tfrac{1-\delta}{\tau-1})]\big)&= \underline l(\e^M) \e^{-M(\tau-1)} - \underline l(\e^{M C^{k^\star}\tfrac{1-\delta}{\tau-1}}) \e^{-MC^{k^\star}(1-\delta)}\\
&\ge  \e^{-M(\tau-1)-\ve}\ea \] for all $\ve>0$ and $n$ sufficiently large by Assumption~\ref{assu:mild-vertex}, Potter's bound, and by the fact that the second term on the rhs of the first line is negligible compared to the first one by \eqref{eq:kstar}\cbb.

  For brevity, let us denote by $\mathcal C$ the connected component containing the first vertex of largest weight (under some arbitrary ordering of $\CV_n$).
   By the previous argument,  $\mathcal C$ contains all of $\CV_{n,M,\CS}$ whp. 
 Then the above argument  shows that $\E[|\mathcal C|] \geq n q_2$. This implies $\Pr(|\mathcal C| > nq_2/2) \geq q_2/2$, since otherwise 
 \[
 	\E[|\mathcal C|] \leq n \Pr(|\mathcal C| > nq_2/2) + (nq_2/2) \cdot \Pr(|\mathcal C| \leq nq_2/2) < n q_2/2 + nq_2/2=nq_2
 \]
 would lead to a contradiction. In other words, we have shown that with at least \emph{a constant probability}, $\BGIRG$ (and hence $\GIRG$) contains at least one linear-sized component.
  \end{proof}
 Claim~\ref{claim:giant}   ensures that with strictly positive probability $\BGIRG$ contains a linear-sized component. The proof of Theorem~\ref{thm:giant}, which says that \emph{with high probability} there is a \emph{unique} giant component of linear size, runs along the same lines but requires a bit more care and the use of Claim~\ref{claim:giant}. 
 \begin{proof}[Proof of Theorem \ref{thm:giant}]
 We first show that a linear-sized component exists \emph{whp}. Choose constants $M,\delta >0$; we will require $M$ to be sufficiently large and $\delta$ sufficiently small, but we will not specify their exact values. Define $C,D$ as in~\eqref{eq:parameters}. We condition throughout on the event that every vertex has weight at most $M_n$; by Assumption~\ref{assu:mild-vertex}, this event occurs whp and implies that the weight of every vertex independently follows a distribution $\Pr(W \ge x) = \ell^{(n)}(x)x^{-(\tau-1)}$, where $\llow(x) \le \ell^{(n)}(x) \le \lhigh(x)$ for some functions $\llow$ and $\lhigh$ which vary slowly at infinity. Recall the definition of successful vertices from \eqref{eq:successful}, define $\eta(\delta)$ as in Lemma~\ref{lem:highw-connect}, and let $\wih\eta:=\eta(\delta)\wedge 1$. Let $\wih{w} := (\log n)^{2/\wit \eta}$.
 
 By Lemma~\ref{lem:highw-connect}, every vertex with weight at least $\wih{w}$ is successful with probability at least $1 - 5\exp(-(\underline c \wedge 1) 2^{-d-7}\wih{w}^\eta) \ge 1 - 1/n^2$ for $n$ sufficiently large, since $\eta(\delta)\cdot 2/\wih \eta\ge 2$. So by a union bound over all vertices in $[n]$, whp every vertex with weight at least $\wih{w}$ is successful. By Claim~\ref{claim:giant}, whp these vertices all lie in the same connected component $\mathcal{C}$; denote this event by $\mathcal{E}_1$. 
 
 We will now partition most of $\calX_d(n)$ into ``medium-sized boxes'' $(\wih {\mathrm B}_i)_{i\le \wih m_n}$, whose sizes are chosen to ensure that whp they each contain at least one vertex in $\mathcal{C}$ (as we prove below). Let $\wih{m}_n := \ceil{n/(2(\log n)^{3\tau/\wih \eta})}$, and let $(\wih{\mathrm{B}}_i)_{i \in [\wih{m}_n]}$ be a collection of disjoint boxes in $\calX_d(n)$, each with volume $(\log n)^{3\tau/\wih\eta}$. We call these \emph{medium-sized boxes}. For all $i \in [\wih{m}_n]$, let $\wih{Z}_i$ be the number of vertices in $\wih{\mathrm{B}}_i$ with weight at least $\wih{w}$. Then since each vertex falls into $\wih {\mathrm B}_i$ with probability $ \mathrm{Vol}(\wih {\mathrm B}_i)/n $ and has i.i.d.\! weight from distribution $W^{(n)}$, by Assumption \ref{assu:mild-vertex} and Potter's bound, when $n$ is sufficiently large 
 we have
 \be\label{eq:meanbound} \ba \Ev[\wih Z_i]&=n\Pv(W^{(n)}\ge \wih w) \mathrm{Vol}(\wih {\mathrm B}_i)/n \ge (\log n)^{3\tau/\wih\eta} \cdot \underline \ell(\wih w)\wih w^{-(\tau-1)}\\
 & \ge (\log n)^{(\tau+2)/\wih\eta}\underline \ell(\wih w) \geq (\log n)^{2/\wih\eta}.\ea\ee 
 Since $\wih{Z}_i$ is binomial, it follows by Lemma~\ref{lem:chernoff} (Chernoff bound) and a union bound that
 \begin{equation}\label{eq:0bi}
 	\Pv(\exists i\le \wih m_n: \wih Z_i =0 ) \le \wih m_n 2\exp(-\E[\wih Z_i]/3) \le n\exp(-(\log n)^{2/\wih\eta}/3) \to 0 \mbox{ as }n\to\infty.
 \end{equation}
 Thus whp, each box $\wih{\mathrm{B}}_i$ contains at least one vertex of weight at least $\wih{w}$; denote this event by $\mathcal{E}_2$.
 
 We say that a medium-sized box $\wih{\mathrm{B}}_i$ contains a \emph{local giant} if a constant proportion of its vertices lie in the same component, and moreover this component contains a vertex of weight at least $\wih{w}$.  We will show now that whp, a constant proportion of boxes contain local giants. Since the whp event $\mathcal{E}_1 \cap \mathcal{E}_2$ then implies that all these components are identical, it follows that whp there is a linear-sized component.
 
To ensure independence between medium-sized boxes, we first expose the number of vertices in each medium-sized box. For all non-negative integers $k_1,\dots,k_{\wih{m}_n}$ with $k_1 + \dots + k_{\wih{m}_n} \le n$, let $\mathcal{F}(k_1,\dots,k_{\wih{m}_n})$ be the event that $\wih{\mathrm{B}}_i$ contains exactly $k_i$ vertices for all $i \in [\wih{m}_n]$. Let $\mathcal{P}_n$ be the set of all tuples $(k_1, \dots, k_{\wih{m}_n})$ such that $k_1 + \dots + k_{\wih{m}_n} \le n$ and, for all $i \in [\wih{m}_n]$, $k_i \in [(\log n)^{3\tau/\wih\eta}/2,\ 2(\log n)^{3\tau/\wih\eta}]$. Let
 \[
 	\mathcal{F} := \bigcup_{(k_1,\dots,k_{\wih{m}_n}) \in \mathcal{P}_n} \mathcal{F}(k_1,\dots,k_{\wih{m}_n})
 \]
 be the event that for all $i \in [\wih{m}_n]$, $\wih{\mathrm{B}}_i$ contains between $(\log n)^{3\tau/\wih\eta}/2$ and $2(\log n)^{3\tau/\wih\eta}$ vertices. Since the number of vertices within each box is binomial, Lemma~\ref{lem:chernoff} and a union bound implies that for sufficiently large $n$,
 \begin{equation}\label{eqn:vertex-counts-good}
 	\Pr(\neg\mathcal{F}) \le 2\wih{m}_n \e^{-(\log n)^{3\tau/\wih\eta}/12} \le 2n\e^{-(\log n)^2} \to 0 \mbox{ as }n\to\infty.
 \end{equation}
 
To obtain independence from the global weight vector, we will subsample the edges independently, so that the probability of having an edge between two vertices $u$ and $v$ is no longer given by 
$g_n^{u,v}(x_u, x_v, (W_i^{(n)})_{i\le n} )$, but rather by the lower bound on $g_n^{u,v}$ given in \eqref{eq:blown-h} that does not depend on all the weights $(W_i^{(n)})_{i\le n} $, but only on $W_u^{(n)}, W_v^{(n)}$.  For all $i \in [\wih{m}_n]$, let $\calE_{3,i}$ be the event that $\wih{\mathrm{B}}_i$ contains a local giant after subsampling; we will bound $\Pr(\calE_{3,i} \mid \mathcal{F}(k_1,\dots,k_{\wih{m}_n}))$ below for all $(k_1,\dots,k_{\wih{m}_n}) \in \mathcal{P}_n$. Conditioned on $\mathcal{F}(k_1,\dots,k_{\wih{m}_n})$, vertices in $\wih{\mathrm{B}}_i$ are distributed uniformly, their weights are drawn independently from $W^{(n)}$, and two vertices $u$ and $v$ are joined with a probability which only depends on $x_u$, $x_v$, $W_u$ and $W_v$. Hence
\be\label{indep} \text{ Conditioned on $\mathcal{F}(k_1,\dots,k_{\wih{m}_n})$, the events $\mathcal{E}_{3,i}$ are independent of each other.} \ee 

Conditioned on $\mathcal{F}(k_1,\dots,k_{\wih{m}_n})$, form $\wih{\mathrm{B}}_i'$ by translating $\wih{\mathrm{B}}_i$ so that its center is at the origin, and rescaling it to have side length $k_i^{1/d}$; then the corresponding graph after subsampling in $\wih{\mathrm{B}}_i'$ is a realisation of a $\mbox{BGIRG}_{W,L}(k_i)$. Moreover, its value of $\tau$ is unchanged and therefore still satisfies $\tau \in(2,3)$. Its connection probability functions $g_n^{u,v}$ change by at most a factor of 2 due to the rescaling (since $(k_1,\dots,k_{\wih m_n}) \in \mathcal{P}_n$) and therefore still satisfy Assumption~\ref{assu:mild-connect} with different values of $\underline{c}$ and $\overline{c}$. Finally, its weight distribution is $W^{(n)}$, which satisfies Assumption~\ref{assu:mild-vertex} as $n\to\infty$ and hence also as $k_i\to\infty$. Thus by Claim~\ref{claim:giant}, there exists a constant $p_\star > 0$ such that
 \[
 	\Pr(\calE_{3,i} \mid \mathcal{F}(k_1,\dots,k_{\wih m_n})) \ge p_\star \mbox{ for all }(k_1,\dots,k_{\wih m_n}) \in \mathcal{P}_n.
 \]
 Since the events $\mathcal{E}_{3,i}$ are conditionally independent by \eqref{indep}, Lemma~\ref{lem:chernoff} implies that
 \[
 	\Pr(|\{i \mid \calE_{3,i}\mbox{ occurs}\}| \ge p_\star/2) \ge 1 - 2e^{-\wih{m}_np_\star/12} \to 1\mbox{ as }n\to\infty.
 \]
 Together with the fact that $\calE_1$ and $\calE_2$ also occur whp, this implies that whp a constant proportion of boxes $\wih{\mathrm{B}}_i$ contain local giants which intersect a common component $\mathcal{C}$ containing all vertices of weight at least $\wih{w}$. In particular, this implies that whp $\mathcal{C}$ is a \emph{linear-sized giant component}.\medskip
 
 It remains to prove that $\mathcal{C}$ is whp \emph{unique}. To do so, we uncover the graph in two stages. In the first stage we uncover all vertices of weight smaller than $\wih w$, and all edges between these vertices, yielding the vertex set $\CV^{(1)}$. Then, in the second stage, we uncover the vertices of weight at least $\wih w$ and all edges incident to them. Note that whp $|\CV^{(1)}| \geq n/2$. Since the vertices of the second stage are whp in the same connected component $\CC$, we only need to show that they swallow up every large component formed by $\CV^{(1)}$. So let $\eta > 0$ be a constant; we will show that if $\CV^{(1)}$ contains a component $\mathcal C'$ with $|\CC'|\ge \eta n$, then this component whp merges with $\CC$ after the second stage.  
  
Partition the space $\calX_d(n)$ into \emph{small boxes} of side length $\floor{n^{1/d}}/n^{1/d}$ (and hence of volume between $1/2$ and $1$), 
and denote these by $(\mathrm{sB}_j)_{j \in [N]}$. We say that a component $\CC'$ \emph{hits} a small box $\mathrm{sB}_j$ if there is a vertex $v\in \CC'$ with location $x_v\in \mathrm{sB}_j$.
Then, by a standard argument on Poisson processes~\cite[Lemma 6.2]{lengler2017existence}, there exists 
$\sigma = \sigma(\eta)>0$ such that whp \emph{every} subset of $\eta n$ vertices in $\CV^{(1)}$ must hit at least $\sigma n$ small boxes. In particular, conditioned on this event, if we now expose $\CV^{(1)}$ and suppose that it contains a component $\mathcal C'$ with $|\CC'| \ge \eta n$, then $\mathcal C'$ must hit at least $\sigma n$ boxes. 

Since vertices have i.i.d.\! locations, every vertex $\wih u$ of the second stage (that is, with weight $\ge\wih w$) 
 has, independently,  a probability at least $\sigma$ to lie in 
 a box with a vertex $v\in \CC'$. If this occurs, since $W^{(n)}\ge 1$ and each box has volume at most 1, the minimum in \eqref{eq:blown-h} is taken at the first term. Thus there is an edge between $\wih{u}$ and $v$ with probability at least $\underline c \cdot l(W_{\wih{u}})l(W_v)$. By Assumption~\ref{assu:mild-vertex} and a union bound, whp every vertex in $G$ has weight at most $n^{2/(\tau-1)}$. 
Thus the total probability that $\wih{u}$ sends an edge to $\CC'$ is at least $\sigma\underline c \cdot l(W_{\wih u})l(W_v) \geq \sigma\underline c\exp(-2c_2(2/(\tau-1))^\gamma\cdot ( \log n)^\gamma)$, which is at least $n^{-1/3}$ for sufficiently large $n$ (since $\gamma < 1$). 

Since $\wih{w} = (\log n)^{2/\wih\eta}$, again by Assumption~\ref{assu:mild-vertex}, the number of vertices $|\CV_2|$ in the second stage is binomial with expectation $n\Pv(W^{(n)}\ge \wih w)\ge n^{1/2}$ when $n$ is sufficiently large. We have already shown that whp these vertices all belong to $\CC$. Since each of these vertices independently has probability at least $n^{-1/3}$ of connecting to a vertex in $\CC'$, whp at least one of them does connect to a vertex from $\CC'$ by Lemma~\ref{lem:chernoff}. Therefore, whp there is at most (and hence exactly) one component of size at least $\eta n$.
   \end{proof}
   \section{Extension to finite-sized models}\label{proof-finite}
   We devote this section to the proof of Theorems \ref{thm:GIRG1} and~\ref{thm:GIRG-fine}. The proof for the special case of $\mu =0$ ($f\equiv 1$) was carried out in detail in \cite{julia-girg}, and the method for the conservative case and the lower bound for the explosive case carries over for general $f$. We therefore only provide a sketch of these parts of the proof. The proof of the upper bound uses the same `scheme' as  \cite{julia-girg}: from each of the two uniformly chosen vertices, we find a path to a vertex with large enough weight, with cost that approximates the explosion time.  Then we connect the two high-weight vertices with a low-cost path. The estimate of the cost of the connecting path requires more care than that in \cite{julia-girg}. Hence, we spell out the differences but use some lemmas from \cite{julia-girg} for the upper bound. We start with Theorem~\ref{thm:GIRG-fine}, which is the more precise result. We will then derive Theorem \ref{thm:GIRG1} as a corollary.
   
\begin{proof}[Proof of Theorem \ref{thm:GIRG-fine}] Recall from Definition \ref{def:GIRG} that $\GIRG$ has vertex-space $[-1/2,1/2]^d$. To relate $\GIRG$ on $n$ vertices to the infinite models, we map the vertex locations to $\mathcal{X}_d(n)=[-n^{1/d}/2,n^{1/d}/2]^d$ using the transformation in Definition \ref{def:bgirg}, obtaining the equivalent blown-up model $\BGIRG=(\CV_B(n), \CE_B(n))$. In $\mathrm{BGIRG}_{W, L}(n)$ the vertex-density stays constant as $n$ increases, and the number of vertices in sets of smaller volume converges to a Poisson distribution with intensity $1$. Under the extra assumption that the connection probabilities $g_n^{u,v}$ in \eqref{eq:blown-h} converge to some limiting function $h$ (more precisely, \cite[Assumption 2.4, 2.5]{julia-girg}), we can thus relate this model to an $\mathrm{IGIRG}_{W, L}(1)$ graph restricted to $\mathcal{X}_d(n)$, as follows. We find a sequence $k_n\to\infty$ such that for a fixed vertex $v$, whp the $k_n$-neighborhood (including vertex- and edge-weights) is identical in $\BGIRG$ and in $\mathrm{IGIRG}_{W, L}(1)$ under a suitable coupling. Hence, for two uniformly random vertices $v_n^1, v_n^2$, whp the costs of leaving the $k_n$-neighborhoods of $v_n^1$ and $v_n^2$ converge in distribution to the explosion time of those vertices in  $\mathrm{IGIRG}_{W, L}(1)$. Since $v_n^1,v_n^2$ are typically far away in Euclidean space ($\|v_n^1-v_n^2\|=\Theta(n^{1/d})$), the $k_n$-neighborhoods are contained in disjoint geometric parts of $\mathcal{X}_d(n)$, and the two costs become asymptotically independent. 
	
Unfortunately, the details are quite tedious, since the aforementioned perfect coupling of graphs only works locally. Globally, the total number of vertices in $\mathrm{BGIRG}_{W, L}(n)$ is exactly $n$, while it is Poisson in the $\mathrm{IGIRG}_{W, L}(1)$ model. So instead, we squeeze the vertex sets of $\mathrm{BGIRG}_{W, L}(n)$ and of $\mathrm{IGIRG}_{W, L}(1)$ between two models $\mathrm{IGIRG}_{W, L}(1-\xi_n)$ and $\mathrm{IGIRG}_{W, L}(1+\xi_n)$, for a parameter $\xi_n\downarrow 0$. More precisely, by \cite[Claim 3.3 and 3.4]{julia-girg} we can choose $\xi_n$ such that under a suitable coupling, almost surely for almost all $n$,
	\be\label{eq:vertex-containment} (\mathcal V_{1-\xi_n}\cap \calX_d(n)) \subseteq \{x_v\}_{v\in [n]} \subseteq (\mathcal V_{1+\xi_n}\cap \calX_d(n)),\ee
and $\CV_{\la_1} \subseteq \CV_{\la_2}$ and $\CE_{\la_1} \subseteq \CE_{\la_2}$ whenever $\la_1\le \la_2$. Note that the latter condition implies that $\mathrm{IGIRG}_{W, L}(1)$ is also sandwiched between the models $\mathrm{IGIRG}_{W, L}(1-\xi_n)$ and $\mathrm{IGIRG}_{W, L}(1+\xi_n)$. Moreover, \cite[Eqs.\! (5.11), (5.21), (5.22)]{julia-girg} show that there is a choice of $k_n, M_n$ both tending to infinity (sufficiently slowly) such that for two uniformly random vertices $v_n^1,v_n^2$, whp the following event $A_{k_n,M_n}$ occurs: 
\be \begin{minipage}{\textwidth-5em}$A_{k_n,M_n}$ is the event that the two boxes (wrt Euclidean distance) of radius $M_n$ around $v_n^1$
and $v_n^2$ are disjoint and contained in $\mathcal{X}_d(n)$, and the $k_n$-neighborhoods with respect
to graph distance of $v_n^1,v_n^2$ are contained in these boxes, and these $k_n$-neighborhoods coincide in all four models as vertex- and edge-weighted graphs. \label{def:akn}
\end{minipage}
\ee
	 \emph{Lower bound on $d_{f,L}(v_n^1,v_n^2)$}. 
	Recall from Def.~\ref{def:distances} that we add subscript $\la$ and $n$ to the metric balls and their boundaries when the underlying model is $\Ilambda$ and $\BGIRG$, respectively.
	Observe that when the event $A_{k_n, M_n}$ holds, any path connecting $v_n^1,v_n^2$  must intersect the boundary of the graph distance balls. Hence,	\be\label{overviewlowerbound1}
	d_{f,L}^{\mathrm{BGIRG}}\left(v_n^1,v_n^2\right)\geq  \ind_{A_{k_n,M_n}}\cdot \Big(d_{f,L}^{1}\left(v_n^1,\partial B^G_{1}\left(v_n^1, {k_n}\right)\right)+d_{f,L}^{1}\left(v_n^2,\partial B^G_{1}\left(v_n^2, {k_n}\right)\right)\Big),
	\ee  
	where $d_{f,L}^{\mathrm{BGIRG}}$ and $d_{f,L}^{1}$ are cost distances in $\BGIRG$ and $\Ione$, respectively.

From here the proof of the conservative case follows: in $\Ione$, the cost to reach the boundary of these graph distance balls tends to infinity and the result follows. 
 The lower bound of the explosive case is finished by showing that the variables on the rhs of \eqref{overviewlowerbound1} tend, in distribution, to two i.i.d.\ copies of the explosion time $Y_{f}^{\mathrm{I}}(0)$ of the origin in $\mathrm{IGIRG}_{W,L}(1)$. Intuitively, the asymptotic independence follows since conditioned on the events $A_{k_n,M_n}$ (which occur whp), the two variables are determined by the subgraphs induced by two boxes of $\calX_d(n)$, and these have the same distribution as the neighborhood of the origin by the translation invariance of the model $\Ione$. For a more detailed proof of asymptotic independence, see the arguments between \cite[Equations (3.14)--(3.19)]{julia-girg} that show that this implies that for all $x$, and for all $\ve>0$, it holds that  
 \be\label{eq:distri-lower}\Pv( d_{f,L}^{\mathrm{BGIRG}}(v_n^1,v_n^2) \le x ) \le \Pv(Y_{f}^{\mathrm{I} (1)}(0) +Y_{f}^{\mathrm{I} (2)}(0) \le x  ) +\ve,  \ee
 where  $Y_{f}^{\mathrm{I} (1)}(0), Y_{f}^{\mathrm{I} (2)}(0)$ are two i.i.d.\ copies of $Y_{f}^{\mathrm{I}}(0)$.
For the corresponding lower bound, we need an upper bound on 	$d_{f,L}^{\mathrm{BGIRG}}(v_n^1,v_n^2)$.
	
\emph{Upper bound on $d_{f,L}(v_n^1,v_n^2)$.}
For the upper bound, one can use the same coupling event $A_{k_n,M_n}$ of the $k_n$-neighborhoods in $\BGIRG$ and in $\Ione$, but now one needs to construct a path connecting $v_n^1$ and $v_n^2$ in $\BGIRG$, such that its cost is a good approximation of the sum of explosion times of $v_n^1$ and $v_n^2$ in $\Ione$. 

The first step is to show that when  $v_n^1$ and $v_n^2$ are in the giant component $\mathcal{C}_{\infty}$ of $\mathrm{BGIRG}_{W,L}(n)$, then the event $\CE_{\infty}(v_n^1, v_n^2)$ that they are in the infinite component of $\mathrm{IGIRG}_{W,L}(1)$, occurs whp. This was shown in~\cite[Lemma 3.7]{julia-girg}. Formally:
\be \lim_{n\to \infty}\Pv( \CE_{\infty}(v_n^1, v_n^2) \mid v_n^1, v_n^2 \in \mathcal{C}_{\infty} ) =1.\ee
Our goal is to show that for all $x\ge 0$ and all $\eps,\eps'>0$, there exists $n_0\in \N$ such that for all $n \geq n_0$,
\be \Pv(d_{f,L}(v_n^1, v_n^2) \le x) \ge \Pv(Y_{f}^{\mathrm{I} (1)}(0) +Y_{f}^{\mathrm{I} (2)}(0) +\eps' \le x) -\ve, \ee
where $Y_{f}^{\mathrm{I} (1)}(0),Y_{f}^{\mathrm{I} (2)}(0)$ are two i.i.d. copies of the explosion time of the origin in $\Ione$.
We will do this by first showing that 
\be\label{eq:finite-lower-bound} \Pv(d_{f,L}(v_n^1, v_n^2) \le x) \ge \Pv( Y_{f}^{\mathrm{I}}(v_n^1) +  Y_{f}^{\mathrm{I}}(v_n^2)+ \epsilon'\le x ) -  \eps.\ee 
This is then sufficient to show the counterpart of \eqref{eq:distri-lower} by an argument given between \cite[Equations (3.22)--(3.24)]{julia-girg}, combined with the statement that jointly,
\be\label{eq:distr-limit} ( Y_{f}^{\mathrm{I}}(v_n^1), Y_{f}^{\mathrm{I}}(v_n^2)) \toindis (Y_{f}^{\mathrm{I} (1)}(0), Y_{f}^{\mathrm{I} (2)}(0)),\ee
two i.i.d.\ copies of the explosion time. Rigorously, \eqref{eq:distr-limit} is \cite[Equation (3.21)]{julia-girg}, and its proof can be found in \cite[Equations (3.25)--(3.29)]{julia-girg}.
Heuristically, \eqref{eq:distr-limit} is natural. Even though the explosion time of $v_n^1$ and $v_n^2$ are dependent for  fixed $n$, their values are to a large extent determined by disjoint boxes around these vertices, in particular, the approximations 
$d_{f,L}^{1}\left(v_n^1,\partial B^G_{1}\left(v_n^1, {k_n}\right)\right)$, 
$d_{f,L}^{1}\left(v_n^2,\partial B^G_{1}\left(v_n^2, {k_n}\right)\right)$ are independent, hence the asymptotic independence follows. Thus it remains to show~\eqref{eq:finite-lower-bound}. 

From here we continue by a different argument than the one in \cite{julia-girg}, since now we need to take the weight penalty function $f$ into account. However, the idea remains the same: for each $q \in \{1,2\}$, we will find a high-weight vertex $u_n^q$ whose cost-distance to $v_n^q$ is less than the explosion time of $v_n^q$. We will then apply Lemma \ref{lem:cheap-connect} to establish a low-cost connecting path between $u_n^1$ and $u_n^2$. 

We decompose the proof into several steps. 

\medskip
 
{\bf Step 1)} \emph{Cheap path via vertices of high enough weight.} Fix $x, \eps,\eps' >0$ as in~\eqref{eq:finite-lower-bound}. We may assume $\eps <1$. 
We again work with the event $A_{k_n, M_n}$ from \eqref{def:akn}. For $q\in\{1,2\}$, in the first phase we will try to find a vertex $u_n^q$ in the $k_n$-neighborhood of $v_n^q$, that has weight $\ge K_1$ and cost-distance at most $t_q$ from $v_n^q$. (Here $K_1$ is a suitably large constant which we will define later.) We will denote the event that we succeed in finding $u_n^q$ by $\calA_q(t_q)$. This first phase will not always succeed, but we will show that
\be\label{eq:goalstar}
\Pv\big(\calA_1(t_1) \cap \calA_2(t_2) \mid Y_f^\mathrm{I}(v_n^1)\le t_1 \mbox{ and }Y_f^\mathrm{I}(v_n^2) \le t_2\big) \ge 1-\frac{\eps}{4}.
\ee
Moreover, we will show that if the first phase does succeed, then
\be \label{eq:dist-u1-u2} \Pr[d_{f,L}(u_n^1, u_n^2) \le \eps' \mid \calA_1(t_1) \cap \calA_2(t_2)] \geq 1-\frac{\eps^2}{32x\Pv\big(\calA_1(t_1) \cap \calA_2(t_2)\big)}.\ee
We first show how~\eqref{eq:finite-lower-bound} follows from~\eqref{eq:goalstar} and~\eqref{eq:dist-u1-u2}. We fix $x>0$ and define
\be\label{eq:tgood}\ba I_{\text{good}} := \Big\{t \in [0,x-\eps'] : \Pr\big(Y_f^\mathrm{I}(v_n^1)\leq t \text{ and }  Y_f^\mathrm{I}(v_n^2) \le x-\eps'-t\big) > \eps/4x \big\}. \ea\ee
Then 
\be\label{eq:tgood-matters}\ba
	\Pr\big(Y_f^\mathrm{I}(v_n^1)+  Y_f^\mathrm{I}(v_n^2) \le x-\eps'\big) & = \Pr\big(\exists t \in [0,x-\eps'] \colon Y_f^\mathrm{I}(v_n^1)\le t \mbox{ and }Y_f^\mathrm{I}(v_n^2) \le x-\eps'-t\big) \\
&	\le \Pr\big(\exists t \in I_{\text{good}} \colon Y_f^\mathrm{I}(v_n^1)\le t \mbox{ and }Y_f^\mathrm{I}(v_n^2) \le x-\eps'-t\big) \\
&	\qquad + \int_{t \in [0,x-\eps']\setminus I_{\text{good}}} \frac{\eps}{4x} \mathrm dt\\
&	\le \Pr\big(\exists t \in I_{\text{good}} \colon Y_f^\mathrm{I}(v_n^1)\le t \mbox{ and }Y_f^\mathrm{I}(v_n^2) \le x-\eps'-t\big) +\eps/4.
\ea\ee
For $t \in I_{\text{good}}$ we obtain by~\eqref{eq:goalstar},
\be\ba \label{eq:good-t}\Pv(\calA_1(t) \cap \calA_2(x-\eps'-t) ) \ge \frac34 \Pr\big(Y_f^\mathrm{I}(v_n^1)\leq t \text{ and }  Y_f^\mathrm{I}(v_n^2) \le x-\eps'-t\big) \ge \eps/8x. \ea\ee 
Plugging this into the denominator on the rhs of \eqref{eq:dist-u1-u2}, we obtain for all $t \in I_{\text{good}}$,
\be \label{eq:dist-u1-u2-simplified} \Pr[d_{f,L}(u_n^1, u_n^2) \le \eps' \mid \calA_1(t) \cap \calA_2(x-\eps'-t)] \geq 1-\eps/4.\ee
Combining this with~\eqref{eq:goalstar} and the fact that $\calA_q(t_q)$ implies $d_{f,L}(u_n^q,v_n^q) \le t_q$ by construction, we get
\be\label{eq:86-proof}\ba
	\Pv(d_{f,L}(v_n^1,v_n^2) \le x) &\ge \Pv\big(\exists t \in I_{\text{good}} \colon \calA_1(t) \cap \calA_2(x-\eps'-t) \big) - \eps/4\\
	&\ge \Pr\big(\exists t \in I_{\text{good}} \colon Y_f^\mathrm{I}(v_n^1)\le t \mbox{ and }Y_f^\mathrm{I}(v_n^2) \le x-\eps'-t\big) - \eps/2\\
	&\ge \Pr\big(Y_f^\mathrm{I}(v_n^1) + Y_f^\mathrm{I}(v_n^2) \le x-\eps'\big) - \eps,
\ea\ee
establishing~\eqref{eq:finite-lower-bound} and proving the result.

So it remains to show~\eqref{eq:goalstar} and \eqref{eq:dist-u1-u2}. We start with the latter. We will carefully define $u_n^q$ so that we can find it without revealing any information about vertices of larger weight. This allows us to apply Lemma \ref{lem:cheap-connect} to $u_n^q$, since this lemma only uses information about the number and locations of vertices with weight strictly larger than the weight of $u_n^q$. Conditioned on the existence of $u_n^q$, Lemma \ref{lem:cheap-connect} connects it via a box-increasing greedy path $\pi^{\text{greedy}}(u_n^q)$ to the core of highest-weight vertices. This path will have cost at most
\be \label{eq:cheap-K1} K_1^{-\rho(\delta)/2} + 2 K_1^{-\frac{\tau-1}{1-\delta}\xi(\delta)/2},\ee
and the construction will fail with probability at most $\Xi(K_1)$ in \eqref{eq:xi-error}. Both the cost of the path and the failure probability tend to $0$ as $ K_1 \to \infty$. We choose $K_1$ so large that the cost of the two paths (one each from $u_n^1$ and $u_n^2$) are each at most $\eps'/3$ with probability at least $1-\eps^2/96$. The end vertices $c_{\mathrm{end}}^q$ of these paths have weights as in~\eqref{eq:weight-end}, and by Lemma~\ref{lem:core}, whp all pairs of vertices with such weights can be connected by paths of cost $K(\zeta) n^{-\zeta}=o(1)$ (see \eqref{eq:core-cost}). We will choose $n$ large enough that $K(\zeta)n^{-\zeta} \le \eps'/3$ and the failure probability of Lemma~\ref{lem:core} is at most $\eps^2/(96x)$. 

To avoid dependencies between the constructions for $u_n^1$ and $u_n^2$, let $\mathcal{F}_1$ be the event that $\mathcal{A}_1(t_1)$ occurs, but there is no box-increasing path from $u_n^1$ to the high-weight core of cost at most $\eps'/3$. Likewise, let $\mathcal{F}_2$ be the event that $\mathcal{A}_2(t_2)$ occurs, but there is no box-increasing path from $u_n^2$ to the high-weight core of cost at most $\eps'/3$. Finally, let $\mathcal{F}_3$ be the event that there exists a pair of vertices in the core not joined by a path of cost $\le \eps'/3$. Then 
as argued before, the probabilities of $\mathcal{F}_1$, $\mathcal{F}_2$ and $\mathcal{F}_3$ are at most $\eps^2/(96x)$ each by Lemmas~\ref{lem:cheap-connect} and~\ref{lem:core}. On the other hand, deterministically, if $\calA_1(t_1) \cap \calA_2(t_2)$ occurs and $\mathcal{F}_1$, $\mathcal{F}_2$ and $\mathcal{F}_3$ do not occur, then we have $d_{f,L}(u_n^1,u_n^2) \le \eps'$. Thus by a union bound,
\[
	\Pv\big(d_{f,L}(u_n^1,u_n^2) \le \eps' \mid \calA_1(t_1) \cap \calA_2(t_2)\big) \ge 1 - \frac{\Pr(\mathcal{F}_1) + \Pr(\mathcal{F}_2) + \Pr(\mathcal{F}_3)}{\Pv\big(\calA_1(t_1) \cap \calA_2(t_2)\big)} \geq 1-\frac{\eps^2}{32x\Pv\big(\calA_1(t_1) \wedge \calA_2(t_2)\big)}.
\]
This proves~\eqref{eq:dist-u1-u2}, and it remains only to show~\eqref{eq:goalstar}. 

We stress again that in finding $u_n^q$, we must avoid exposing any information about the locations of, or edges incident to, vertices with weight greater than $u_n^q$. The remaining steps are devoted to finding the vertex $u_n^q$. In the following, we fix $q\in \{1,2\}$, and we condition on the whp event $\CE_{\infty}(v_n^1, v_n^2)\cap A_{k_n,M_n}$. Recall that when the event $A_{k_n,M_n}$ occurs, the graphs of $\BGIRG$ and $\Ione$ coincide around $v_n^1$ and $v_n^2$ up to graph distance $k_n$, which allows us to work with $\Ione$ instead of $\BGIRG$.

As useful notation, for any real number $w\ge 1$,  let us write $\CV_1^{\le w}, \CV_1^{> w}$ for the set of vertices with weight $\le w$ and $> w$ in $\Ione$, respectively, and $\CV_{B}(n)^{\le w}, \CV_{B}(n)^{> w}$ for the same in $\BGIRG$. 

\medskip

{\bf Step 2)} \emph{Defining truncated balls.} 
We study the \emph{truncated balls} around a vertex $v$ in $\Ione$. 
We define $d_{f,L}^{\leq w}$ as the cost-distance in the sub-graph induced by the set of vertices of weight at most $w$. Then we define a truncated ball, where we impose cost-truncation $T$ and weight-truncation $w$, as:
\be \CB_{f,L}( v,  T, w) :=\{ u \in \CV_1^{\le w}: d_{f,L}^{\leq w}(v,u) \le T  \}. \ee
For any given $T$ and $w$, we now show quickly that
\be\label{eq:finite-size-cbl} |\CB_{f,L}( v,  T, w)|<\infty. \ee 
This is a consequence of the fact that explosion is only possible via unbounded weights, which is proved in \cite[Corollary 4.2]{julia-girg}: For any $w>0$, explosion is impossible in the subgraph of $\mathrm{IGIRG}_{W,L}(1)$ restricted to vertices with weight $\le w$, thus any infinite path realising a finite total cost must leave the set $\CV_{1}^{\le w}$. Note that this carries over from the case $\mu=0$ considered in \cite{julia-girg} (i.e., $f\equiv 1$) to arbitrary polynomial penalty functions $f$, since any such $f$ takes bounded values for input weights in $[1,w]$, and thus the costs change at most by a constant factor when we replace $f\equiv 1$ by a different $f$.  This establishes~\eqref{eq:finite-size-cbl}.

Observe that if $Y_f^\mathrm{I}(v) \le T$, then deterministically we have $|\CB_{f,L}(v, T, \infty)| = \infty$, i.e.\ the ball contains infinitely many vertices when we drop the weight-restriction. Thus:
\be\label{eq:expl-implies-U}
	\mbox{If }Y_f^\mathrm{I}(v) \le T,\mbox{ then }\CB_{f,L}(v,T,\infty)\mbox{ contains vertices of arbitrarily high weight,}
\ee
since otherwise~\eqref{eq:finite-size-cbl} would be violated for some $w$. 
Finally, for all $w \geq W_{v_n^q}$, we observe that  $\CB_{f,L}( v_n^q,  T_q , w)$ is \emph{entirely} determined by the subgraph spanned by vertices of weight $\le w$.\medskip

{\bf Step 3)} \emph{The exterior of truncated cost balls.} Next we  study the edges emanating from the truncated cost ball $\CB_{f, L}(v,  T , K_1)$ in $\Ione$. For a vertex $v$ and $T>0$, consider an edge $(u_1, u_2)$ where $u_1\in \CB_{f, L}(v,  T , K_1)$, $u_2 \in \CV_1$, $W_{u_2}>K_1$, and $d_{f,L}^{\leq K_1}(v, u_1) + C_{(u_1,u_2)} \le T$. We call such vertices $u_2$ \emph{exterior} to $\CB_{f,L}(v,T,K_1)$. Note that, conditioned on the event that $Y_f^\mathrm{I}(v)\le T$, there is at least one vertex $u_2$ with this property by~\eqref{eq:expl-implies-U}. 

Motivated by this observation, we define the exterior of the truncated cost-ball by
\be\label{eq:exterior}\ba
w_\mathrm{min} &:= \min\{w \colon \exists u_2 \mbox{ exterior to }\CB(v,T,K_1)\mbox{ with weight }w\} \\
\mathrm{Ex}(v, T, K_1) &:= \{u_2 \colon u_2 \mbox{ is exterior to }\CB(v,T,K_1)\mbox{ and }W_{u_2} = w_\mathrm{min}\}.
\ea\ee
 Thus $\mathrm{Ex}(v,T,K_1)$ is a set of vertices $u_2$ with the smallest weight-value $w_\mathrm{min} > K_1$ that are reachable from $v$ via a path of cost at most $T$ through $\CB( v,  T,K_1)$. By \eqref{eq:finite-size-cbl}, the minimum is almost surely taken over a finite set, and by the observation above, this set is almost surely non-empty conditioned on the event that $Y(v)\leq T$. 
 To make an almost surely unique choice of a vertex from this set, we define $U(v, T, K_1)$ as the vertex in $\mathrm{Ex}(v, T, K_1)$ with smallest Euclidean distance to $v$, if it exists.

We have just shown that conditioned on $Y(v) \le T$, the vertex $U(v, T, K_1)$ almost surely exists. 

\medskip

{\bf Step 4)} \emph{Weight bounds on $v_n^q$ and $u_n^q$}. We require $K_1$ to be large enough that
\be\label{eq:start-vertex-error} \Pv\big( W_{v_n^q} > K_1\big) \le \ve/(256x).\ee
for each $q\in\{1,2\}$. Moreover, we choose $K_2$, $K_3$ and $K_4$ so large that for each $q\in\{1,2\}$,
\be\label{eq:euc-error} \ba 
\Pv\big( W_{U(v_n^q, t_q, K_1)} > K_2\mid U(v_n^q, t_q, K_1) \text{ exists}\big) &\le \ve/(256x), \\
\Pv\big( (\CB_{f,L}(v_n^q,t_q,K_1) \cup \mathrm{Ex}(v_n^q,t_q,K_1) )\not\subseteq B^{G}(v_n^q,K_3) \mid U(v_n^q, t_q, K_1) \text{ exists}\big) &\le \ve/(256x),\\
\Pv\big( (\CB_{f,L}(v_n^q,t_q,K_1) \cup \mathrm{Ex}(v_n^q,t_q,K_1) )\not\subseteq B^{2}(v_n^q,K_4) \mid U(v_n^q, t_q, K_1) \text{ exists}\big) &\le \ve/(256x).
\ea\ee 
Note that $K_2$, $K_3$ and $K_4$ must exist, since $|\mathrm{Ex}(v_n^q,t_q,K_1)|<\infty$ almost surely for each $q\in\{1,2\}$ and its distribution does not depend on $n$.
Observe that conditioned on the complement of the events in\eqref{eq:start-vertex-error} and~\eqref{eq:euc-error}, $W_{v_n^q} \le K_1 < W_{u_n^q} \le K_2$. Moreover, if $n$ is large enough then $k_n >K_3$ and $M_n >K_4$, so conditioned on $A_{k_n,M_n}$, the sets $\CB_{f,L}(v_n^q,t_q,K_1)$ and $\mathrm{Ex}(v_n^q, t_q, K_1)$ are contained in the $k_n$-neighborhood of $v_n^q$ and in the Euclidean ball $B^2(v_n^q, M_n)$.

Then, taking a union bound over the events of \eqref{eq:start-vertex-error} and~\eqref{eq:euc-error} and the event that $A_{k_n,M_n}$ fails, and over $q\in\{1,2\}$, we obtain
\be\label{eq:u1u2good} \ba \Pv\big( &A_{k_n,M_n} \text{ and }\forall q\in\{1,2\}: U(v_n^q, t_q, K_1) \text{ exists}, W_{v_n^q} \le K_1 < W_{U(v_n^q, t_q, K_1)} \le K_2, \\
&d^G(U(v_n^q, t_q, K_1), v_n^q) \le K_3 \mbox{ and } d^2(U(v_n^q, t_q, K_1), v_n^q) \le K_4 \mid Y_f^\mathrm{I}(v_n^1)\le t_1 \mbox{ and }Y_f^\mathrm{I}(v_n^2)\le t_2\big) \\
& \ge  \frac{1 - \ve/(16x)}{\Pr(Y_f^\mathrm{I}(v_n^1)\le t_1 \mbox{ and }Y_f^\mathrm{I}(v_n^2)\le t_2)} \ge 1-\eps/4.\ea\ee
The final inequality follows since we have $t_2 = x-\eps'-t_1$ for some $t_1 \in I_\text{good}$. We are now finally able to describe the procedure for finding $u_n^1$ and $u_n^2$.
\medskip

{\bf Step 5)} \emph{Defining $u_n^1,u_n^2$.}
To determine $u_n^1,u_n^2$, we first uncover the vertices in 
\[ \CB(v_n^q, t_q, K_1)\cap B^2(v_n^q, K_4)\]  for $q\in\{1,2\}$. By doing so, we only use information about vertices  with weight $\le K_1$ within the Euclidean ball of radius $M_n$ around $v_n^1$ and $v_n^2$. 
By gradually increasing $w$, we then reveal the exterior $\mathrm{Ex}(v_n^q, t_q, K_1)$ within the ball $B^2(v_n^q, K_4)$. 
If these sets are non-empty, then we define $u_n^q := U(v_n^q, t_q, K_1)$. Conditioned on the event $A_{k_n,M_n}$, the $k_n$-neighborhood of $v_n^q$ is contained in the ball $B^2(v_n^q, K_4)$, so our conditional lower bound~\eqref{eq:goalstar} on the probability of successfully finding $u_n^1$ and $u_n^2$ then follows immediately from~\eqref{eq:u1u2good}. Moreover, we find each vertex $u_n^q$ (if successful) by revealing only vertices with weight $\le W_{u_n^q}$. 
This establishes~\eqref{eq:goalstar} and concludes the proof.
\end{proof}

\begin{proof}[Proof of Theorem \ref{thm:GIRG1}]
The proof will follow from Theorem \ref{thm:GIRG-fine} via the following coupling arguments. For the tightness of $d_{\mu,L}(v_n^1,v_n^2)$ we will find an upper bound on the cost-distance between $v_n^1, v_n^2$, via coupling $\BGIRG$ to a model  $\underline{\mathrm{BGIRG}}_{\underline W,L}(n)$ (on the same set of vertices) that contains less edges with higher costs, and that satisfies the assumptions of Theorem \ref{thm:GIRG-fine}.

First we define the vertex-weights $(\underline W_i)_{i\in[n]}$ in  $\underline{\mathrm{BGIRG}}_{\underline W,L}(n)$. $W^{(n)}$, satisfying only Assumption \ref{assu:mild-vertex}, might not converge in distribution, while \cite[Assumption 2.4]{julia-girg} requires this. Hence, let us denote by $\underline W, \overline W$ two random variables with respective cdfs $F_{\underline W}(x):=1-\underline \ell(x)x^{-(\tau-1)}$, $F_{\overline W}(x):=1-\overline \ell(x)x^{-(\tau-1)}$. We assign to each vertex $i\in [n]$ three vertex weights in a coupled manner: For a collection of i.i.d.\ uniform $[0,1]$ variables $(U_{n,i})_{i\le n}$, we set
\[ W^{(n)}_i:=(1-F_{W^{(n)}})^{(-1)}(U_{n,i}), \qquad  \underline{W}_i:= (1-F_{\underline W})^{(-1)}(U_{n,i}); \qquad \overline W_i:=(1-F_{\overline W})^{(-1)})(U_{n,i}), \]
where $(1-F_X)^{(-1)}(y):=\inf\{t\in \R: 1-F_X(t)\le y\}$ exists and is unique for every $y\in[0,1]$ due to the monotonicity of the cdfs. 
Observe that Assumption \ref{assu:mild-vertex} ensures stochastic domination, and this coupling precisely achieves that 
\be\label{eq:weight-coupling} \Pv(\underline W_i \le W_i^{(n)}\le \overline W\mid W^{(n)}\le M_n )=1.\ee  Now we argue that the assigned three weights are `not far' from each other.
One can verify that $ \underline W_i=\underline \ell^\star(1/U_{n,i}) U_{n,i}^{-1/(\tau-1)}, \overline W_i=\overline \ell^\star(1/U_{n,i}) U_{n,i}^{-1/(\tau-1)}$ for some slowly varying functions $\underline \ell^\star(\cdot), \overline \ell^\star(\cdot)$,  by writing  $(1-F_{\underline W})(x)=\underline \ell(x)/x^{-(\tau-1)}$, switching to $z:=1/x$ and applying \cite[Theorem 1.5.12]{BinGolTeu89}. From this and \eqref{eq:weight-coupling}  it follows that for some slowly varying function $\wit \ell$; 
\be\label{eq:wiwin} \underline W_i\le W_i^{(n)}\le \wit \ell(\underline W_i) \underline W_i.\ee 
 Now we describe the edge connection probabilities in $\underline{\mathrm{BGIRG}}_{\underline W,L}(n)$. Recall the lower bound for $g_n^{u,v}$ in Assumption \ref{assu:mild-connect}, and that $l_{c_2, \gamma}(w)=\exp(-c_2 (\log w)^\gamma)$in \eqref{eq:ellw-111}. While the term $w_u w_v/\|x_u-x_v \|$ is monotone increasing, the term $l_{c_2, \gamma}(w_u)l_{c_2, \gamma}(w_v)$ is monotone \emph{decreasing} in $w_u, w_v$, so we cannot simply use the weights $\underline W_u, \underline W_v$ for a lower bound on $g_n^{u,v}$. We solve this problem as follows:
  Given a specific value of $c_2$ that holds in the lower bound for $g_n^{u,v}$ in Assumption \ref{assu:mild-connect}, choose now $c_2'$ so large that the following inequality holds for all $w\ge 1$: 
 \[  l_{c_2', \gamma }(w):=\exp(-c_2'( \log w)^\gamma) \le  \exp(-c_2 (\log (w \wit \ell(w)) )^\gamma), \]
 and then, by \eqref{eq:wiwin},  for all $i\in [n]$,
 \be\label{eq:lc2} l_{c_2',\gamma}(\underline W_i) \le l_{c_2,\gamma}(W_i^{(n)}). \ee 
 With $\mathcal W_i:=(\underline W_i, W_i^{(n)}, \overline W_i)$, 
for each possible pair of vertices $u,v\in[n]$, we set
\[ \Pv(u\leftrightarrow v \text{ in } \underline{\mathrm{BGIRG}}_{\underline W,L}(n) \mid \mathcal W_u, \mathcal W_v ) :=\un c \cdot \bigg( l_{c_2', \gamma}(\underline W_u)l_{c_2',\gamma}(\underline W_v)\wedge \Big( \frac{\underline W_u \underline W_v}{n \|x_i-x_j\|^{d}}\Big)^\alpha\bigg). \]
The rhs is less then $g_n^{u,v}(x_u, x_v, (W_i^{(n)})_{i\in [n]})$ due to \eqref{eq:wiwin} and \eqref{eq:lc2}, and satisfies \cite[Assumption 2.5]{julia-girg}.  Using coupled variables to realise edges, the edge set $\underline \CE(n)$ of $\underline{\mathrm{BGIRG}}_{\underline W,L}(n)$ is contained in $\CE(n)$ of $\GIRG$. We then use the same variables $L_e$ on edges that are present in the two models.
Finally, to upper bound the cost of edges in $\GIRG$ by that in $\underline{\mathrm{BGIRG}}_{\underline W,L}(n)$, we use a slightly increased penalty function $\overline f$. We write $f(w_1, w_2)=\sum_{i\in \CI} a_i w_1^{\mu_i} w_2^{\nu_i}$ as in~\eqref{eq:pol}.  Given $\beta^+<(3-\tau)/\deg(f)$ from \eqref{eq:beta1}, we set $\overline f(w_1, w_2):=c_4\sum_{i\in \CI} a_i w_1^{\mu_i+\epsilon} w_2^{\nu_i+\epsilon}$, where $\epsilon>0$ is such that $\beta^+<(3-\tau)/\deg(\overline f)$ still holds, and $c_4>0$ is such that $\overline f(\underline W_u, \underline W_v)\ge f(W_u^{(n)}, W_v^{(n)})$ holds for all $u,v\in [n]$. This is possible due to \eqref{eq:wiwin} and Potter's bound.

Then every path in $\underline{\mathrm{BGIRG}}_{\underline W,L}(n)$ has higher cost wrt penalty function $\overline f$ and weights $\underline W$ than the same path in $\GIRG$, wrt penalty function $f$ and weights $W^{(n)}$.
Hence, under this coupling,
\[ d_{f,L}(v_n^1, v_n^2)\le \underline d_{\overline f, L}(v_n^1, v_n^2), \]
where $\underline d$ means distance in $\underline{\mathrm{BGIRG}}_{\underline W,L}(n)$. The proof is finished by applying Theorem \ref{thm:GIRG-fine} on the rhs to see that it converges in distribution. Hence, the lhs  is tight. 

The proof of \eqref{eq:tight-2} is along analogous lines, one now needs a model $\overline {\mathrm{BGIRG}}_{\overline W,L}(n)$ with edge set $\overline \CE(n)$, where vertices have weight $\overline W_i$,  and connection probability that is the upper bound in Assumption \ref{assu:mild-connect}. Then, $\CE(n)\subset \overline \CE(n)$. Finally, given $\beta^-$ in \eqref{eq:beta2}, one uses the penalty function $\underline f(w_1, w_2):=c_5 \sum_{i\in \CI} a_i w_1^{\min(\mu_i-\epsilon,0)} w_2^{\min(\nu_i-\epsilon,0)}$ for a small $\epsilon$ and $c_5>0$ such that $\beta^->(3-\tau)/\deg(\underline f)$ still holds and that $\underline f(\overline W_u, \overline W_v)\le f(W_u^{(n)}, W_v^{(n)})$ for all $u,v\in[n]$. Then every path in $\overline{\mathrm{BGIRG}}_{\overline W,L}(n)$ has lower cost wrt penalty function $\underline f$ and weights $\overline W$ than the same path in $\GIRG$, wrt penalty function $f$ and weights $W^{(n)}$, and there may be more paths in the earlier model, hence now $d_{f,L}(v_n^1, v_n^2)\ge \overline d_{\underline f, L}(v_n^1, v_n^2)$ holds, where $\overline d$ means distance in $\overline{\mathrm{BGIRG}}_{\overline W,L}(n)$. Theorem \ref{thm:GIRG-fine} applies to the rhs of this inequality and it tends to infinity with $n$, finishing the proof.
\end{proof}

\begin{appendices}
\section{Explosion time is realised via a path}\label{s:app-exp}
\begin{proof}[Proof of Lemma \ref{lem:opt-real}] 
 We first rule out sideways explosion. Observe that by the symmetry of the model, (that is, the translation invariance of the Poisson point process and of the connection probabilities $h_{\mathrm q}$),  $\Pv(N_1^t(0)<\infty)=1$ implies the same condition $\Pv(N_1^t(v)<\infty)=1$ for every vertex $v$. For $T\in [0,\infty)$ and $j\in \N$, we denote by $\Gamma_j^T(v)$ the set of vertices that can be reached from $v$ via a path of length at most $j$ and cost at most $T$. Then, inductively assume that $|\Gamma_j^T(v)|<\infty$ a.s.\ for some $j\ge 1$.  Observe that any vertex in $\Gamma_{j+1}^T(v)$ must have an edge of cost at most $T$ from some vertex in $\Gamma_j^T(v)$, so
	 \[  |\Gamma_{j+1}^T(v)| \le \sum_{w\in \Gamma_j^T(v)} N_1^T(w). \]
	 The rhs is almost surely a finite sum over finite summands, so almost surely it is finite. Thus by induction, a.s.\ $|\Gamma_{j+1}^T(v)| < \infty$ for all $j$, and constrasting this with \eqref{eq:sideways1}, we see that sideways explosion a.s.\ does not happen.
	  
	  Next we show that the explosion time $Y_f^\mathrm{I}(v)$ is realised via an infinite path $\pi_{\mathrm{opt}}$.
Recall $\sigma_f^{\mathrm{I}}(v,k):=\inf\{t: |B^{f,L}(v,t)|>k \text{ in } \Ilambda\}$ from Definition \ref{def:explosiontime}. In what follows we consider the sequence $(\sigma_f^{\mathrm I}(v,k))_{k\ge 1}$ as \emph{given}, and show that there exists a sequence $v_0^\star, v_1^\star, \dots$ of distinct vertices with $d_{f,L}(v,v_k^\star) = \sigma_f^{\mathrm I}(v,k)$ for all $k \ge 1$ such that the induced subgraph on $v_0^\star, v_1^\star,\dots, v_k^\star$ is always connected.
This statement may seem obvious when the degrees are finite and $L>0$ almost surely. But, it is non-trivial when $L=0$ happens with positive probability, and one may discover many, even infinitely many, vertices at the same time $t$, possibly forming large zero-cost clusters, and the choice of $v_k^\star$ is far from unique. 
	
	We proceed by induction, taking $v_0^\star = v$; suppose we have found $v_0^\star, \dots, v_{k-1}^\star$ for some $k \ge 1$ forming a connected graph. Since $v \in \calC_\infty$ and $Y_f^{\mathrm I}(v) < \infty$, we have $\sigma_f^{\mathrm I}(v,k) < \infty$. By the definition of $\sigma_f^{\mathrm I}(v,k)$ as the infimum of $t$ where there are at least $k+1$ vertices in the ball $B^{f,L}(v, t)$, for every fixed $\ve>0$,
the set 	
	\be
		\mathcal U(\ve):=B^{f,L}(v,\sigma_f^{\mathrm I}(v,k)+\eps) \setminus \{v_0^\star, \dots, v_{k-1}^\star\}
		\ee
		 is non-empty.
	For every vertex $w(\ve)$ in $\mathcal U(\ve)$, by the definition of cost-distance, there must exist a path $\pi(\eps)$ from $v$ to $w(\ve)$ with cost at most $\sigma_f^{\mathrm I}(v,k)+\eps$. Let $v_k^\star(\eps, w(\ve))$ be the first vertex on $\pi(\eps)$, counting from $v$, which does not lie in $\{v_0^\star, \dots, v_{k-1}^\star\}$; such a vertex must exist since $w(\eps) \notin \{v_0^\star, \dots, v_{k-1}^\star\}$. Let us write $\CN(u)$ for the neighbors of vertex $u$.
	Then $v_k^\star(\ve, w(\ve))$ lies in the set
\[  \CN_k(\ve):=\Big(\cup_{i\le k-1}\CN(v_i^\star) \Big)\cap B^{f,L}(v,\sigma_f^{\mathrm I}(v,k)+\eps),\]
so in particular this set is non-empty for all $\ve>0$.
 It is a closed set by the definition of the cost-distance ball $B^{f,L}(v,r)$ in Section~\ref{sec:notation}. Further, since $B^{f,L}(v,r')\subseteq B^{f,L}(v, r)$ for all $r'<r$, the sequence of sets indexed by $\ve>0$ form a nested sequence as $\ve\downarrow 0$.
 Defining now 
 \[ \CN_k(0):=\cap_{\ve>0} \,\CN_k(\ve),\]
$\CN_k(0)$ is non-empty since it is the intersection of a closed nested sequence of sets.
Intuitively, $\CN_k(0)$ are the neighbors of $\{v_0^\star, \dots, v_{k-1}^\star\}$ that are at cost-distance $\sigma_f^{\mathrm I}(v,k)$ from $v$. To finish the argument, take any vertex in $\CN_k(0)$ as  $v_k^\star$ and continue with the induction. 
	
	We now use $v_0^\star, v_1^\star, \dots$ to show that the explosion time $Y_f^{\mathrm I}(v)=\lim_{k\to \infty} \sigma_f^{\mathrm I}(v,k)<\infty$ is attained as a cost of some infinite path $\pi_{\mathrm{opt}}$ from $v$. For all $k$, we have $d_{f,L}(v,v^\star_k) = \sigma_f^{\mathrm I}(v,k)$, and by definition, $\sigma_f^{\mathrm I}(v,k) \to Y^{\mathrm I}_f(v)$ as $k\to\infty$. Moreover, by construction, there is at least one least-cost path from $v$ to $v_k^\star$ whose internal vertices lie in $\{v_0^\star, \dots, v_{k-1}^\star\}$. Let $\CG_{\text{least}}$ be a breadth-first search tree from $v$ on $(v_k^\star)_{k \ge 0}$, so that for all $k$ there is a unique path $\pi_k$ from $v$ to $v_k^\star$ in $\CG_{\text{least}}$ with cost-length $\sigma_f^{\mathrm I}(v,k)$. By definition, $\CG_{\text{least}}$ is a connected infinite graph. Moreover, the degree of each vertex $v_k^\star$ in $\CG_{\text{least}}$ is a.s.\ finite, since sideways explosion does not occur. Thus $\CG_{\text{least}}$ must contain an infinite path $\pi_{\mathrm{opt}} := v_{k_1}^\star v_{k_2}^\star \dots$ from $v$. Any reordering of a convergent subsequence converges to the same limit as the original sequence, so the cost of $\pi_{\mathrm{opt}}$ is $\lim_{j\to\infty} \sigma_f^{\mathrm I}(v,k_j) = \lim_{k\to\infty} \sigma_f^{\mathrm I}(v,k) = Y^{\mathrm I}_f(v)$, as required.
\end{proof}

\section{Proof of Auxiliary lemmas for finite-size models}\label{s:app}
\subsection{Presence of cores}\label{s:app-box}
\begin{proof}[Proof of Claim \ref{claim:ER}]
  Recall $\mathcal W_n := (x_v,W_v^{(n)})_{v \in \CV_B(n)}$. We will first find $q_r>0$ such that any two vertices in the graph $\mathrm{Core}_{\mathrm{B}}$ are connected by an edge with probability at least $q_r$, independently of other vertices. Thus, $\mathrm{Core}_{\mathrm{B}}$ is dominated from below by an  Erd{\H os}-R\'enyi random graph on $\CV_{\mathrm{Core}}(\mathrm B)$ vertices and connection probability $q_r$, establishing \eqref{eq:ER-domin}.
Then, we will bound the number of vertices $|\CV_{\mathrm{Core}}(\mathrm B)|$ above whp by some $n_r$. We finish the argument by showing that an  Erd{\H os}-R\'enyi random graph with parameters $n_r,q_r$ is connected whp. 
 
  We start by estimating the connection probability between vertices in $\CV_{\mathrm{Core}}(\mathrm B)$ conditioned on $\CV_B(n)$.
 Any $v_1, v_2 \in \CV_{\mathrm{Core}}(\mathrm B)$  
 are distance at most $\sqrt{d}r$ apart and have weights at least the lower end of $I_r$ in \eqref{eq:ir}, so
 \be\label{eq:core-wwd}  W_{v_1} W_{v_2} / \|x_{v_1}-x_{v_2}\|^d \ge d^{-d/2}(r^{d})^{\frac{2(1-\delta)}{DC(\tau-1)}-1}. \ee
 Observe that since $D \in \calI_\delta$ by hypothesis, the exponent is positive by \eqref{eq:cdx2} of Claim~\ref{cl:parameters}. Since Assumption~\ref{assu:mild-connect} holds, so does~\eqref{eq:blown-h}; we have shown that the minimum in the lower bound of~\eqref{eq:blown-h} is taken at the first term. Thus when $r$ is sufficiently large, for all~$v_1,v_2\in\CV_{\mathrm{Core}}(\mathrm B)$,
 \be\label{eq:nk-pk} \Pv( v_1\leftrightarrow v_2 \mid \mathcal W_n) \ge \underline c \exp\big( {-}2 c_2 (\log(r^d))^\gamma  (\tfrac{1+ \delta}{\tau-1})^{\gamma}\big)=:q_r. \ee
This bound holds uniformly for each vertex-pair $v_1, v_2\in \CV_{\mathrm{Core}}(\mathrm B)$. Since connections are present conditionally independently given the vertex positions and weights, the domination from below by the Erd{\H os}-R\'enyi graph follows. 

To establish connectedness, we continue by bounding $|\CV_{\mathrm{Core}}(\mathrm B)|$ from above. 
By Definitions \ref{def:GIRG} and \ref{def:bgirg}, each vertex in $[n]$ has an  i.i.d.\! uniform location in  $\calX_d(n)$.
The probability that a random vertex falls into $\mathrm{B}$ is $r^d/n$. Since vertex-weights are i.i.d.\ $W^{(n)}$, $|\CV_{\mathrm{Core}}(\mathrm B)|$ is a binomial variable with parameter $n$ and acceptance probability $\Pv( W^{(n)} \in I_r)r^d/n $. By Assumption \ref{assu:mild-vertex}, when $r$ is sufficiently large, for all $\eps>0$, its mean is
 \be\ba\label{eq:bin} 
 \Pv\big(W^{(n)} \in I_r\big)r^d &= \Big(\Pv\big(W^{(n)} \ge (r^d)^{(1-\delta)/(DC(\tau-1))} \big)- \Pv\big(W^{(n)} \ge (r^d)^{(1+\delta)/(\tau-1)}\big) \Big)r^d\\
 &\ge \Big(\llow \big(r^{d(1- \delta)/(DC(\tau-1))}\big)r^{-d(1-\delta)/(DC)}r^d\\
 &\qquad\qquad\qquad\qquad - \lhigh \big(r^{d(1+ \delta)/(\tau-1)}\big)r^{-d(1+\delta)} - o(1/n)\Big)r^d\\
 &\ge r^{d(1-\eps-(1-\delta)/(DC))} - r^{d(-\delta+\eps)} + o(r^d/n) \ge r^{d(1-\eps-(1-\delta)/(DC))}/2.
 \ea\ee
where the second inequality follows by Potter's bound, and the third inequality requires $\eps$ to be suitably small. 
By concentration of binomial variables (Lemma~\ref{lem:chernoff}),
\be\label{eq:bin2} \Pv\big(|\CV_{\mathrm{Core}}(\mathrm B)|\ge r^{d(1-\ve-(1-\delta)/(DC))}/4\big) \ge 1- 2\exp(-r^{d(1-\ve-(1-\delta)/(DC))}/24 ),\ee 
so this event occurs whp. Thus we can set $n_r:=r^{d(1-\ve-(1-\delta)/(DC))}/4$.
Observe that if $r$ is sufficiently large, then $q_r\ge \allowbreak n_r^{-1/2} \ge |\CV_{\mathrm{Core}}(\mathrm{B})|^{-1/2}$.

To finish, we note that the connectivity threshold for Erd{\H os}-R\'enyi graphs is $\log n_r/n_r \ll n_r^{-1/2}$, so $\mathrm{Core}_{\mathrm B}$ is indeed connected whp as $n,r \to \infty$.
 
Observe that the above proof carries through when we first condition on $\{\CV_{n}\cap \mathrm B = n_\mathrm{B}\}$. Indeed, the only change is that $|\CV_{\mathrm{Core}}(\mathrm B)|$ becomes binomial with parameters $n_\mathrm{B}$ and $\Pv(W^{(n)}\in I_r)$, hence the bound  on the mean in \eqref{eq:bin} stays valid up to a constant factor as long as $n_\mathrm{B}/ \Ev[|\CV_{\mathrm{Core}}(\mathrm B)|] \in [c_1, c_2]$ for some constants $c_1, c_2$. The rest of the proof is then unchanged. 
 \end{proof}
\subsection{Costs of paths}\label{s:app-cost}
\begin{proof}[Proof of Lemma \ref{lem:highw-connect}] Recall Lemma \ref{lem:bipartite}, which we will apply with $\ve:=\delta$. Given a weight $w_u>1$, define
\be\label{eq:r} {r_u}:=\lceil (\log\log w_u-\log(M\tfrac{1-\delta}{\tau-1}))/\log C \rceil = (\log\log w_u-\log (M\tfrac{1-\delta}{\tau-1}))/\log C + a_{u},\ee where $a_{u}\in [0,1)$ is the upper fractional part. In words, $r_u$ is the smallest index $k$ such that $\exp(M C^k \tfrac{1-\delta}{\tau-1})\ge w_u$. 
Note that $1 \le r_u \le k^\star$ by the assumed bounds on $w_u$. 
Elementary calculation yields that with this notation, and $a_u\in[0,1)$ from \eqref{eq:r},
 \be\label{eq:cmru}\exp(MC^{r_u} \tfrac{1-\delta}{\tau-1})=\exp((\log w_u) C^{a_u}) = w_u^{C^{a_u}} \in [w_u,w_u^C).\ee 
The idea behind the proof is that if $u$ has an edge to some $\delta$-good leader in the annulus $\Gamma_{r_u}$, then $u$ is successful (with $k_0 = r_u-1$) whenever $\cap_{r_u \le k\le k^\star} F_k^{(2)}$ occurs, where  $F_k^{(2)} = F_k^{(2)}(\de)$ in \eqref{boundsweight}.  
 By the law of total probability, 
\be \label{eq:cmru-split}\Pv(\neg \CS_u) \le \Pv\big(\neg(\cap_{r_u\le k\le k^\star} F_k^{(1)}\cap F_k^{(2)}) ) + \Pv(\neg \mathcal S_u \mid \cap_{r_u\le k\le k^\star} F_k^{(1)}\cap F_k^{(2)}). \ee
 For the first term, let us write $F_k^{(1)}(\de, M)$ and $F_k^{(2)}(\de, M)$ to emphasise how $F_k^{(1)}$ and $F_k^{(2)}$ depend on $M$ in \eqref{boundsweight}. Observe that the boxes with parameter $M$ and index $r_u+k$ for some $k\ge0$ can be considered as boxes with parameter $MC^{r_u}$ (instead of $M$) and index $k\ge 0$. 
 Then, by the definition of $F_k^{(1)}$ and $F_k^{(2)}$, we have $F_k^{(1)}(\delta,M) = F_{k-r_u}^{(1)}(\delta, MC^{r_u})$ and $F_k^{(2)}(\delta,M) = F_{k-r_u}^{(2)}(\delta,MC^{r_u})$ for all $k \geq r_u$.  Hence, taking $M$ suitably large relative to $\delta$ and applying Lemma~\ref{lem:bipartite} with $\delta:=\eps$ and $\la:=1$, 
 the first term in~\eqref{eq:cmru-split} is bounded~by 
 \be\ba
 \Pv\big(\neg(\cap_{r_u\le k\le k^\star} F_k^{(1)}\cap F_k^{(2)}) ) &\leq 1-p_{MC^{r_u}} = 3\exp\Big({-}\e^{MC^{r_u}((D-1)\wedge 1)(1-\delta)}2^{-d}/75\Big).
 \ea\ee
 Applying~\eqref{eq:cmru}, it follows that 
 \be\label{eq:rho1}\ba
 \Pv\big(\neg(\cap_{r_u\le k\le k^\star} F_k^{(1)}\cap F_k^{(2)}) ) &\leq 3\exp \big( -w_u^{\tfrac{\tau-1}{1-\delta} ((D-1)\wedge 1)(1-\delta)}2^{-d}/75\big)\\
 &\le 3\exp \big( -w_u^{\eta(\delta)}2^{-d}/75\big).
 \ea\ee 
It remains to bound the second term in \eqref{eq:cmru-split}. Conditioned on $\cap_{r_u\le k\le k^\star} F_k^{(1)}\cap F_k^{(2)}$, let us expose $\mathcal W_n = (x_v,W_v^{(n)})_{v \in \CV_B(n)}$ (which exposes the set of $\delta$-good leader vertices), together with all edges between $\delta$-good leader vertices in annuli $\Gamma_k$ with $k \ge r_u$. Note that this is sufficient to determine the event $\cap_{r_u\le k\le k^\star} F_k^{(1)}\cap F_k^{(2)}$, so the remaining edges are present independently. Since $\cap_{r_u \le k \le k^*}(F_k^{(1)}\cap F_k^{(2)})$ occurs, there is a box-increasing path from every $\delta$-good leader in $\Gamma_{r_u}$ to a leader in $\Gamma_{k^\star}$. Consequently, $\CS_u$ occurs if there is an edge between $u$ and some $\delta$-good leader in $\Gamma_{r_u}$.

Since $F_{r_u}^{(1)}$ occurs, there are at least $b_k/2$ many  $\delta$-good leaders in $\Gamma_{r_u}$ (see \eqref{eq:delta-good} and \eqref{boundsweight}), that is, leaders with 
weights in the interval 
\be\label{eq:weight-leader2} \Big(w_u^{C^{a_u}}, w_u^{C^{a_u} (1+\delta)/(1-\delta)}\Big]. \ee
Moreover, the Euclidean distance of the leaders in $\Gamma_{r_u}$ from $u$ is at most $d$ times the outer radius of $\Gamma_{r_u}$, which is $d\exp(M C^{r_u}D/d) \le dw_u^{C (\tau-1) D/((1-\delta)d)}$. Since Assumption~\ref{assu:mild-connect} holds, we may use the lower bound~\eqref{eq:blown-h} on connection probabilities in $\BGIRG$ to lower-bound the connection probability between $u$ and a $\delta$-good leader $v$ in $\Gamma_{r_u}$.
To estimate the resulting expression, we first estimate a term appearing in~\eqref{eq:blown-h} by
\[ w_u W_v \|x_u-x_v\|^{-d}\geq d^{-d} w_u^{1+\frac{1+\delta}{1-\delta} - D C \tfrac{\tau-1}{1-\delta}}
\ge d^{-d} w_u^{2-DC\tfrac{\tau-1}{1-\delta}}.
\]
We argue that the exponent of $w_u$ is strictly positive whenever $D\in \mathcal I_\delta$ in Claim \eqref{cl:parameters}. Indeed, multiplying the exponent by $(1-\delta)/(\tau-1)$, the exponent becomes the lhs of \eqref{eq:cdx2} with $s=1$, which is positive by Claim \eqref{cl:parameters}.
Since $w_u \ge \e^M$, it follows that $w_uW_v||x_u-x_v||^{-d} \ge 1$ when $M$ is sufficiently large relative to $\delta$ and $d$.
As a result, the minimum in the lower bound on $g_n^{u,v}$ in~\eqref{eq:blown-h} is taken at $\un c l(w_u)l(w_v)$, which, by \eqref{eq:weight-leader2}, is at least
 \[ q_u:=\underline c \exp\Big( - c_2 (\log w_u)^\gamma\big(1+ C^{\gamma}(\tfrac{1+\delta}{1-\delta})^\gamma\Big). \] 
Recall that, since $F_{r_u}^{(1)}$ occurs, the number of $\delta$-good leaders in the annulus $\Gamma_{r_u}$ is at least $b_{r_u}/2\ge 2^{-d-2} \exp((\log w_u) (D-1)\frac{\tau-1}{1-\delta})$, by \eqref{eq:boundbk-girg} and \eqref{eq:cmru}.
Let $N_{r_u}(u)$ denote the number of $\delta$-good leaders in annulus $\Gamma_{r_u}$ that $u$ is adjacent to. 
Then, $N_{r_u}(u)$ is dominated below by a binomial random variable $Z$ with mean at least
\be \ba 2^{-d-2}w_u^{(D-1)\frac{\tau-1}{1-\delta}} q_u &\ge (\underline c/2^{d+2}) \exp\!\Big(  (\log w_u) \tfrac{\tau-1}{1-\delta}(D-1) \! -\! c_2  (\log w_u)^{\gamma} \big(1+ C^{\gamma}(\tfrac{1+\delta}{1-\delta})^\gamma \Big) \\
&\ge 
(\underline c/2^{d+2}) \exp\!\Big(  (\log w_u) \tfrac{\tau-1}{1-\delta}(D-1) \! -\! 3c_2  (\log w_u)^{\gamma}\Big)\\
&\ge (\underline c/2^{d+2}) w_u^{(\tau-1)(D-1)}=:\mu, 
 \ea\ee
where the inequality between the first and second line holds whenever $M$ is sufficiently large relative to $\delta$, since $w_u \ge \e^M$ and $\delta>0$\cbb.  In particular, by the Chernoff bound of Lemma~\ref{lem:chernoff} with $\ve=1$,
\be \ba \Pv\big(\neg\CS_u\mid \cap_{r_u\le k\le k^\star} F_k^{(1)}\cap F_k^{(2)} \big) &\le 2 \exp( -\mu/3) \le 2\exp\big( {-}(\underline c/(3\cdot 2^{d+2})) w_u^{(\tau-1)(D-1)}\big)\\
&\le 2\exp\big( {-}(\underline c\cdot 2^{-d}/12) w_u^{\eta(\delta)}\big)
\ea \ee 
This shows \eqref{eq:su-bound-1}, by combining this with \eqref{eq:cmru-split} and \eqref{eq:rho1} and combining constant prefactors.
\end{proof}

\begin{proof}[Proof of Lemma \ref{lem:cheap-connect}]
We start by constructing a boxing system around $u\in \BGIRG$ with parameter $M_{w_u}=\tfrac{\tau-1}{1-\delta}\log w_u$. The bound $w_u< n^{(1-\delta)/(D(\tau-1))}$ ensures that the boxing system is non-trivial.
We will choose the value of $K$ along the proof.
By the weight bound on $\delta$-good leaders given in~\eqref{eq:delta-good}, the lowest weight of any $\delta$-good leader in the whole boxing system is strictly larger than $\exp(M_{w_u} \tfrac{1-\delta}{\tau-1})=w_u$. The only vertices in the path we construct  other than $u$ will be $\delta$-good leaders, so this part of the result will follow immediately.

We apply \eqref{boundsweight} and \eqref{boundsweight-error} from Lemma~\ref{lem:bipartite} to bound the number of $\delta$-good leaders. With probability at most
\be\label{eq:pmwu} \Xi_1(w_u):=p_{M_{w_u}}=3 \exp\Big(  - w_u^{\frac{\tau-1}{1-\delta} ((D-1)\wedge 1) (1-\delta) } 2^{-d}/75 \Big), \ee
the event $\cap_{k\le k^\star}(F_k^{(1)}\cap F_k^{(2)})$ does not hold. We then construct the greedy path and use Lemma \ref{lem:cost-greedy}, with 
\[
\zeta_k:=\exp\Big(\frac{M_{w_u} C^k \beta^+\xi(\delta)}{2(1+\delta)}\Big) 
\]
Similarly to~\eqref{eq:fini} and~\eqref{eq:explo-sum-bound},
 with additional error probability $\sum_{k=0}^{\infty} \exp(-\zeta_k)$, the greedy path emanating from a $\delta$-good leader in annulus $0$ in this boxing system has cost at most
 \be \label{eq:explo-sum-bound333} |\pi^{\text{greedy}} |_{f,L} \le \sum_{k=0}^{k^\star(M_{w_u}, n)} \zeta_k^{(1+\delta)/\beta^+}\exp\Big(M_{w_u}C^k \cdot\big((\mu + \nu C) \frac{1+\delta}{\tau-1} - \frac{(1-\delta)^2}{\beta^+}C(D-1)\big)\Big).\ee 
As shown after \eqref{eq:explo-sum-bound}, the coefficient of $M_{w_u}C^k$ is strictly negative whenever $D\in \CI_\delta$, and we denoted this coefficient by $-\xi(\delta)$ in \eqref{eq:xi-rho}.  Then the $\zeta_k$ factor in \eqref{eq:explo-sum-bound333} can be merged with the exponential factor, yielding
\be\label{eq:48-cost-1}
|\pi^{\text{greedy}} |_{f,L} \le \sum_{k=0}^\infty \e^{-M_{w_u} C^k \xi(\delta)/2 } \le 2 \e^{-M_{w_u} \xi(\delta)/2}=2w_u^{-\tfrac{\tau-1}{1-\delta}\xi(\delta)/2},\ee
with failure probability at most 
\be\label{eq:xi2} \Xi_2(w_u):=\sum_{k=0}^\infty\e^{-\zeta_k}\le 2 \e^{-\zeta_0}= 2w_u^{-\tfrac{\tau-1}{1-\delta} \tfrac{\beta^+}{1+\delta} \tfrac{\xi(\delta)}{2}}, \ee
since the $\zeta_k$ sequence is thinner than a geometric series. 

If $u$ is itself a $\delta$-good leader in $\Gamma_0$, then we are done. Suppose not; then we will connect $u$ to a $\delta$-good leader $c_0^\star$ in  $\Gamma_0$, and then concatenate that edge with the greedy path emanating from $c_0^\star$. The event $F_0^{(1)}$ implies also that the number of $\delta$-good leaders in $\Gamma_0$ is at least
$b_0/2 \ge \exp(M_{w_u} (D-1))/2^{d+2}=w_u^{\frac{\tau-1}{1-\delta}(D-1)}/2^{d+2}$ by \eqref{eq:boundbk-girg}, and the diameter of $\Gamma_0$ is at most $d\e^{M_{w_u}D/d}$. For the connection probability between $u$ and any of these $\delta$-good leader vertices, we observe that 
these leaders have weight in the interval $(w_u, w_u^{\frac{1+\delta}{1-\delta}}]$ by~\eqref{eq:delta-good},
 hence we can bound a term in the minimum of the connection probability \eqref{eq:blown-h} below by
\[ w_u W_{c_0^*} / \| u- c_0^*\|^d \ge w_u^2 / (d^d\e^{M_{w_u}D})= w_u^{2-\frac{\tau-1}{1-\delta} D}/d^d.  \]
Since $D\in \CI_\delta$, the exponent is positive by \eqref{eq:cdx2} of Claim~\ref{cl:parameters} applied with $s=0$. Hence, choosing $K$ sufficiently large, and $w_u\ge K$ ensures that $w_uW_{c_0^*}/\|u-c_0^*\|^d \ge 1$ and so the minimum in  \eqref{eq:blown-h} evaluates to $l(w_u)l(W_{c_0^*})$. Let $N_{0}(u)$ denote the number of $\delta$-good leaders adjacent to $u$ in $\Gamma_0$. It follows from the above discussion that under our conditioning,
\[ N_{0}(u)\  {\buildrel d \over \ge }\ \mathrm{Bin}\Big(w_u^{\frac{\tau-1}{1-\delta}(D-1)}/2^{d+2}, \underline c l(w_u) l(w_u^{\frac{1+\delta}{1-\delta}})\Big). \]
Since $l(\cdot) $ varies slowly at infinity, the mean is at least $d_0 := w_u^{(\tau-1)(D-1)}$ if $w_u$ is sufficiently large, which is ensured by increasing $K$ when necessary. Observe that the lower bound $w_u\ge K$ ensures this.
Applying the Chernoff bound of Lemma \ref{lem:chernoff} to the variable on the rhs, we arrive at
\be\label{eq:xi3} \Xi_3(w_u):=\Pv(N_0(u)< d_0/2) \le \exp( - w_u^{(\tau-1)(D-1)}/12).\ee
Conditioned on the  event $\{N_0(u)\ge d_0/2\}$,   we now take $c_0^\star$ to be a vertex with minimal $L_{(u, c_0)}$ among the $\delta$-good leader neighbors of $u$ in $\Gamma_0$, of which there are $N_0(u)$.
The cost of this edge is then
\be\label{eq:Cbound}
  C_{(u, c_{0}^\star)} \ {\buildrel d \over \le}\ 
w_{u}^\mu w_{u}^{\nu\tfrac{1+\delta}{1-\delta}} \min_{i\le  \underline c w_{u}^{(D-1)(\tau-1)}} L_i,
\ee
where the $L_i$'s are i.i.d.\ copies of $L$. As in~\eqref{eq:lower-min}, for all $N,\zeta > 0$, we have
\[
	\Pv\big(\min_{j\leq N} L_j > F_L^{(-1)}(\zeta/N)\big) = \big(1-F_L(F_L^{(-1)}(\zeta/N))\big)^N \le (1 -\zeta/N)^N\le \e^{-\zeta};
\] 
it follows from~\eqref{eq:Cbound} that for all $\zeta>0$,
\[
\Pr\Big(C_{(u,c_0^*)} > w_{u}^\mu w_{u}^{\nu \tfrac{1+\delta}{1-\delta}} F_L^{(-1)} \Big( (\zeta /\underline c) \cdot  w_{u}^{-(D-1)(\tau-1)}\Big)\Big) \le \e^{-\zeta}.
\]
Using the fact that  $F_L^{(-1)}(y) \le y^{(1-\delta)/\beta^+}$  holds for all sufficiently small $y>0$, it follows that when $w_u$ is sufficiently large, (which is ensured by increasing $K$ when necessary),
\be\label{eq:wunq}
	\Pr\Big(C_{(u,c_0^*)} > (\zeta/ \underline c)^{(1-\delta)/\beta^+}  w_{u}^{\mu+\nu\tfrac{1+\delta}{1-\delta}} w_{u}^{-(D-1)(\tau-1)(1-\delta)/\beta^+}\Big) \le \e^{-\zeta}.
\ee
We will set the value of $\zeta>0$ shortly. 
The exponent of $w_{u}$ in \eqref{eq:wunq} is (by \eqref{eq:xi-rho}),
\be\label{eq:rho-delta} \ba -\rho(\delta)&=\mu +\nu\frac{1+\delta}{1-\delta} - \frac{(D-1)(\tau-1)(1-\delta)}{\beta^+}\\
&= \frac{\tau-1}{1-\delta}\Big(\frac{\mu+C\nu}{\tau-1} - \frac{(1-\delta)^2(D-1)}{\beta^+} \Big);
\ea\ee
since $D \in \mathcal{I}_\delta$, this is strictly negative by \eqref{eq:mubeta-3} of Claim \eqref{cl:parameters} (taking $s=1$). 
We now set $\zeta$ so that $(\zeta/\underline c)^{(1-\delta)/\beta^+}:=w_u^{\rho(\delta) /2}$, and then   \eqref{eq:wunq} implies
\be \label{eq:cu-bound} C_{(u, c_{0^\star})} \le  w_u^{\rho(\delta)/2} \cdot w_u^{-\rho(\delta)} \le  w_u^{-\rho(\delta)/2}, \ee
with failure probability at most  
\be\label{eq:xi4} \Xi_4(w_0):=\e^{-\zeta}=\exp(-\underline c w_u^{\rho(\delta)\beta^+/(2(1-\delta))}).\ee
Collecting the error terms $\Xi_i(w_u)$ for $i\le 4$ from \eqref{eq:pmwu}, \eqref{eq:xi2}, \eqref{eq:xi3}, \eqref{eq:xi4}, we observe that \eqref{eq:xi2} dominates for all sufficiently large $w_u$, (which is ensured by increasing $K$ when necessary),
since the other terms are exponentially small in $w_u$. Hence, the cheap path can be constructed with failure probability $\Xi(w_u)$, as specified in \eqref{eq:xi-error}. Adding up the costs $C_{(u, c_0^\star)} + | \pi^{\text{greedy}}|_{f,L}$
yields \eqref{eq:xi-cost}, by~\eqref{eq:cu-bound} and~\eqref{eq:48-cost-1}. We have already observed that all vertices on $\pi^{\text{greedy}}$ except $u$ have weights greater than $w_u$.

We finish by calculating the weight of the end-vertex of $\pi^{\text{greedy}}$. 
We look at the definition of  $k^\star(n, M_{w_u})$ from \eqref{eq:kstar}, 
yielding that 
\[
	k^\star(n, M_{w_u})= \frac{1}{\log C}\log \Big(\frac{\log n}{M_{w_u}D}\Big) -s
\]
for some fractional part $s\in [0,1)$. Hence by the definition of $\delta$-good leaders in~\eqref{eq:delta-good}, the weight of the last leader vertex is in the interval
\begin{align*}
 \Big(\e^{M_{w_u} \frac{1-\delta}{\tau-1} C^{k^\star}},\ \e^{M_{w_u} \frac{1+\delta}{\tau-1} C^{k^\star}  }\Big]
 = \Big(n^{\frac{(1-\delta)C^{-s}}{D (\tau-1)}},\ n^{\frac{(1+\delta)C^{-s}}{D (\tau-1)}}\Big] \subseteq \Big( n^{\frac{(1-\delta)}{D C(\tau-1)}}, n^{\frac{(1-\delta)}{D (\tau-1)}}\Big].
\end{align*}
This finishes the proof.
\end{proof}

\begin{proof}[Proof of Lemma \ref{lem:core}]
Let $\eps > 0$ be a small constant, whose value will be determined later. Let $\mathrm{mon}(f)$ be the number of monomials in $f$.
Let $t_0(\eps)$ be such that $F_L(t) \ge t^{\beta^+(1+\eps)}$ on $[0,t_0]$; such a constant must exist by the definition of $\beta^+$ in~\eqref{eq:beta1}. Let $\CV_\textrm{end}$ be the set of all vertices with weights in $I_\textrm{end}$, and let $\CG_{\textrm{end}}$ be the subgraph of $\BGIRG$ on this vertices set whose edges $e$ have cost 
\be\label{eq:keep-core}
	L_e \le n^{-(1+\eps)(\deg f)(1+\delta)/(D(\tau-1))}/\mathrm{mon}(f).
\ee

We will first dominate $\CG_{\textrm{end}}$ below by an Erd\H{o}s-R\'enyi random graph, using an argument similar to the proof of Claim~\ref{claim:ER}. Note that the lower bound of $I_\mathrm{end}$ is equal to the lower bound of $I_r$ in that claim with $r^d=n$, so many of the calculations carry over. As in that proof, conditioned on the positions and weights of the vertices, edges are present independently. To bound the probability with which each any two vertices are adjacent, we note that~\eqref{eq:core-wwd} goes through unchanged, so the minimum in~\eqref{eq:blown-h} is still taken at the left term; again as in that proof, it follows that when $r$ is sufficiently large, the probability of any two vertices in $\CV_\mathrm{end}$ being adjacent in $\BGIRG$ (conditioned on all other vertex positions and weights) is at least
\be q_\mathrm{Core} := \exp\Big( -2c_2 (\log n)^\gamma \big(\tfrac{1+\delta}{D(\tau-1)}\big)^\gamma \Big). \ee
By~\eqref{eq:keep-core} and the existence of $t_0(\eps)$, when $n$ is sufficiently large, the probability of such an edge remaining in $\CG_\textrm{end}$ under the same conditioning is
\[   p_f := F_L(n^{-(1+\ve)(\deg f)(1+ \delta)/(D(\tau-1))}/\mathrm{mon}(f)) \ge n^{ -\beta^+ (1+\ve)^2(\deg f)(1+ \delta)/(D(\tau-1))}/\mathrm{mon}(f)^{\beta^+}. \]
Thus $\CG_\mathrm{end}$ is dominated below by an Erd\H{o}s-R\'enyi graph with edge probability $q_\mathrm{Core}p_f$.

As in~\eqref{eq:bin} and~\eqref{eq:bin2} in the proof of Claim~\ref{claim:ER}, for any constant $\eps>0$, by the standard Chernoff bound (Lemma~\ref{lem:chernoff}) we have
\be\label{eq:g-end} |\CV_{\mathrm{end}}| \ge \tfrac{1}{4}n^{1-\eps-(1-\delta)/(DC)}\ee
with probability at least $1 - 2\exp(-n^{(1-\eps-(1-\delta)/(DC))}/24)$; thus this occurs whp. Let $\underline N:=\frac12 n^{1-(1+\ve)(1-\delta)/(DC)}$ from \eqref{eq:g-end}.
We show that, for some $\eta>0$,  $p_f\cdot q_{\mathrm{Core}}> \underline N^{\eta-1}$. 
Hence, by the result of Bollob\'as \cite{boll}, whp, $\CG_{\mathrm{end}}$ is connected and its diameter is at most a constant $K(\eta)$. 
Then \eqref{eq:keep-core} implies that the cost-distance between any two vertices with weight in  $I_{\mathrm{end}}$ in \eqref{eq:iend} is at most 
\be\label{eq:core-cost-2}
 d_{f,L}(u_1, u_2)\le K(\eta) L_{(u_1,u_2)} \cdot \mathrm{mon}(f) \cdot (n^{(1+\delta)/(D(\tau-1))})^{\deg(f)} \le   K(\eta)\cdot n^{-\ve(\deg f)(1+ \delta)/(D(\tau-1))}, 
 	\ee
as required in~\eqref{eq:core-cost} (taking $\zeta = \ve(\deg f)(1+\delta)/(D(\tau-1))$).

It remains to prove that $p_f\cdot q_{\mathrm{Core}}> \underline N^{\eta-1}$ for an appropriate choice of $\eta>0$. Since $q_\mathrm{Core}$ varies slowly, when $n$ is sufficiently large, we have
\[ p_f\cdot q_{\mathrm{Core}}\cdot \underline N \ge n^{-\beta^+ (1+\ve)^3(\deg f)(1+ \delta)/(D(\tau-1))   }\cdot n^{ 1-(1+\ve) (1-\delta)/(DC)}=:n^{\chi},\]
where we increased the exponent of $1+\ve$ from two to three in order to remove lower-order terms. 
We argue that $\chi$, the exponent of $n$, is positive.
Indeed, this holds when 
\[  -\frac{\chi D}{\beta^+ (1+\ve)^3}=(\deg f)\frac{1+ \delta}{\tau-1}  - \frac{1}{\beta^+ (1+\ve)^3} \frac{DC - (1+\ve) (1-\delta) }{C}<0.
 \]
 When we taking $\ve$ sufficiently small, we have $(1+\ve)(1-\delta) \le (1+\delta) = C$ (by \eqref{eq:parameters}) and hence $\chi>0$ when
 \[ (\deg f)\frac{1+ \delta}{\tau-1}  - \frac{1}{\beta^+ (1+\ve)^3} (D-1) <0. \]
Choosing $\ve$ small enough that $1/(1+\ve)^3\ge (1-\delta)^2$, we obtain \eqref{eq:mubeta-3} with $s=0$. Since by hypothesis we chose $C$, $D$ and $\delta$ as in Claim~\ref{cl:parameters}, it follows that $\chi > 0$. We can therefore set 
 $\eta= \chi$ to obtain $p_f q_{\mathrm{Core}} > \underline N^{\eta - 1}$ and finish the proof.
\end{proof}

\end{appendices}
	\bibliographystyle{abbrv}
	\bibliography{GIRG_SI_explosion_2019-08-27-arxiv}

\begin{thebibliography}{10}

\bibitem{AbdBodFou16}
M.~A. Abdullah, N.~Fountoulakis, and M.~Bode.
\newblock Typical distances in a geometric model for complex networks.
\newblock {\em Internet Mathematics 1 (electronic)}, August 2017.

\bibitem{adriaans2018weighted}
E.~Adriaans and J.~Komj{\'a}thy.
\newblock Weighted distances in scale-free configuration models.
\newblock {\em Journal of Statistical Physics}, 173(3):1082--1109, Nov 2018.

\bibitem{albert2002statistical}
R.~Albert and A.-L. Barab{\'a}si.
\newblock Statistical mechanics of complex networks.
\newblock {\em Reviews of modern physics}, 74(1):47, 2002.

\bibitem{AmiDev13}
O.~Amini, L.~Devroye, S.~Griffiths, and N.~Olver.
\newblock On explosions in heavy-tailed branching random walks.
\newblock {\em Annals of Probability}, 41(3 B):1864--1899, May 2013.

\bibitem{bajardi2011dynamical}
P.~Bajardi, A.~Barrat, F.~Natale, L.~Savini, and V.~Colizza.
\newblock Dynamical patterns of cattle trade movements.
\newblock {\em Public Library of Science PLoS One}, 6(5):e19869, 2011.

\bibitem{baroni2017nonuniversality}
E.~Baroni, R.~van~der Hofstad, and J.~Komj{\'a}thy.
\newblock Nonuniversality of weighted random graphs with infinite variance
  degree.
\newblock {\em Journal of Applied Probability}, 54(1):146--164, 2017.

\bibitem{BinGolTeu89}
N.~H. Bingham, C.~M. Goldie, and J.~L. Teugels.
\newblock {\em Regular variation}, volume~27 of {\em Encyclopedia of
  Mathematics and its Applications}.
\newblock Cambridge University Press, Cambridge, 1989.

\bibitem{Bisk04}
M.~Biskup.
\newblock On the scaling of the chemical distance in long-range percolation
  models.
\newblock {\em The Annals of Probability}, 32(4):2938--2977, 10 2004.

\bibitem{BlaFriKat19}
T.~Bl{\"a}sius, T.~Friedrich, M.~Katzmann, U.~Meyer, M.~Penschuck, and
  C.~Weyand.
\newblock Efficiently generating geometric inhomogeneous and hyperbolic random
  graphs.
\newblock In {\em Proceedings of the 27th Annual European Symposium on
  Algorithms (ESA 2019)}, volume 144 of {\em Leibniz International Proceedings
  in Informatics (LIPIcs)}, pages 21:1--21:14, Dagstuhl, Germany, 2019.

\bibitem{boccaletti2006complex}
S.~Boccaletti, V.~Latora, Y.~Moreno, M.~Chavez, and D.-U. Hwang.
\newblock Complex networks: Structure and dynamics.
\newblock {\em Physics Reports}, 424(4-5):175--308, 2006.

\bibitem{BodFouMul15}
M.~Bode, N.~Fountoulakis, and T.~M{\"u}ller.
\newblock On the largest component of a hyperbolic model of complex networks.
\newblock {\em The Electronic Journal of Combinatorics}, 22:1--43, 2015.

\bibitem{BogPas03}
M.~Bogu\~n\'a and R.~Pastor-Satorras.
\newblock Class of correlated random networks with hidden variables.
\newblock {\em Physical Review E}, 68:036112, Sep 2003.

\bibitem{BogPasDia04}
M.~Bogu\~n\'a, R.~Pastor-Satorras, A.~D\'{\i}az-Guilera, and A.~Arenas.
\newblock Models of social networks based on social distance attachment.
\newblock {\em Physical Review E}, 70:056122, Nov 2004.

\bibitem{BogPapaKriou10}
M.~Bogu{\~n}\'a, F.~Papadopoulos, and D.~Krioukov.
\newblock Sustaining the internet with hyperbolic mapping.
\newblock {\em Nature Communications}, 1(6), September 2010.

\bibitem{boll}
B.~Bollob\'as.
\newblock {\em Random graphs}.
\newblock Cambridge studies in advanced mathematics ; 73. Cambridge University
  Press, Cambridge ; New York, 2nd ed. edition, 2001.

\bibitem{BolSvaRio07}
B.~Bollob\'{a}s, S.~Janson, and O.~Riordan.
\newblock The phase transition in inhomogeneous random graphs.
\newblock {\em Random Structures \& Algorithms}, 31(1):3--122, Aug. 2007.

\bibitem{bringmann2016average}
K.~Bringmann, R.~Keusch, and J.~Lengler.
\newblock Average distance in a general class of scale-free networks with
  underlying geometry.
\newblock arXiv:1602.05712, 2016.

\bibitem{BriKeuLen15}
K.~Bringmann, R.~Keusch, and J.~Lengler.
\newblock {Sampling Geometric Inhomogeneous Random Graphs in Linear Time}.
\newblock In {\em 25th Annual European Symposium on Algorithms (ESA 2017)},
  volume~87 of {\em Leibniz International Proceedings in Informatics (LIPIcs)},
  pages 20:1--20:15, Dagstuhl, Germany, 2017.

\bibitem{BriKeuLen19}
K.~Bringmann, R.~Keusch, and J.~Lengler.
\newblock Geometric inhomogeneous random graphs.
\newblock {\em Theoretical Computer Science}, 760:35--54, 2019.

\bibitem{BriKeuLenYan17}
K.~Bringmann, R.~Keusch, J.~Lengler, Y.~Maus, and A.~R. Molla.
\newblock Greedy routing and the algorithmic small-world phenomenon.
\newblock In {\em Proceedings of the ACM Symposium on Principles of Distributed
  Computing}, PODC '17, pages 371--380, New York, NY, USA, 2017. ACM.

\bibitem{CanFou14}
E.~Candellero and N.~Fountoulakis.
\newblock Clustering and the hyperbolic geometry of complex networks.
\newblock In {\em Algorithms and Models for the Web Graph}, pages 1--12, Cham,
  2014. Springer International Publishing.

\bibitem{CanFou16}
E.~Candellero and N.~Fountoulakis.
\newblock Bootstrap percolation and the geometry of complex networks.
\newblock {\em Stochastic Processes and their Applications}, 126(1):234 -- 264,
  2016.

\bibitem{CanSta18}
E.~Candellero and A.~Stauffer.
\newblock Coexistence of competing first passage percolation on hyperbolic
  graphs.
\newblock arXiv:18210.04593, 2018.

\bibitem{centola2010spread}
D.~Centola.
\newblock The spread of behavior in an online social network experiment.
\newblock {\em Science}, 329(5996):1194--1197, 2010.

\bibitem{DeiHofHoo13}
M.~Deijfen, R.~van~der Hofstad, and G.~Hooghiemstra.
\newblock Scale-free percolation.
\newblock {\em Annales de l'Institut Henri Poincar{\'e}, Probabilit{\'e}s et
  Statistiques}, 49(3):817--838, 2013.

\bibitem{DepWut18}
P.~Deprez and M.~V. W{\"u}thrich.
\newblock Scale-free percolation in continuum space.
\newblock {\em Communications in Mathematics and Statistics}, Jul 2018.

\bibitem{dorogovtsev2008critical}
S.~N. Dorogovtsev, A.~V. Goltsev, and J.~F. Mendes.
\newblock Critical phenomena in complex networks.
\newblock {\em Reviews of Modern Physics}, 80(4):1275, 2008.

\bibitem{feldman2017high}
J.~Feldman and J.~Janssen.
\newblock High degree vertices and spread of infections in spatially modelled
  social networks.
\newblock In {\em International Workshop on Algorithms and Models for the
  Web-Graph}, pages 60--74. Springer, 2017.

\bibitem{FouMul18}
N.~Fountoulakis and T.~M{\"u}ller.
\newblock Law of large numbers for the largest component in a hyperbolic model
  of complex networks.
\newblock {\em The Annals of Applied Probability}, 28(1):607--650, 02 2018.

\bibitem{GanKeaNew92}
A.~Gandolfi, M.~S. Keane, and C.~M. Newman.
\newblock Uniqueness of the infinite component in a random graph with
  applications to percolation and spin glasses.
\newblock {\em Probability Theory and Related Fields}, 92(4):511--527, 1992.

\bibitem{giuraniuc2006criticality}
C.~Giuraniuc, J.~Hatchett, J.~Indekeu, M.~Leone, I.~P. Castillo,
  B.~Van~Schaeybroeck, and C.~Vanderzande.
\newblock Criticality on networks with topology-dependent interactions.
\newblock {\em Physical Review E}, 74(3):036108, 2006.

\bibitem{grey1974}
D.~Grey.
\newblock Explosiveness of age-dependent branching processes.
\newblock {\em Probability Theory and Related Fields}, 28(2):129--137, 1974.

\bibitem{gruhl2004information}
D.~Gruhl, R.~Guha, D.~Liben-Nowell, and A.~Tomkins.
\newblock Information diffusion through blogspace.
\newblock In {\em Proceedings of the 13th international conference on World
  Wide Web}, pages 491--501. ACM, 2004.

\bibitem{GugPanPet12}
L.~Gugelmann, K.~Panagiotou, and U.~Peter.
\newblock Random hyperbolic graphs: degree sequence and clustering.
\newblock In {\em 39th International Colloquium on Automata, Languages, and
  Programming (ICALP)}, pages 573--585, 2012.

\bibitem{hammersley1965first}
J.~M. Hammersley and D.~J. Welsh.
\newblock First-passage percolation, subadditive processes, stochastic
  networks, and generalized renewal theory.
\newblock In {\em Bernoulli 1713, Bayes 1763, Laplace 1813}, pages 61--110.
  Springer, 1965.

\bibitem{HeyHulJor17}
M.~Heydenreich, T.~Hulshof, and J.~Jorritsma.
\newblock Structures in supercritical scale-free percolation.
\newblock {\em The Annals of Applied Probability}, 27(4):2569--2604, 2017.

\bibitem{isham2011spread}
V.~Isham, J.~Kaczmarska, and M.~Nekovee.
\newblock Spread of information and infection on finite random networks.
\newblock {\em Physical Review E}, 83(4):046128, 2011.

\bibitem{JLR}
S.~Janson, T.~{\L}uczak, and A.~Rucinski.
\newblock {\em Random Graphs}.
\newblock John Wiley \& Sons, 2000.

\bibitem{janssen2017rumors}
J.~Janssen and A.~Mehrabian.
\newblock Rumors spread slowly in a small-world spatial network.
\newblock {\em SIAM Journal on Discrete Mathematics}, 31(4):2414--2428, 2017.

\bibitem{JorrGIRG2}
J.~Jorritsma.
\newblock Random graph visualizations (version v0.1.0).
\newblock {\em Zenodo}, Apr 2020.
\newblock doi:10.5281/zenodo.3735853.

\bibitem{karsai2006nonequilibrium}
M.~Karsai, R.~Juh{\'a}sz, and F.~Igl{\'o}i.
\newblock Nonequilibrium phase transitions and finite-size scaling in weighted
  scale-free networks.
\newblock {\em Physical Review E}, 73(3):036116, 2006.

\bibitem{karsai2011small}
M.~Karsai, M.~Kivel{\"a}, R.~K. Pan, K.~Kaski, J.~Kert{\'e}sz, A.-L.
  Barab{\'a}si, and J.~Saram{\"a}ki.
\newblock Small but slow world: How network topology and burstiness slow down
  spreading.
\newblock {\em Physical Review E}, 83(2):025102, 2011.

\bibitem{KiwMit18}
M.~Kiwi and D.~Mitsche.
\newblock Spectral gap of random hyperbolic graphs and related parameters.
\newblock {\em The Annals of Applied Probability}, 28(2):941--989, 04 2018.

\bibitem{KocLen16}
C.~Koch and J.~Lengler.
\newblock {Bootstrap Percolation on Geometric Inhomogeneous Random Graphs}.
\newblock In {\em 43rd International Colloquium on Automata, Languages, and
  Programming (ICALP 2016)}, volume~55 of {\em Leibniz International
  Proceedings in Informatics (LIPIcs)}, pages 147:1--147:15, Dagstuhl, Germany,
  2016.

\bibitem{julia-girg}
J.~Komj{\'a}thy and B.~Lodewijks.
\newblock Explosion in weighted hyperbolic random graphs and geometric
  inhomogeneous random graphs.
\newblock {\em Stochastic Processes and their Applications}, 2019.

\bibitem{krioukov2010hyperbolic}
D.~Krioukov, F.~Papadopoulos, M.~Kitsak, A.~Vahdat, and M.~Bogun{\'a}.
\newblock Hyperbolic geometry of complex networks.
\newblock {\em Physical Review E}, 82(3):036106, 2010.

\bibitem{lengler2017existence}
J.~Lengler and L.~Todorovic.
\newblock Existence of small separators depends on geometry for geometric
  inhomogeneous random graphs.
\newblock arXiv:1711.03814, 2017.

\bibitem{Matt96}
T.~Mattfeldt.
\newblock Stochastic geometry and its applications.
\newblock {\em Journal of Microscopy}, 183(3):257--257, 1996.

\bibitem{MeeRoy96}
R.~Meester and R.~Roy.
\newblock {\em Continuum percolation}, volume 119 of {\em Cambridge Tracts in
  Mathematics}.
\newblock Cambridge University Press, Cambridge, 1996.

\bibitem{merler2015spatiotemporal}
S.~Merler, M.~Ajelli, L.~Fumanelli, M.~F. Gomes, A.~P. y~Piontti, L.~Rossi,
  D.~L. Chao, I.~M. Longini~Jr, M.~E. Halloran, and A.~Vespignani.
\newblock Spatiotemporal spread of the 2014 outbreak of ebola virus disease in
  liberia and the effectiveness of non-pharmaceutical interventions: a
  computational modelling analysis.
\newblock {\em The Lancet Infectious Diseases}, 15(2):204--211, 2015.

\bibitem{miritello2013time}
G.~Miritello, E.~Moro, R.~Lara, R.~Mart{\'\i}nez-L{\'o}pez, J.~Belchamber,
  S.~G. Roberts, and R.~I. Dunbar.
\newblock Time as a limited resource: Communication strategy in mobile phone
  networks.
\newblock {\em Social Networks}, 35(1):89--95, 2013.

\bibitem{newman2018networks}
M.~Newman.
\newblock {\em Networks}.
\newblock Oxford university press, 2018.

\bibitem{newman2011structure}
M.~Newman, A.-L. Barabasi, and D.~J. Watts.
\newblock {\em The structure and dynamics of networks}, volume~19.
\newblock Princeton University Press, 2011.

\bibitem{nordvik2006number}
M.~K. Nordvik and F.~Liljeros.
\newblock Number of sexual encounters involving intercourse and the
  transmission of sexually transmitted infections.
\newblock {\em Sexually transmitted diseases}, 33(6):342--349, 2006.

\bibitem{papadopoulos2010greedy}
F.~Papadopoulos, D.~Krioukov, M.~Bogu{\~n}{\'a}, and A.~Vahdat.
\newblock Greedy forwarding in dynamic scale-free networks embedded in
  hyperbolic metric spaces.
\newblock In {\em Proceedings of the International Conference on Computer
  Communications (INFOCOM 2010)}, pages 1--9. IEEE, 2010.

\bibitem{pastor2015epidemic}
R.~Pastor-Satorras, C.~Castellano, P.~Van~Mieghem, and A.~Vespignani.
\newblock Epidemic processes in complex networks.
\newblock {\em Reviews of modern physics}, 87(3):925, 2015.

\bibitem{pastor2007evolution}
R.~Pastor-Satorras and A.~Vespignani.
\newblock {\em Evolution and structure of the Internet: A statistical physics
  approach}.
\newblock Cambridge University Press, 2007.

\bibitem{SerKriBog08}
M.~A. Serrano, D.~Krioukov, and M.~Bogu\~n\'a.
\newblock Self-similarity of complex networks and hidden metric spaces.
\newblock {\em Physical Review Letters}, 100:078701, Feb 2008.

\bibitem{Trap10}
P.~Trapman.
\newblock The growth of the infinite long-range percolation cluster.
\newblock {\em Ann. Probab.}, 38(4):1583--1608, 07 2010.

\bibitem{HofKom18}
R.~van~der Hofstad and J.~Komjathy.
\newblock Explosion and distances in scale-free percolation.
\newblock arXiv:1706.02597, 2017.

\end{thebibliography}
\end{document}